\documentclass[11pt,amssymb]{amsart}
\usepackage{amssymb}
\usepackage[mathscr]{eucal}
\usepackage[all,cmtip]{xy}
\usepackage{amscd}

\newcommand{\im}{{\rm im}\:}

\newcommand{\Gu}{\underline{G}}

\newcommand{\Z}{{\mathbb Z}}

\newcommand{\Q}{{\mathbb Q}}

\newcommand{\R}{{\mathbb R}}

\newcommand{\Br}{\mathrm{Br}}

\newcommand{\Ga}{\mathrm{Gal}}
\newtheorem{thm}{Theorem}[section]
\newtheorem{lemma}[thm]{Lemma}
\newtheorem{prop}[thm]{Proposition}
\newtheorem{cor}[thm]{Corollary}

\theoremstyle{definition}

\newcommand{\gen}{\mathbf{gen}}

\newcommand{\he}{H_{\mathrm{\acute{e}t}}}

\newcommand{\uG}{\underline{G}}

.240pk scaled 1200 .240pk
\usepackage[all,cmtip]{xy}
\begin{document}

\title[Simple algebraic groups with good reduction]{Simple algebraic groups with the same maximal tori, weakly commensurable Zariski-dense subgroups, and good reduction}

\author[V.I.~Chernousov]{Vladimir I. Chernousov}

\author[A.S.~Rapinchuk]{Andrei S. Rapinchuk}

\author[I.A.~Rapinchuk]{Igor A. Rapinchuk}

\begin{abstract}
We provide a new condition for an absolutely almost simple algebraic group to have good reduction with respect to a discrete valuation of the base field which is formulated in terms of the existence  of  maximal tori with special properties. This characterization, in
particular, shows that the Finiteness Conjecture for forms of an absolutely almost simple algebraic group over a finitely generated
field that have good reduction at a divisorial set of places of the field (cf. \cite{RR-survey}) would imply the finiteness of the genus
of the group at hand. It also leads to a new phenomenon that we refer to as ``killing the genus by a purely transcendental extension."
Yet another application  deals with the investigation of ``eigenvalue rigidity" of Zariski-dense subgroups (cf. \cite{R-ICM}), which in turn is
related to the analysis of length-commensurable Riemann surfaces and general locally symmetric spaces. Finally, we analyze the Finiteness Conjecture and the genus problem for
simple algebraic groups of type $\textsf{F}_4$.
\end{abstract}

\address{Department of Mathematics, University of Alberta, Edmonton, Alberta T6G 2G1, Canada}

\email{vladimir@ualberta.ca}

\address{Department of Mathematics, University of Virginia,
Charlottesville, VA 22904-4137, USA}

\email{asr3x@virginia.edu}

\address{Department of Mathematics, Michigan State University, East Lansing, MI
48824, USA}

\email{rapinchu@msu.edu}

\maketitle

\hfill \parbox[t]{5cm}{\it To Gopal Prasad}

\vskip3mm

\section{Introduction}\label{S:Intro}

Let $G$ be a reductive affine algebraic group over a field $k$. Given a discrete valuation $v$ of $k$, we denote by $k_v$ the corresponding completion, with valuation ring $\mathcal{O}_v$ and residue field $k^{(v)}$. We recall that $G$ has \emph{good reduction} at $v$ if there exists a reductive group scheme $\mathscr{G}$ over $\mathcal{O}_v$  whose generic fiber $\mathscr{G}  \times_{\mathcal{O}_v} k_v$ is isomorphic to $G \times_k k_v$;  then the closed fiber $\mathscr{G} \times_{\mathcal{O}_v} k^{(v)}$ is called the {\it reduction} of $G$ at $v$ and will be denoted $\uG^{(v)}$ (see \S\ref{S:GR} for more details, including the uniqueness of reduction, and variations). The focus of the recent work \cite{CRR3}, \cite{CRR4}, \cite{CRR-Spinor}, \cite{RR-tori1}, \cite{RR-tori2} was on the analysis of $k$-forms of $G$ that have good reduction at all valuations in some natural set $V$ of discrete valuations of $k$. We refer the reader to the survey \cite{RR-survey} for a detailed discussion of this problem and some natural choices for $k$ and $V$. One is particularly interested in the case
where $k$ is a finitely generated field and $V$ is a {\it divisorial set} of valuations of $k$ (which means that $V$ consists of the discrete valuations that correspond to all prime divisors on a model $\mathfrak{X}$ of $k$, i.e. an irreducible separated normal scheme of finite type over $\mathbb{Z}$ with function field $k$ --- see \cite[5.3]{RR-survey}). In this case, there is the following  \underline{Finiteness Conjecture} (cf. \cite[Conjecture 5.7]{RR-survey}): {\it the set of $k$-isomorphism classes of $k$-forms of $G$ that have good reduction at all $v \in V$ is finite} (at least when the characteristic of $k$ is ``good"). This conjecture has been established in a number of cases, but the general case remains the focus of ongoing work. Its significance for the current effort to develop the arithmetic theory of algebraic groups over higher-dimensional fields is predicated on deep connections with other important problems. In particular, the validity of the conjecture for an absolutely almost simple simply connected $k$-group $G$ and any divisorial set of places of $k$ would imply the properness of the global-to-local map $H^1(k , \overline{G}) \to \prod_{v \in V} H^1(k_v , \overline{G})$ in Galois cohomology for the corresponding adjoint group $\overline{G}$ (cf. \cite[\S 6]{RR-survey}). In the present paper, we will focus on several other applications of the Finiteness Conjecture, particularly those related to the {\it genus problem} for absolutely almost simple algebraic groups. This includes a new phenomenon that we have termed ``killing the genus by a purely transcendental extension," and the investigation of ``eigenvalue  rigidity" of Zariski-dense subgroups (cf. \cite{R-ICM}) --- the latter is related to the analysis of length-commensurable Riemann surfaces and general locally symmetric spaces in differential geometry (cf. \cite{PR-WC}, \cite{PR-Appl}). Finally, we develop new techniques for tackling the genus problem for some groups of type $\textsf{F}_4$ and obtain several finiteness results in this case.

To prepare for the discussion of  the genus problem, we recall that two semisimple algebraic groups $G_1$ and $G_2$ defined over a field $k$ are said to have the {\it same isomorphism classes of maximal $k$-tori} if every maximal $k$-torus $T_1$ of $G_1$ is $k$-isomorphic to some maximal $k$-torus $T_2$ of $G_2$, and vice versa. We then define the ($k$-)genus $\gen_k(G)$ (resp., the extended ($k$-)genus $\gen^+_k(G)$) of an absolutely almost simple $k$-group $G$ as the set of $k$-isomorphism classes of {\it inner} $k$-forms of $G$ (resp., {\it all} $k$-forms of $G$) that have the same isomorphism classes of maximal $k$-tori as $G$. (We note that we always have the inclusion $\gen_k(G) \subset \gen^+_k(G)$, which, in fact, is an equality whenever $k$ is finitely generated --- see Corollary \ref{C:+}.)
The analysis of the genus is the subject of the {\it genus problem}. In particular, one expects to prove that the genus is always finite whenever the field $k$ is finitely generated (of ``good" characteristic) and is trivial in some special situations (see \cite[\S 8]{RR-survey}). One of our main results is the following theorem that relates the genus problem to good reduction.

\begin{thm}\label{T:GoodReduction1}
Let $G$ be an absolutely almost simple linear algebraic group over a field $k$ and let $v$ be a discrete valuation of $k$. Assume that the residue field $k^{(v)}$ is finitely generated and that $\mathrm{char}\: k^{(v)} \neq 2$ if $G$ is of type $\textsf{B}_{\ell}$ $(\ell \geq 2)$.
If $G$ has good reduction at $v$, then any $G' \in \gen_k(G)$ also has good reduction at $v$. Moreover, the
reduction ${\uG'}^{(v)}$ lies in the genus $\gen_{k^{(v)}}(\uG^{(v)})$ of the reduction $\uG^{(v)}$.
\end{thm}

It should be pointed out that the proof of this theorem is based on an entirely new approach to good reduction of simple algebraic groups that shows
that the existence of good reduction can be characterized in terms of the presence of maximal tori with certain specific properties --- see Theorems \ref{T:GR-not} and \ref{T:AB} for precise statements. This approach enables us to extend to absolutely almost simple groups the techniques developed
earlier in \cite{CRR-Bull}, \cite{CRR3}, \cite{CRR-Isr}, and \cite{RR-manuscr} to analyze the genus of a division algebra. In particular, just like the finiteness
of the $n$-torsion of the unramified Brauer group ${}_n\mathrm{Br}(k)_V$ of a finitely generated field $k$ with respect to a divisorial set of places $V$
(provided that $n$ is prime to $\mathrm{char}\: k$) implies the finiteness of the genus of any central simple algebra $D$ of degree $n$ over $k$ (cf. \cite{CRR-Bull}, \cite{CRR3}), the above Finiteness Conjecture, in view of the following corollary of Theorem \ref{T:GoodReduction1}, would imply the finiteness of the genus of any absolutely almost simple algebraic $k$-group.

\begin{cor}\label{C:genus-GR}
Let $G$ be an absolutely almost simple algebraic group over an infinite finitely generated field $k$, and let $V$ be a divisorial set of places of $k$. Assume that $\mathrm{char}\: k \neq 2$ if $G$ is of type $\textsf{B}_{\ell}$ $(\ell \geq 2)$. Then there exists a finite subset $S \subset V$ such that every $G' \in \gen_k(G)$ has good reduction at all $v \in V \setminus S$.
\end{cor}

Next, applying Theorem \ref{T:GoodReduction1}, in conjunction with the theorem of Raghunathan-Ramanathan \cite{RagRam}, we obtain the following statement
concerning the effect of a purely transcendental base change on the genus.
\begin{thm}\label{T:transc}
Let $G$ be an absolutely almost simple algebraic group over a finitely generated field $k$ of characteristic $\neq 2$, and let $K = k(x)$ be the field of rational functions. Then any $H \in \gen_K(G \times_k K)$ is of the form $H = H_0 \times_k K$ for some $H_0 \in \gen_k(G)$.
\end{thm}
In view of \cite[Theorem 7.5]{PR-WC}, the following is an immediate consequence of Theorem \ref{T:transc}.
\begin{cor}\label{C:transc}
Let $G$ be an absolutely almost simple simply connected algebraic group over a number field $k$, and let $K = k(x_1, \ldots , x_m)$ be the field of rational functions in $m \geqslant 1$ variables. Then the genus $\gen_K(G \times_k K)$ is finite, and in fact reduces to a single element if the type of $G$ is different from $\textsf{A}_{\ell}$ $(\ell > 1)$, $\textsf{D}_{2\ell + 1}$ $(\ell > 1)$, and $\textsf{E}_6$.
\end{cor}

Theorem \ref{T:transc} prompts the question of whether for an absolutely almost simple algebraic group $G$ over a field $k$ and any $G' \in \gen_k(G)$, the
group $G' \times_k K$ obtained by base change to the field of rational functions $K = k(x)$ lies in $\gen_K(G \times_k K)$. It turns out that not only is the answer to this question negative, but in fact one should expect an opposite phenomenon that we have termed ``killing the genus by a purely transcendental extension." The nature of this phenomenon reveals itself in the following two statements.
\begin{thm}\label{T:Kill1}
Let $A$ be a central simple algebra of degree $n$ over a finitely generated field $k$, and let $G = \mathrm{SL}_{1 , A}$. Assume that $\mathrm{char}\: k$ is prime to $n$, and let $K = k(x_1, \ldots , x_{n-1})$ be the field of rational functions in $(n-1)$ variables. Then $\gen_K(G \times_k K)$ consists of (the isomorphism classes of) groups of the form $H \times_k K$, where $H = \mathrm{SL}_{1 , B}$ and $B$ is a central simple algebra of degree $n$ such that its class $[B]$ in the Brauer group $\Br(k)$ generates the same subgroup as the class $[A]$.
\end{thm}

The proof uses Amitsur's theorem on generic splitting fields \cite{Amit}, and a result of D.~Saltman \cite{Salt1}, \cite{Salt2} on
function fields of Severi-Brauer varieties.
\begin{thm}\label{T:Kill2}
Let $G$ be a group of type $\textsf{G}_2$ over a finitely generated field $k$ of characteristic $\neq 2, 3$, and let $K = k(x_1, \ldots , x_6)$ be the field of rational functions in 6 variables. Then $\gen_K(G \times_k K)$ reduces to a single element.
\end{thm}

The proof relies on properties of Pfister forms (cf. \cite{Lam}). These results prompt the following conjecture.

\vskip2mm

\addtocounter{thm}{1}

\noindent {\bf Conjecture 1.7.} {\it Let $G$ be an absolutely almost simple group over a finitely generated field $k$. Assume that $\mathrm{char}\: k$ is prime to the order of the Weyl group of $G$. Then there exists a purely transcendental extension $K = k(x_1, \ldots , x_m)$ of transcendence degree $m$ depending only on the Cartan-Killing type of $G$ such that every $H \in \gen_K(G \times_k K)$ is of the form $H_0 \times_k K$, where $H_0$ has the property that $H_0 \times_k F \in \gen(G \times_k F)$ for \emph{any} field extension $F/k$.}

\vskip2mm

In \S \ref{S:killing}.4, we relate this conjecture to the notion of the {\it motivic genus}  that was proposed by A.S.~Merkurjev.

Next, we will discuss applications of our results to the analysis of weakly commensurable Zariski-dense subgroups, which was initiated in
\cite{PR-WC} in connection with some problems in differential geometry. So, let $G_1$ and $G_2$ be absolutely almost simple algebraic groups over a field $F$ of characteristic zero, and let $\Gamma_1 \subset G_1(F)$ and $\Gamma_2 \subset G_2(F)$ be two finitely generated Zariski-dense subgroups. We refer the reader to \cite[\S 1]{PR-WC} (see also \S \ref{S:generic} of the present paper) for the technical definition of the relation of {\it weak commensurability}; here, we only mention that it is a way of matching the eigenvalues of semisimple elements of $\Gamma_1$ and $\Gamma_2$. This relation is expected to lead to a new form of rigidity, called ``eigenvalue rigidity," for arbitrary finitely generated Zariski-dense subgroups, where traditional forms of rigidity are inapplicable (cf. \cite{R-ICM}). In this paper, we will show that one of the key issues in eigenvalue rigidity can be reduced to the Finiteness Conjecture. To provide more context, we recall that given a Zariski-dense subgroup $\Gamma \subset G(F)$, where $G$ is an absolutely almost simple algebraic group defined over a field $F$, the {\it trace field} $k_{\Gamma}$ is defined to be the subfield of $F$ generated by the traces $\mathrm{tr}(\mathrm{Ad}\: \gamma)$ of elements $\gamma \in \Gamma$ in the adjoint representation on the Lie algebra $\mathfrak{g}$. According to a theorem of E.B.~Vinberg \cite{Vinberg}, the field $k = k_{\Gamma}$ is the {\it minimal field of definition of} $\Gamma$. This means that $k$ is the minimal subfield of $F$ with the property that all transformations in $\mathrm{Ad}\: \Gamma$ can be simultaneously represented by matrices having all entries in $k$ in a suitable basis of $\mathfrak{g}$. If such a basis is chosen, then the Zariski-closure of $\mathrm{Ad}\: \Gamma$ in $\mathrm{GL}(\mathfrak{g})$ is a simple algebraic $k$-group $\mathscr{G}$. It is an $F/k$-form of the adjoint group $\overline{G}$  called the {\it algebraic hull} of $\mathrm{Ad}\: \Gamma$. It should be mentioned that if $\Gamma$ is {\it arithmetic}, the pair $(k , \mathscr{G})$ determines the commensurability class of $\Gamma$. While for general Zariski-dense subgroups this is no longer the case, the pair $(k , \mathscr{G})$ remains an important invariant of the commensurability class.

Now let $\Gamma_1 \subset G_1(F)$ and $\Gamma_2 \subset G_2(F)$ be finitely generated Zariski-dense subgroups of absolutely almost simple algebraic groups $G_1$ and $G_2$. Assume that $\Gamma_1$ and $\Gamma_2$ are weakly commensurable. Then $k_{\Gamma_1} = k_{\Gamma_2} =: k$ (cf. \cite[Theorem 2]{PR-WC}). Furthermore, $G_1$ and $G_2$ either have the same Cartan-Killing type, or one of them has type $\textsf{B}_{\ell}$ and the other type $\textsf{C}_{\ell}$ for some $\ell \geq 3$. So, apart from the ambiguity between types $\textsf{B}$ and $\textsf{C}$, the corresponding algebraic hulls $\mathscr{G}_1$ and $\mathscr{G}_2$ are $k$-forms of one another. The remaining critical issue is the relationship between $\mathscr{G}_1$ and $\mathscr{G}_2$. More precisely, if we fix $\Gamma_1$, what can one say about the set of the forms $\mathscr{G}_2$ as $\Gamma_2 \subset G_2(F)$ runs through finitely generated Zariski-dense subgroups that are weakly commensurable to $\Gamma_1$? There is a conjecture (cf. \cite[Conjecture 6.1]{R-ICM}) that this set consists of finitely many $k$-isomorphism classes --- see \S\ref{S:lattice} for the precise formulation. If true, this would be a very strong statement\footnote{Which, in particular, would be stronger than the finiteness of the genus.} asserting that the eigenvalues of elements of a Zariski-dense subgroup (which could be, for example, just a free group on two generators) determine the ambient algebraic group up to finitely many possibilities. For example, if $G = \mathrm{SL}_{1 , A}$, where $A$ is a central simple algebra of degree $n$ over a field $k$, and $\Gamma \subset G(k)$ is a finitely generated Zariski-dense subgroup with trace field $k$, then there would be only finitely many choices for a central simple $k$-algebra $A'$ (necessarily of the same degree $n$) such that for $G' = \mathrm{SL}_{1 , A'}$, the group $G'(k)$ contains a finitely generated Zariski-dense subgroup weakly commensurable to $\Gamma$. What we will see in \S \ref{S:WC}  is that this conjecture again can be derived from the Finiteness Conjecture with the help of the following result (and in fact, the above statement for groups of type $\mathrm{SL}_{1 , A}$ is already a theorem due to the fact that the Finiteness Conjecture has been confirmed in this case).
\begin{thm}\label{T:WC-GR}
Let $G$ be an absolutely almost simple algebraic group over a finitely generated field $k$ of characteristic zero, and let $V$ be a divisorial set of places
of $k$. Given a finitely generated Zariski-dense subgroup $\Gamma \subset G(k)$ with trace field $k$, there exists a finite subset $S(\Gamma) \subset V$ such that every absolutely almost simple algebraic $k$-group $G'$ with the property that there exists a finitely generated Zariski-dense subgroup $\Gamma' \subset
G'(k)$ that is weakly commensurable to $\Gamma$ has good reduction at all $v \in V \setminus S(\Gamma)$.
\end{thm}

The results on weakly commensurable arithmetic groups developed in \cite{PR-WC} were used to settle some long-standing problems about isospectral and length-commensurable locally symmetric spaces. Here we will give only one application of the results on good reduction to not necessarily arithmetically defined Riemann surfaces. For a Riemannian manifold $M$, we denote by $L(M)$ the {\it (weak) length spectrum} of $M$, i.e. the collection of the lengths of all closed geodesics in $M$. We then call two Riemannian manifolds $M_1$ and $M_2$ {\it length-commensurable} if $\Q \cdot L(M_1) = \Q \cdot L(M_2)$. Consider a Riemann surface $M$ of the form $\mathbb{H}/\Gamma$, where $\mathbb{H}$ is the complex upper half-plane and $\Gamma \subset \mathrm{SL}_2(\R)$ is a discrete subgroup with torsion-free image in $\mathrm{PSL}_2(\R)$. We will assume that $\Gamma$ is finitely generated and Zariski-dense in $\mathrm{SL}_2$ (which is automatically true if $M$ is, for example, compact). Then one can naturally associate to $\Gamma$ a quaternion algebra $A_{\Gamma}$ whose center is the trace field of $\Gamma$ --- see \cite[3.2]{MacReid} and \S\ref{S:WC}. If $\Gamma$ is arithmetic, then $A_{\Gamma}$ is {\it the} quaternion algebra required for its description, and in the general case it is an invariant of the commensurability class of $\Gamma$. In \S\ref{S:WC}, we will  prove the following result that contains no arithmeticity assumptions.
\begin{thm}\label{T:RiemSurf}
Let $M_i = \mathbb{H}/\Gamma_i$ $(i \in I)$ be a family of length-commensurable Riemann surfaces, where $\Gamma_i \subset \mathrm{SL}_2(\R)$ is a discrete finitely generated Zariski-dense subgroup with torsion-free image in $\mathrm{PSL}_2(\R)$. Then the quaternion algebras $A_{\Gamma_i}$ $(i \in I)$ belong to finitely many isomorphism classes over the common center (= trace field of all the $\Gamma_i$'s).
\end{thm}

To the best of our knowledge, this statement is one of the first examples of applications of techniques from arithmetic geometry to  nonarithmetic Riemann surfaces.

\vskip1mm

We conclude the paper with a series of results on forms with good reduction and the genus of simple algebraic groups of type $\textsf{F}_4$, which have never been previously analyzed over fields more general than number fields. The first three results treat those forms that split over a quadratic extension of the base field (see Appendix 2 for a characterization of such forms in terms of cohomological invariants). We recall that the $\mathbb{Q}$-forms of type $\textsf{F}_4$ that have good reduction at all primes were described in \cite{Gross} and \cite{ConradZ}, and that for any simple group $G$ of that type over a number field $k$, the genus $\gen_k(G)$ is trivial \cite[Theorem 7.5]{PR-WC}. We will prove the following version of the ``Stability Theorem" that was established previously for groups of the form $\mathrm{SL}_{1,A}$, where $A$ is a central simple algebra of exponent 2 (cf. \cite{CRR-Bull}), and groups of type $\textsf{G}_2$ (cf. \cite{CRR-Spinor}).
\begin{thm}\label{T:F4Stab}
Let $k_0$ be a number field, and set $k = k_0(x)$. Then for any absolutely simple algebraic $k$-group $G$ of type $\textsf{F}_4$ that splits over a quadratic extension of $k$, the genus $\gen_k(G)$ is trivial.
\end{thm}

Next, following Kato \cite{Kato}, we recall that a 2-dimensional global field is defined to be the function field of either a curve over a number field or a surface over a finite field.
\begin{thm}\label{T:F4genus}
Let $k$ be either a 2-dimensional global field of characteristic $\neq 2, 3$ or a purely transcendental extension $k = k_0(x , y)$ of transcendence degree 2 of a number field $k_0$. Then for any absolutely simple $k$-group $G$ of type $\textsf{F}_4$ that splits over a quadratic extension of $k$, the genus $\gen_k(G)$ is finite.
\end{thm}

This is derived by combining  Corollary \ref{C:genus-GR} with the following theorem.

\begin{thm}\label{T:F4GR}
Let $k$ be either a 2-dimensional global field of characteristic $\neq 2, 3$ or a purely transcendental extension $k = k_0(x , y)$ of transcendence degree 2 of a number field $k_0$, and let $V$ be a divisorial set of discrete valuations of $k$. Then the set $\mathscr{I}$ of $k$-isomorphism classes of $k$-forms of type $\textsf{F}_4$ that split over a quadratic extension of $k$ and have good reduction at all $v \in V$ is finite.
\end{thm}

We note that similar results for groups of type $\textsf{G}_2$ were obtained in \cite{CRR-Spinor} over 2-dimensional global field and in \cite{RR-tori1} over purely transcendental extensions of transcendence degree 2 of number fields.

Our final result applies to {\it all} forms of type $\textsf{F}_4$ and contributes to one of the main remaining problems in the theory of Jordan algebras. We refer the reader to subsection A2.1 of Appendix 2 for the definition of the map $\phi$ that describes forms of type $\textsf{F}_4$ in terms of the cohomological invariants $f_3$, $f_5$ and $g_3$. J.-P.~Serre has raised the problem of whether $\phi$ is injective. We will show that, assuming the Finiteness Conjecture, we can at least confirm that $\phi$ is proper.
\begin{thm}\label{T:F4phi}
Let $k$ be a finitely generated field of characteristic $\neq 2, 3$. Assume that the Finiteness Conjecture holds for $k$-groups of type
$\textsf{F}_4$ with respect to any divisorial set $V$ of discrete valuations of $k$. Then the map $\phi$ is proper, i.e. the preimage of a finite set is finite.
\end{thm}

\vskip3mm

\noindent {\bf Notations and conventions.} We use standard notations associated with the Galois cohomology of algebraic groups (cf. \cite{Serre-GC}). In particular, given an algebraic group $G$ defined over a field $k$ and a Galois extension $\ell/k$, we denote by $H^1(\ell/k , G)$ the set $H^1(\mathrm{Gal}(\ell/k) , G(\ell))$ of noncommutative continuous Galois cohomology, and we write $H^1(k , G)$ for $H^1(k^{\mathrm{sep}}/k , G(k^{\mathrm{sep}})),$ where $k^{\mathrm{sep}}$ is a separable closure of $k$. Similar conventions are used for the set $Z^1(\ell/k , G)$ of noncommutative continuous 1-cocycles. We will slightly abuse notation and use
lowercase Greek letters to denote both cocycles and cohomology classes whenever this does not lead to confusion. However, when we need to distinguish between the two, we will write $\xi$ for a cocycle and $[\xi]$ for the corresponding cohomology class. We extend these notations also to \'{e}tale (\v{C}ech) cocycles and the cohomology classes they define.

For an algebraic torus $T$, we let $X(T)$ and $X_*(T)$ denote the corresponding groups of characters and cocharacters, respectively. Furthermore, we denote by $\mathbb{G}_m$ the one-dimensional split torus.

Next, given a field $k$ equipped with a discrete valuation $v$, we denote by $k_v$ and $k^{(v)}$ the corresponding completion and residue field, respectively. Furthermore, we set $\mathcal{O}_v \subset k_v$ and $\mathcal{O}_{k,v} \subset k$ to be the associated valuation rings.

Finally, we recall some definitions and notations pertaining to commutative Galois cohomology and unramified cohomology, which will be needed mainly in \S\ref{S:F4} and in Appendix 2. For a $\mathrm{Gal}(k^{\mathrm{sep}}/k)$-module $M$, we write $H^i(k, M)$ for the Galois cohomology group $H^i(\mathrm{Gal}(k^{\mathrm{sep}}/k), M).$ Now, if $\mathrm{char}~k^{(v)}$ is prime to $n$, then there exists a residue map
$$
\rho^i_v \colon H^i(k, \mu_n^{\otimes j}) \to H^{i-1}(k^{(v)}, \mu_n^{\otimes (j-1)}).
$$
We say that a class $x \in H^i(k, \mu_n^{\otimes j})$ is {\it unramified at $v$} if $x \in \ker \rho^i_v$ and that it is {\it ramified} otherwise. Furthermore, if $V$ is a set of discrete valuations of $k$ such that $\mathrm{char}~k^{(v)}$ is prime to $n$ for all $v \in V$, then one defines the corresponding {\it unramified cohomology of degree $i$} to be
$$
H^i(k, \mu_n^{\otimes j})_V = \bigcap_{v \in V} \ker \rho^i_v.
$$
We refer the reader to \cite[Ch. III and IV]{GMS} for further details on these constructions.


\vskip5mm

\section{Groups with good reduction}\label{S:GR}

\vskip3mm

\noindent {\bf 2.1. Good reduction: definition and examples.} Even though the definition of good reduction for a reductive algebraic group at a discrete valuation of the base field has already been mentioned in \S 1, we begin by repeating it here for the convenience of further references.

\vskip1mm

\noindent {\bf Definition 2.1.} Let $G$ be a reductive algebraic group over a field $k$, and let $v$ be a discrete valuation of $k$. 
We say that $G$ has {\it good reduction} at $v$ if there exists a reductive group scheme\footnotemark \ $\mathscr{G}$ over the valuation ring $\mathcal{O}_v \subset k_v$ with generic fiber
$$
\mathscr{G} \times_{\mathcal{O}_v} k_v \simeq G \times_k k_v.
$$
Then the $k^{(v)}$-group scheme $\mathscr{G} \times_{\mathcal{O}_v} k^{(v)}$ is called the {\it reduction} of $G$ at $v$ and will be denoted $\underline{G}^{(v)}$.

\footnotetext{Let $A$ be a commutative ring and $S = \mathrm{Spec}\: A$. A reductive $A$-group scheme is a smooth affine group scheme $\mathscr{G} \to S$ such that the geometric fibers $\mathscr{G}_{\bar{s}}$ are connected reductive groups for all geometric points $\bar{s}$ of $S$, cf. \cite[Exp. XIX, Definition 2.7]{SGA3} or \cite[Definition 3.1.1]{ConradZ}.}

\vskip.5mm

(We will see in subsection 2.2 below that the reduction $\underline{G}^{(v)}$ is well-defined.) We now consider a couple of examples of good reduction that are relevant for the present paper.

\vskip1mm

\noindent {\bf Example 2.2.} Let $G = \mathrm{SL}_{1 , A}$, where $A$ is a central simple algebra of degree $n$ over $k$, and let $v$ be a discrete valuation of $k$. Assume that $A$ is unramified at $v$, which means that there exists an Azumaya $\mathcal{O}_v$-algebra $\mathscr{A}$ such that $A \otimes_k k_v \simeq \mathscr{A} \otimes_{\mathcal{O}_v} k_v$ as $k_v$-algebras (cf. \cite{IR1} and references therein). Let $\mathscr{G} = \mathrm{SL}_{1 , \mathscr{A}}$ be the semisimple group scheme over $\mathcal{O}_v$ associated with $\mathscr{A}$ (cf. \cite[3.5.0.91]{CalFas}). Then
\begin{equation}\label{E:GRed1}
G \times_k k_v \simeq \mathscr{G} \times_{\mathcal{O}_v} k_v
\end{equation}
as $k_v$-groups, hence $G$ has good reduction at $v$.

Conversely, suppose $G = \mathrm{SL}_{1 , A}$ has good reduction at $v$, and let $\mathscr{G}$ be the corresponding reductive scheme over $\mathcal{O}_v$. It is known that any inner form of the $\mathcal{O}_v$-group scheme $\mathrm{SL}_n$ is of the form $\mathrm{SL}_{1 , \mathscr{A}}$ for some Azumaya $\mathcal{O}_v$-algebra $\mathscr{A}$ of degree $n$ (cf. \cite[3.5.0.92]{CalFas}).  So, if we  write  $\mathscr{G}$ as $\mathrm{SL}_{1 , \mathscr{A}}$,  the isomorphism (\ref{E:GRed1}) implies that
$$
\text{either} \ \ A \otimes_k k_v \simeq \mathscr{A} \otimes_{\mathcal{O}_v} k_v \ \ \text{or} \ \ A \otimes_k k_v \simeq \mathscr{A}^{\mathrm{op}}  \otimes_{\mathcal{O}_v} k_v.
$$
In either case, $A \otimes_k k_v$ comes from an Azumaya $\mathcal{O}_v$-algebra, and therefore $A$ is unramified. Thus, $G = \mathrm{SL}_{1 , A}$ has good reduction at $v$ if and only if $A$ is unramified at $v$.

\vskip1mm

\noindent {\bf Example 2.3.} Let $G = \mathrm{Spin}_n(q)$,  where $q$ is a nondegenerate quadratic form of dimension $n \geq 2$ over a field $k$ of characteristic $\neq 2$, and let $v$ be a discrete valuation of $k$ with residue characteristic $\mathrm{char}\: k^{(v)} \neq 2$. We will show that $G$ has good reduction at $v$ if and only if $q$ is equivalent over $k_v$ to a quadratic form of the shape
\begin{equation}\label{E:qform}
\lambda (u_1 x_1^2 + \cdots + u_n x_n^2), \ \ \text{with} \ \ \lambda \in k_v^{\times} \ \ \text{and} \ \ u_1, \ldots , u_n \in \mathcal{O}_v^{\times}.
\end{equation}
First, let us assume that $q$ is $k_v$-equivalent to such a form and set $q_0 = u_1 x_1^2 + \cdots + u_n x_n^2$. Then $G \times_k k_v = \mathrm{Spin}_n(q) = \mathrm{Spin}_n(q_0)$. On the other hand, since $q_0$ is a {\it regular} quadratic form on $\mathcal{O}_v^n$, there is a semisimple group scheme $\mathscr{G} = \mathscr{SPIN}_n(q_0)$ over $\mathcal{O}_v$ with generic fiber $G \times_k k_v$ (cf. \cite[4.5.2.6, 6.2.0.28, 8.2.0.59]{CalFas}).  This means that $G$ has good reduction at $v$.

Conversely, suppose $G = \mathrm{Spin}_n(q)$ has good reduction at $v$. When $n = 2$, the group $G$ is a 1-dimensional torus whose splitting field is unramified at $v$, implying that $q$ is  equivalent to a form as in (\ref{E:qform}). Now suppose $n > 2$. Let $q_0$ be a an $n$-dimensional split quadratic form, and let $\mathscr{G}_0 = \mathscr{SPIN}_n(q_0)$. Assume that $G = \mathrm{Spin}_n(q)$ has good reduction at $v$, i.e. there exists a reductive group $\mathcal{O}_v$-scheme $\mathscr{G}$ with generic fiber $G \times_k k_v$. Then $\mathscr{G}$ is obtained from $\mathscr{G}_0$ by twisting using an \'{e}tale 1-cocycle $\xi \in Z^1(\mathcal{O}_v , \mathrm{Aut}(\mathscr{G}_0))$. If $n$ is odd, then $\mathrm{Aut}(\mathscr{G}_0) = \mathscr{SO}_n(q_0)$. Then $\xi$ can be used to twist the quadratic form $q_0$ and obtain thereby a regular quadratic form $q'$ over $\mathcal{O}_v$, in which case $\mathscr{G} = \mathscr{SPIN}_n(q')$. Passing to the generic fiber, we obtain that $\mathrm{Spin}_n(q) \simeq \mathrm{Spin}_n(q')$ over $k_v$ and therefore $q$ is $k_v$-equivalent to a scalar multiple of $q'$. On the other hand, since $\mathrm{char}\: k^{(v)} \neq 2$, the form $q'$, being regular, can be diagonalized over $\mathcal{O}_v$ as $u_1 x_1^2 + \cdots + u_n x_n^2$ with $u_i \in \mathcal{O}_v^{\times}$, proving our claim.

The same argument works when $n$ is even provided we can show that in this case, the cohomology class $[\xi]$ lies in the image of the map $\lambda \colon H^1(\mathcal{O}_v , \mathscr{O}_n(q_0)) \to H^1(\mathcal{O}_v , \mathrm{Aut}(\mathscr{G}_0))$ coming from the canonical morphism $\nu \colon \mathscr{O}_n(q_0) \to \mathrm{Aut}(\mathscr{G}_0)$. First, we observe that when $n = 8$, the group scheme $\mathscr{G}$ cannot be a triality form as otherwise the generic fiber $G$ would also be a triality form, which is not the case. This means that in all cases, $[\xi]$ is represented by a cocycle having values in $B = {\rm Im}~\nu$ (we note that $B$ is represented by $\mathscr{PSO}_n(q_0) \rtimes \mathbb{Z}/2\mathbb{Z}$). The exact sequence
$$
1 \to \mu_2 \longrightarrow \mathscr{O}_n(q_0) \stackrel{\nu}{\longrightarrow} B \to 1
$$
gives rise to the exact sequence
\begin{equation}\label{E:ES1}
H^1(\mathcal{O}_v , \mathscr{O}_n(q_0)) \stackrel{\lambda}{\longrightarrow} H^1(\mathcal{O}_v , B) \stackrel{\theta}{\longrightarrow} H^2(\mathcal{O}_v , \mu_2) = {}_2\mathrm{Br}(\mathcal{O}_v).
\end{equation}
We note that $\theta([\xi])$ is precisely the class of the Azumaya algebra involved in the description of $\mathscr{G}$. Since the generic fiber of $\mathscr{G}$ is the spinor group of a quadratic form, the image of $\theta([\xi])$ under the map $\mathrm{Br}(\mathcal{O}_v) \to \mathrm{Br}(k_v)$ is trivial, and then $\theta([\xi])$ is itself trivial since the latter map is well-known to be injective (cf. \cite[Ch. IV, Corollary 2.6]{Milne-EC}). The exact sequence (\ref{E:ES1}) then yields that $[\xi]$ lies in the image of $\lambda$, as required. \hfill $\Box$

\vskip1mm

\noindent {\bf 2.2. The Grothendieck-Serre conjecture and its consequences.}  The Grothendieck-Serre conjecture predicts that for a reductive group scheme $\mathscr{G}$ over a regular local ring $A$ with fraction field $k$, the map of nonabelian \'etale cohomology sets
$$
H^1(A , \mathscr{G}) \to H^1(k , G) \ \ \text{(where \ } G = \mathscr{G} \times_A k \text{)}
$$
has trivial kernel. Very significant progress on the conjecture was achieved in \cite{FPan}, where it was proved under the assumption that $R$ contains
an infinite field; the case where $A$ contains a finite field was treated in \cite{Pan}. In the present paper, however, we will only need the case where $A$ is a discrete valuation ring, which goes back to work of Y.~Nisnevich \cite{Nisn1}, \cite{Nisn2} (in fact, we will only need the case of a {\it complete} discrete valuation ring).

\addtocounter{thm}{3}

\begin{thm}\label{T:Nisn}
Let $\mathscr{G}$ be a reductive group scheme over a  discrete valuation ring $A$. Then the map of nonabelian \'etale cohomology sets
$$
H^1(A , \mathscr{G}) \to H^1(k , G) \ \ \text{(where \ } G = \mathscr{G} \times_A k \text{)}
$$
is injective.
\end{thm}

This is \cite[Ch. 2, Theorem 7.1]{Nisn1} and \cite[Theorem 4.2]{Nisn2} in the case where the residue field of $A$ is {\it perfect}. The general case is treated in \cite[Theorem 1]{Guo}. This result has the following important consequence.

\begin{prop}\label{P:UnMod}
{\rm \cite[\S 6, Corollary 3]{Guo}}
Let $A$ be a discrete valuation ring, and $k$ be its field of fractions. Then any reductive $k$-group has at most one reductive model over $A$.
\end{prop}

For semisimple groups and perfect residue fields, this statement appears as Theorem 5.1 in \cite{Nisn2}; the general case is treated in \cite[\S 6]{Guo}, where the argument combines  Theorem \ref{T:Nisn} with a general statement concerning the uniqueness of reductive models over regular semilocal rings --- see
\cite[\S 6, Proposition 14]{Guo} for the details. In the context of Definition 2.1, Proposition \ref{P:UnMod} yields the uniqueness up to isomorphism of the $\mathcal{O}_v$-scheme of $\mathscr{G}$, implying, in particular, that the reduction $\underline{G}^{(v)}$ is well-defined.

\vskip1mm

\noindent {\bf 2.3. A different approach to good reduction.} The above definition of good reduction is most convenient for our purposes, in particular,
for investigating connections with local-global principles. We note, however, that it is more traditional to define good reduction without passing to completions, i.e. by requiring the existence of a reductive group scheme $\mathscr{G}$ over the valuation ring $\mathcal{O}_{k,v} \subset k$ with generic fiber $\mathscr{G} \times_{\mathcal{O}_{k,v}} k$ isomorphic to $G$. Of course, our definition is less restrictive, so, a priori, we are  considering a more general situation. For the sake of completeness, however, we will now briefly explain that for tori and absolutely almost simple  groups, the two definitions are equivalent.

If a $k$-torus $T$ has good reduction at $v$ in the sense of Definition 2.1, then $v$ is unramified in the minimal splitting field $k_T$ (cf. \cite{NX}). Let $\widetilde{\mathcal{O}}$ denote the integral closure of $\mathcal{O}_{k,v}$ in $k_T$. Then $\widetilde{\mathcal{O}}/\mathcal{O}_{k,v}$ is a Galois extension of rings. Let $d = \dim T$, and let $\xi \in Z^1(k_T/k , \mathrm{GL}_d(\mathbb{Z}))$ be a cocycle such that the corresponding twist ${}_{\xi}(\mathbb{G}_m^d)$ of the $d$-dimensional $k$-split torus is $k$-isomorphic to $T$ (here we identify the automorphism group $\mathrm{Aut}(\mathbb{G}_m^d)$ with $\mathrm{GL}_d(\mathbb{Z})$ through the action on the character group $X(\mathbb{G}_m^d) = \mathbb{Z}^d$).  Then the Hopf $k$-algebra $k[T]$ is obtained by Galois descent from $k_T[T] = k_T[\mathbb{Z}^d]$ for the action of $\mathrm{Gal}(k_T/k)$ that coincides with the standard action on $k_T$ and is given by $\xi$ on $\mathbb{Z}^d$. Since $\widetilde{\mathcal{O}}/\mathcal{O}_{k,v}$ is a Galois extension of rings, we can likewise carry out Galois descent on $\widetilde{\mathcal{O}}[\mathbb{Z}^d]$ for the same action. This generates a Hopf $\mathcal{O}_{k,v}$-algebra that yields a torus $\mathscr{T}$ over $\mathcal{O}_{k,v}$ with generic fiber $T$, verifying thereby the traditional definition.

Next, let $G$ be an absolutely almost simple simply connected algebraic $k$-group that has good reduction at $v$ in the sense of Definition 2.1, and let $\ell$ be the minimal Galois extension of $k$ over which $G$ becomes an inner form of the split group $G_0$. The fact that $G$ has good reduction at $v$ implies that the extension $\ell/k$ is unramified at $v$. Let $\widetilde{\mathcal{O}}$ denote the integral closure of $\mathcal{O}_{k,v}$ in $\ell$; then $\widetilde{\mathcal{O}}/\mathcal{O}_{k,v}$ is a Galois extension of rings. Let $G_1$ be a quasi-split inner $k$-form of $G$. Then $G_1$ is isomorphic to the twist ${}_{\xi}G_0$ for some $\xi \in Z^1(\ell/k , \Sigma)$, where $\Sigma$ is the group of symmetries of the Dynkin diagram of $G_0$, naturally considered as a subgroup of the automorphism group $\mathrm{Aut}(G_0)$. Recall that $G_0$ can be identified with the base change $\mathbf{G} \times_{\mathbb{Z}} k$ of the Chevalley group scheme $\mathbf{G}$ over $\Z$, and that the action of $\Sigma$ comes from its action on $\mathbf{G}$. The Hopf algebra
$k[G_1]$ is then obtained from $\ell[G_1] = \ell[G_0]$ by Galois descent for the action of $\mathrm{Gal}(\ell/k)$ on $\ell[G_0] = \ell \otimes_{\mathbb{Z}}
\mathbb{Z}[\mathbf{G}]$, which coincides with the standard action on $\ell$ and the action via the homomorphism $\xi \colon \mathrm{Gal}(\ell/k) \to \Sigma$ on
$\mathbb{Z}[\mathbf{G}]$. This action leaves $\widetilde{\mathcal{O}}[G_0] := \widetilde{\mathcal{O}} \otimes_{\mathbb{Z}} \mathbb{Z}[\mathbf{G}]$ invariant, and since $\widetilde{\mathcal{O}}/\mathcal{O}_{k,v}$ is a Galois extension of rings, we can carry out Galois descent in this situation. This yields a Hopf $\mathcal{O}_{k,v}$-algebra that corresponds to a reductive group scheme $\mathscr{G}_1$ over $\mathcal{O}_{k,v}$ with generic fiber $G_1$. This verifies that $G_1$ has good reduction in the traditional sense. To prove this fact for $G$, we need a result of Harder, whose proof ultimately depends on weak approximation.

\vskip1mm

To give the statement, we first need to introduce some notations that are different from the ones used elsewhere in this paper.
So, let $A$ be a Dedekind domain with fraction field $k$. For each maximal ideal $\mathfrak{p} \subset A$, denote by $\widehat{k}_{\mathfrak{p}}$ the corresponding completion of $k$ with valuation ring $\widehat{A}_{\mathfrak{p}} \subset \widehat{k}_{\mathfrak{p}}$. Given a flat group scheme $\mathbb{G}$ over $A$, we let $G$ denote its generic fiber $\mathbb{G} \times_A k$, and set $H^1_A(k, G)$ to be the image of the natural map $H^1_{fppf}(A, \mathbb{G}) \to H^1_{fppf}(k, G)$ of flat cohomology. Furthermore, for $\xi \in H^1_{fppf}(k, G)$, we denote by $\xi_{\mathfrak{p}}$ its image in $H^1_{fppf}(\widehat{k}_{\mathfrak{p}}, G)$ under the restriction map.
\begin{prop}\label{P:Harder} {\rm (\cite[Lemma 4.1.3]{Harder})}
Let $\mathbb{G}$ be a flat group scheme of finite type over $A$ whose generic fiber $G$ is a reductive $k$-group. Then
$$
H^1_A(k, G) = \{ \xi \in H^1(k , G) \mid \xi_{\mathfrak{p}} \in \mathrm{Im} (H^1_{fppf}(\widehat{A}_{\mathfrak{p}}, \mathbb{G}) \to H^1_{fppf}(\widehat{k}_{\mathfrak{p}}, G)) \ \ \text{for all maximal ideals} \ \mathfrak{p} \subset A \}.
$$
\end{prop}

Next, let $k$ be a field equipped with a discrete valuation $v$, and set $A$ to be the corresponding valuation ring $\mathcal{O}_{k,v}$. Then it follows from the proposition that for any reductive group scheme $\mathscr{G}$ over $\mathcal{O}_{k,v}$ with generic fiber $G$, we have the following (note that since $\mathscr{G}$ is, by definition, smooth, its flat cohomology coincides with \'etale cohomology).
\begin{cor}\label{C-CorHarder}
The natural diagram of pointed sets
$$
\xymatrix{\he^1(\mathcal{O}_{k,v}, \mathscr{G}) \ar[r] \ar[d]_{\varphi} & \he^1(\mathcal{O}_v, \mathscr{G}) \ar[d]^{\varphi_v} \\ H^1(k, G) \ar[r]^{r_v} & H^1(k_v, G)}
$$
is cartesian.
\end{cor}

Now let $\overline{G}_1$ be the adjoint group for the quasi-split group $G_1$ considered above, and let $\overline{\mathscr{G}}_1$ be a reductive $\mathcal{O}_{k,v}$-scheme with generic fiber $\overline{G}_1$. By construction, $G$ is an inner twist of $G_1$, so we can choose $\xi \in Z^1(k ,
\overline{G}_1)$ so that $G = {}_{\xi} G_1$. We then consider the cartesian diagram from Corollary \ref{C-CorHarder} with $\mathscr{G}$ and $G$ replaced by
$\overline{\mathscr{G}}_1$ and $\overline{G}_1$, respectively, and keeping the notations for the maps. The fact that $G$ has good reduction at $v$ means that we may assume that $r_v([\xi]) \in \im \varphi_v$. We then conclude from the diagram that $[\xi] = \varphi([\zeta])$ for some $\zeta \in Z^1(\mathcal{O}_{k,v} , \overline{\mathscr{G}}_1)$. Then $\mathscr{G} := {}_{\zeta}\mathscr{G}_1$ is a reductive group scheme over $\mathcal{O}_{k,v}$ with generic fiber $G$, as required.

\vskip2mm

\vskip.5cm

\section{Generic tori, generic elements, and applications to weak commensurability}\label{S:generic}

\noindent {\bf 3.1. Generic tori.} For an algebraic torus $T$ defined over a field $k$, we denote by $k_T$ the minimal splitting field of $T$.
It is well-known that the Galois group $\mathscr{G}_T = \Ga(k_T/k)$ acts faithfully on the group of characters $X(T)$. Now, let
$G$ be a semisimple $k$-group, $T$ a maximal $k$-torus of $G$, and $\Phi(G , T)$
the corresponding root system. Then the action of $\mathscr{G}_T$ on $X(T)$ permutes the roots, yielding a group homomorphism
$$
\theta_T \colon \mathscr{G}_T \longrightarrow \mathrm{Aut}(\Phi(G ,
T)) \ \ (\subset \mathrm{GL}(X(T) \otimes_{\mathbb{Z}} \mathbb{Q})).
$$
Since $\Phi(G , T)$ generates a finite index subgroup of $X(T)$, this homomorphism is {\it injective}.
We say that $T$ is {\it generic} over $k$, or $k$-{\it generic}, if the image of $\theta_T$ contains the Weyl group $W(G , T)$.
It is known that if $k$ is an infinite finitely generated field, then every semisimple $k$-group $G$ contains $k$-generic
maximal $k$-tori. This can be established by first showing that $G$ possesses a generic torus over a purely transcendental extension
of $k$ and then specializing the parameters in order to obtain a required generic torus defined over $k$ --- see Voskresenski\={\i} \cite[4.2]{Voskr} and
also \cite{PR-Irred}; we note that the specialization part is based on the fact that an infinite finitely generated field is Hilbertian --- see \cite[Theorem 13.4.2]{FrJar}.

A different approach to the construction of generic tori was first developed in \cite{PR-MRL1} over fields of characteristic zero and then extended to fields of arbitrary characteristic in \cite{PR-Vin}. Among other things, this approach demonstrates that to ensure the genericity of a maximal $k$-torus, it is enough to prescribe its local  behavior at finitely many specially chosen valuations. More precisely, assuming that $\mathrm{char}\: k = 0$, one can choose $r$ distinct primes $p_1, \ldots , p_r$, where $r$ is the number of conjugacy classes in the Weyl group of $G$, such that there exist embeddings $\iota_i \colon k \hookrightarrow \mathbb{Q}_{p_i}$ for $i = 1, \ldots , r$. Furthermore, letting $v_i$ denote the pullback of the $p_i$-adic valuation under $\iota_i$,  for each $i = 1, \ldots , r$, one can specify a maximal $k_{v_i}$-torus $T_i$ of $G$ so that any maximal $k$-torus $T$ of $G$ that is conjugate to $T_i$ by an element of $G(k_{v_i})$ for all $i = 1, \ldots , r$ is necessarily $k$-generic. For a very similar statement in the case of positive characteristic, we refer the reader to \cite{PR-Vin}. This construction of generic tori yields the following stronger form of the existence theorem.
\begin{thm}\label{T:Exist}
{\rm (cf. \cite[Theorem 3.1]{PR-Fields}, \cite{PR-Vin})} Let $G$ be a semisimple algebraic group over an infinite finitely generated field $k$. For any finitely generated extension $\ell$ of $k$, the group $G$ contains a maximal $k$-torus that is generic over $\ell$.
\end{thm}

In connection with this local-global construction, we would like to recall the following approximation statement for maximal tori and
derive one consequence needed for our purposes.
\begin{lemma}\label{L:tor-appr}
Let $G$ be a reductive algebraic group over a field $k$, and let $V$ be a finite set of discrete valuations of $k$. Suppose that for each $v \in V$,
we are given a maximal $k_v$-torus $T_v$ of $G \times_k k_v$. Then there exists a maximal $k$-torus $T$ of $G$ that is conjugate to $T_v$ by an element
of $G(k_v)$ for all $v \in V$.
\end{lemma}

This is Corollary 3 in \cite[\S7.2]{PR}; the proof uses the fact that the variety of maximal tori is rational over $k$.

\vskip.5mm

\begin{cor}\label{C:gen}
Let $G$ be a semisimple algebraic group over a field $k$, and let $v$ be a discrete valuation of $k$. If $G' \in \gen_k(G)$, then
$G' \times_k k_v \in \gen_{k_v}(G \times_k k_v)$.
\end{cor}
Indeed, the lemma implies that every  maximal $k_v$-torus of $G \times_k k_v$ (resp., of $G' \times_k k_v$) is $k_v$-isomorphic to a maximal $k$-torus of $G$
(resp., of $G'$), and our claim immediately follows from the definitions.

\vskip1mm

\begin{prop}\label{P:Prop-gen1}
Let $G_1$ and $G_2$ be absolutely almost simple algebraic groups
over a finitely generated field $k$, and let $\ell_i$ be the minimal
Galois extension of $k$ over which $G_i$ becomes an inner form of the split group.
Assume that $G_1$ and $G_2$ have the same isogeny classes of maximal
$k$-tori. Then

\vskip2mm

\noindent \ {\rm (i)} \parbox[t]{16cm}{either $G_1$ and $G_2$ are of
the same Killing-Cartan type, or one of them is of type $\textsf{B}_{\ell}$ and the
other is of type $\textsf{C}_{\ell}$ for some $\ell \geqslant 3$;}

\vskip1mm

\noindent {\rm (ii)} \parbox[t]{16cm}{$\ell_1 = \ell_2$, and consequently, if the groups $G_1$ and $G_2$ are of the same
Killing-Cartan type and are both either simply connected or adjoint, then they are inner twists of one another.}
\end{prop}
\begin{proof}
These statements were proved in the context of the analysis of weakly commensurable Zariski-dense subgroups
in \cite{PR-WC} and \cite[\S 5]{PR-Appl} --- see also Theorem \ref{T:basic} below. So, we will just briefly outline the argument in our present context
of absolutely almost simple algebraic groups with the same tori.
Set $\ell = \ell_1\ell_2$. Using Theorem \ref{T:Exist}, we can find a maximal $k$-torus
$T_1$ of $G_1$ which is generic over $\ell$.
By our assumption, there exist
a maximal $k$-torus $T_2$ of $G_2$ and a $k$-defined isogeny $\nu \colon T_1 \to T_2$. We then have the following commutative
diagram
$$
\xymatrix{& {\mathrm{GL}}(X(T_1) \otimes_{\Z} \Q) 
\\ \Ga(k^{\mathrm{sep}}/k) \ar[ru]^{\theta_{T_1}} \ar[rd]_{\theta_{T_2}} & \\ & {\mathrm{GL}}(X(T_2) \otimes_{\Z} \Q),\ar[uu]_{\tilde{\nu}}
}
$$
where $k^{\mathrm{sep}}$ is a fixed separable closure of $k$, and $\tilde{\nu}$ is the isomorphism induced by $\nu$. We note that for any field
extension $F$ of $k$ contained in $k^{\mathrm{sep}}$, the map $\tilde{\nu}$ gives an isomorphism between the images of $\Ga(k^{\mathrm{sep}}/F)$ under
$\theta_{T_1}$ and $\theta_{T_2}$, hence
\begin{equation}\label{E:2-1}
\vert \theta_{T_1}(\Ga(k^{\mathrm{sep}}/F)) \vert = \vert
\theta_{T_2}(\Ga(k^{\mathrm{sep}}/F)) \vert.
\end{equation}
Since both $G_1$ and $G_2$ are inner forms over $\ell$, by \cite[Lemma 4.1]{PR-WC} we have
$$
\theta_{T_i}(\Ga(k^{\mathrm{sep}}/\ell)) \subset W(G_i , T_i) \ \ \text{for} \ \ i = 1, 2.
$$
Combining this with the fact that $T_1$ was chosen to be generic over $\ell$, we see that actually
\begin{equation}\label{E:2-2}
\theta_{T_1}(\Ga(k^{\mathrm{sep}}/\ell)) = W(G_1 , T_1).
\end{equation}
Thus (\ref{E:2-1}) with $F = \ell$ yields the inequality $\vert W(G_1 , T_1) \vert \leqslant \vert W(G_2 , T_2) \vert$.
Starting now with a  maximal $k$-torus $T'_2$ of $G_2$ that is generic over $\ell$
and considering a maximal $k$-torus $T'_1$ of $G_1$ that is $k$-isogenous to $T'_2$, we similarly obtain the inequality $\vert W(G_2 , T'_2)
\vert \leqslant \vert W(G_1 , T'_1) \vert$. Since $\vert W(G_i , T_i) \vert = \vert W(G_i , T'_i) \vert$ for $i = 1, 2$, we conclude that
\begin{equation}\label{E:2-3}
\vert W(G_1 , T_1) \vert = \vert W(G_2 , T_2) \vert.
\end{equation}
This already yields assertion (i) as the type of a reduced irreducible root system is uniquely determined by the order of the corresponding
Weyl group except for the ambiguity between types $\textsf{B}_{\ell}$ and $\textsf{C}_{\ell}$ for $\ell \geqslant 3$.  In addition, (\ref{E:2-3})
also implies that
\begin{equation}\label{E:2-4}
\theta_{T_2}(\Ga(k^{\mathrm{sep}}/\ell)) = W(G_2 , T_2).
\end{equation}
Now assume that $\ell_2 \not\subset \ell_1$, i.e. $\ell_1 \subsetneqq \ell$. Since $G_1$ is an inner form already over $\ell_1$, we conclude
from (\ref{E:2-2}) that
$$
\theta_{T_1}(\Ga(k^{\mathrm{sep}}/\ell_1)) = W(G_1 , T_1).
$$
On the other hand, it follows from (\ref{E:2-4}) that $\theta_{T_2}(\Ga(k^{\mathrm{sep}}/\ell_1))$ contains $W(G_2 , T_2)$ but is \emph{strictly
bigger} as $G_2$ is \emph{not} an inner form over $\ell_1$. In view of (\ref{E:2-3}), this contradicts (\ref{E:2-1}) with $F = \ell_1$. Thus, $\ell_2 \subset \ell_1$, and by symmetry we conclude that $\ell_1 = \ell_2$, as required. It is well-known that for absolutely almost simple simply connected or adjoint groups, this fact implies that the groups are inner twists of one another.
\end{proof}

\begin{cor}\label{C:+}
Let $G$ be an absolutely almost simple algebraic group over a finitely generated field $k$. Then $\gen^+_k(G) = \gen_k(G)$.
\end{cor}
\begin{proof}
Let $G' \in \gen^+_k(G)$. Then according to Proposition \ref{P:Prop-gen1}, the group $G'$ is an inner twist of $G$, i.e. $G' \in \gen_k(G)$.
\end{proof}
\vskip1mm

\noindent {\bf 3.2. Generic elements.} Let $G$ be a (connected) absolutely almost simple  algebraic group over a field $k$. A regular semisimple element $\gamma \in G(k)$ of infinite order is called $k$-{\it generic} if the $k$-torus $T = C_G(\gamma)^{\circ}$ (connected component of the centralizer) is $k$-generic. The following result yields the existence of generic elements in Zariski-dense subsemigroups under one natural assumption.
\begin{thm}\label{T:ExGenElts}
{\rm (\cite[Theorem 2]{PR-Vin})}
Let $G$ be an absolutely almost simple algebraic group over a finitely generated field $k$, and let $\Gamma \subset G(k)$ be a Zariski-dense subsemigroup that contains an element of infinite order\footnotemark. Then $\Gamma$ contains a regular semisimple element $\gamma \in \Gamma$ of infinite order that is $k$-generic.
\end{thm}

\footnotetext{We recall that any Zariski-dense subgroup $\Gamma \subset G(k)$, where $G$ is a semisimple group over a field $k$ of characteristic zero, automatically contains an element of infinite order.}

In characteristic zero, the existence of generic elements  of infinite order in an arbitrary Zariski-dense subgroup was established already in \cite{PR-MRL1} for any semisimple $G$. The case of positive characteristic (particularly of characteristics 2 and 3) requires a more delicate argument, which was given in \cite{PR-Vin} assuming $G$ to be absolutely almost simple. We will also need the following refined version of Theorem~\ref{T:ExGenElts} over fields of characteristic zero.

\vskip1mm

\begin{thm}\label{T:ExGenElts2}
{\rm (cf. \cite[Theorem 3.4]{PR-Fields})}
Let $G$ be a connected absolutely almost simple algebraic group over a finitely generated field $k$ of characteristic zero, $v$ be a discrete valuation of $k$ such that the completion $k_v$ is locally compact, and $T(v)$ be a maximal $k_v$-torus of $G$. Given a finitely generated Zariski-dense subgroup $\Gamma \subset G(k)$ whose closure in $G(k_v)$ for the $v$-adic topology is open, there exists a regular semisimple element $\gamma \in \Gamma$ of infinite order such that the corresponding torus $T = C_G(\gamma)^{\circ}$ is generic over $k$ and is conjugate to $T(v)$ by an element of $G(k_v)$.
\end{thm}

\vskip1mm

\noindent {\bf 3.3. Weak commensurability.} (Cf. \cite{PR-WC}) Let $\gamma_1 \in \mathrm{GL}_{n_1}(F)$ and $\gamma_2 \in \mathrm{GL}_{n_2}(F)$ be two semisimple matrices over an infinite field $F$ with respective eigenvalues
$$
\lambda_1, \ldots , \lambda_{n_1} \ \ \text{and} \ \ \mu_1, \ldots , \mu_{n_2}
$$
(in an algebraic closure $\overline{F}$). We say that $\gamma_1$ and $\gamma_2$ are {\it weakly commensurable} if there exist integers $a_1, \ldots , a_{n_1}$ and $
b_1, \ldots , b_{n_2}$ such that
$$
\lambda_1^{a_1} \cdots \lambda_{n_1}^{a_{n_1}} \, = \, \mu_1^{b_1} \cdots \mu_{n_2}^{b_{n_2}} \neq 1.
$$
Next, let $G_1 \subset \mathrm{GL}_{n_1}$ and $G_2 \subset \mathrm{GL}_{n_2}$ be two reductive algebraic $F$-groups, and let $\Gamma_1 \subset G_1(F)$ and $\Gamma_2 \subset G_2(F)$ be Zariski-dense subgroups that contain elements of infinite order. We say that $\Gamma_1$ and $\Gamma_2$ are {\it weakly commensurable} if every semisimple element $\gamma_1 \in \Gamma_1$ of infinite order is weakly commensurable to some semisimple element $\gamma_2 \in \Gamma_2$ of infinite order, and vice versa. It is easy to see that this relation does not depend on the choice of the matrix realizations of $G_1$ and $G_2$.

\vskip.5mm

The following theorem summarizes the basic results about weakly commensurable subgroups.
\begin{thm}\label{T:basic}
Let $G_1$ and $G_2$ be absolutely almost simple algebraic groups over a finitely generated field $k$, and let $\ell_i$ be the minimal Galois extension of $k$ over which $G_i$ becomes an inner form of the split group. Furthermore, let $\Gamma_1 \subset G_1(k)$ and $\Gamma_2 \subset G_2(k)$ be Zariski-dense subgroups containing elements of infinite order. Assume that $\Gamma_1$ and $\Gamma_2$ are weakly commensurable. Then

\vskip1.5mm

\noindent {\rm (1)} \parbox[t]{16cm}{the groups $G_1$ and $G_2$ have the same order of the Weyl groups, or equivalently, they are either of the same type or
one of them is of type $\textsf{B}_{\ell}$ and the other of type $\textsf{C}_{\ell}$ for some $\ell \geq 3$;}

\vskip1mm

\noindent {\rm (2)} \parbox[t]{16cm}{if $\mathrm{char}\: k = 0$, then the trace fields of $\Gamma_1$ and $\Gamma_2$ coincide: $k_{\Gamma_1} = k_{\Gamma_2}$;}

\vskip1mm

\noindent {\rm (3)} \parbox[t]{16cm}{$\ell_1 = \ell_2$.}
\end{thm}

In characteristic zero, part (1) is Theorem 1 in \cite{PR-WC}. Its proof in positive characteristic remains exactly the same due to the existence of generic elements in {\it all} characteristics (Theorem \ref{T:ExGenElts}). The result in part (2) as stated is specific to characteristic zero; in fact, it is false in positive characteristic. Technically, part (3) was proved in \cite[Theorem 6.3(2)]{PR-WC} only when $k$ is a number field, so we will quickly sketch the general argument, which is similar to the proof of Proposition \ref{P:Prop-gen1}. We recall that a $k$-torus $T$ is called $k$-{\it irreducible} if it does not contain any proper $k$-defined subtori; the irreducibility of $T$ is equivalent to the fact that the Galois group $\mathrm{Gal}(k^{\mathrm{sep}}/k)$ acts irreducibly on either of  the $\mathbb{Q}$-vector spaces  $X(T) \otimes_{\mathbb{Z}} \mathbb{Q}$ or  $X_*(T) \otimes_{\mathbb{Z}} \mathbb{Q}$, where $X(T)$ and $X_*(T)$ are, respectively, the groups of characters and cocharacters of $T$,
hence the terminology. We will need the following result.
\begin{lemma}\label{L:generate}
{\rm (\cite[Lemma 3.6]{PR-WC})} Let $T$ be a $k$-irreducible torus. For any $t \in T(k)$ of infinite order and any nonzero character $\chi \in X(T)$, the Galois conjugates of $\lambda = \chi(t)$ generate the splitting field $k_T$.
\end{lemma}

\noindent {\it Proof of Theorem \ref{T:basic}(3).}
Set $\ell = \ell_1\ell_2$. It is enough to prove the inclusion $\ell_1 \subset \ell_2$ as the opposite inclusion is obtained by a symmetric argument. Assume the contrary, i.e. $\ell_1 \not\subset \ell_2$. Using Theorem \ref{T:ExGenElts}, we can find a regular semisimple element $\gamma_1 \in \Gamma_1$ of infinite order which is generic over $\ell$. By our assumption, $\gamma_1$ is weakly commensurable to some semisimple element $\gamma_2 \in \Gamma_2$ of infinite order. Let $T_i$ be a maximal $k$-torus of $G_i$ containing $\gamma_i$. Since $T_1$ is $\ell$-generic, we have the inclusion
$
\theta_{T_1}(\mathrm{Gal}(k^{\mathrm{sep}}/\ell)) \supset W(G_1 , T_1).
$
On the other hand, the fact that $G_1$ is an inner form of a split group over $\ell$ implies the opposite inclusion (see \cite[Lemma 4.1]{PR-WC}). Thus,
$$
\theta_{T_1}(\mathrm{Gal}(k^{\mathrm{sep}}/\ell)) = W(G_1 , T_1),
$$
and in particular, $[\ell_{T_1} : \ell] = \vert W(G_1 , T_1) \vert$. The condition that $\gamma_1$ and $\gamma_2$ are weakly commensurable means that there exist characters $\chi_i \in X(T_i)$ for $i = 1, 2$ such that
$$
\lambda := \chi_1(\gamma_1) = \chi_2(\gamma_2) \neq 1.
$$
It follows from Lemma \ref{L:generate} that the Galois conjugates of $\lambda$ generate the splitting field $k_{T_1}$, yielding, in particular, the inclusion $k_{T_1} \subset k_{T_2}$, hence the inequality
\begin{equation}\label{E:ineq1}
[\ell_{T_2} : \ell] \geq [\ell_{T_1} : \ell] = \vert W(G_1 , T_1) \vert.
\end{equation}
At the same time, again by \cite[Lemma 4.1]{PR-WC}, we have the inclusion $\theta_{T_2}(\mathrm{Gal}(k^{\mathrm{sep}}/\ell)) \subset W(G_2 , T_2)$, so
\begin{equation}\label{E:ineq2}
[\ell_{T_2} : \ell] \leq \vert W(G_2 , T_2) \vert.
\end{equation}
However, by part (1) we have $\vert W(G_1 , T_1) \vert = \vert W(G_2 , T_2) \vert$, so comparing (\ref{E:ineq1}) and (\ref{E:ineq2}), we obtain that
$$
\theta_{T_2}(\mathrm{Gal}(k^{\mathrm{sep}}/\ell)) = W(G_2 , T_2).
$$
By our assumption $\ell \neq \ell_2$, so the last equality implies that
$$
\vert \theta_{T_2}(\mathrm{Gal}(k^{\mathrm{sep}}/\ell_2)) \vert > \vert W(G_2 , T_2) \vert.
$$
This, however, contradicts the inclusion $\theta_{T_2}(\mathrm{Gal}(k^{\mathrm{sep}}/\ell_2)) \subset W(G_2 , T_2)$, which again follows
from \cite[Lemma 4.1]{PR-WC} as $G_2$ is an inner form over $\ell_2$. \hfill $\Box$

\vskip2mm

We conclude this section with the following two statements.

\begin{prop}\label{P:Isog}
{\rm (cf. \cite[Isogeny Theorem 4.2]{PR-WC})}
Let $G_1$ and $G_2$ be two connected absolutely almost simple algebraic groups over an infinite field $k$, and for $i = 1, 2,$ let $\ell_i$ be the minimal Galois extension of $k$ over which $G_i$ becomes an inner form of the split group. Assume that $G_1$ and $G_2$ have the same order of the Weyl groups and that $\ell_1 = \ell_2$. Furthermore, let $T_i$ be a maximal $k$-torus of $G_i$, and let $\gamma_i \in T_i(k)$ be an element of infinite order. If $T_1$ is $k$-generic and the elements $\gamma_1$ and $\gamma_2$ are weakly commensurable, then there exists a $k$-isogeny $\pi \colon T_1 \to T_2$.
\end{prop}

\vskip1mm

\begin{cor}\label{C:Isog}
Let $G_1$ and $G_2$ be absolutely almost simple algebraic groups over an infinite finitely generated field $k$, and let $\Gamma_1 \subset G_1(k)$ and
$\Gamma_2 \subset G_2(k)$ be Zariski-dense subgroups containing elements of infinite order. Assume that $\Gamma_1$ and $\Gamma_2$ are weakly commensurable.
If a $k$-generic element $\gamma_1 \in \Gamma_1$ of infinite order is weakly commensurable to a semisimple element $\gamma_2 \in \Gamma_2$ and $T_i$ is a maximal $k$-torus of $G_i$ containing $\gamma_i$, then there exists a $k$-defined isogeny $\pi \colon T_1 \to T_2$. In particular, the minimal splitting fields of $T_1$ and $T_2$ coincide: $k_{T_1} = k_{T_2}$, and hence the fact that $T_1$ is $k$-generic implies that $T_2$ is also $k$-generic.
\end{cor}
\begin{proof}
According to Theorem \ref{T:basic}, the fact that $\Gamma_1$ and $\Gamma_2$ are weakly commensurable implies that $G_1$ and $G_2$ have the same order of the Weyl group and that $\ell_1 = \ell_2$. Now, our assertion follows immediately from Proposition \ref{P:Isog}.
\end{proof}
%
%
%
%
%
%
%

\vskip5mm

\section{One consequence of a result of Klyachko}\label{S:Klyachko}

We refer to \cite[Ch. VI]{Bour} for the terminology and notations pertaining to root systems. In particular, for a reduced irreducible root system $\Phi$ in a $\mathbb{Q}$-vector space $V$, we let $\Phi^{\vee}$ denote the dual root system, $Q(\Phi)$ the sublattice of $V$ generated by the roots ({\it root lattice}), and $P(\Phi)$ the dual lattice of $Q(\Phi^{\vee})$ ({\it weight lattice}); recall that $Q(\Phi) \subset P(\Phi)$. Furthermore, we denote by $W(\Phi)$ the Weyl group of $\Phi$, viewed as a subgroup of the automorphism group $\mathrm{Aut}(\Phi)$. The following result plays an important role in this paper.
\begin{thm}\label{T:Klyachko10}
Let $\Phi$ be a reduced irreducible root system. For any subgroup $\Gamma \subset \mathrm{Aut}(\Phi)$ containing $W(\Phi)$, we have $H^1(\Gamma , P(\Phi)) = 0$ if $\Phi$ is not of the type $\textsf{A}_1$ or $\textsf{C}_{\ell}$, and $\mathbb{Z}/2\mathbb{Z}$ otherwise.
\end{thm}

This theorem is a particular case of the computations of $H^1(\Gamma , M)$ for any $\Gamma$ as in the theorem and any $\Gamma$-invariant lattice $Q(\Phi) \subset M \subset P(\Phi)$ carried out by A.~Klyachko \cite{Klyachko}. Unfortunately, the result in {\it loc. cit.} is {\it false} as stated; however, the argument given therein {\it does} work in the situation described in the theorem. Since \cite{Klyachko} is not readily available, we will reproduce the argument in Appendix~1, where we will also present the original statement and explain the mistake that invalidates the argument in the general case.

\begin{cor}\label{C:triv}
Let $\overline{G}$ be a simple adjoint algebraic group over a field $K$, and $\overline{T}$ be a maximal $K$-torus of $\overline{G}$ which is $K$-generic.
If the type of $\overline{G}$ is different from $\textsf{A}_1$ and $\textsf{B}_{\ell}$ $(\ell \geq 2)$, then for the group of cocharacters $X_*(\overline{T})$, we have $H^1(\mathrm{Gal}(K_{\overline{T}}/K) , X_*(\overline{T})) = 0$.
\end{cor}
\begin{proof}
Let $\Phi = \Phi(\overline{G} , \overline{T})$ be the root system of $\overline{G}$. Since $\overline{G}$ is adjoint, the character group $X(\overline{T})$ coincides with $Q(\Phi)$. So, the dual group $X_*(\overline{T})$ of cocharacters can be identified with $P(\Phi^{\vee})$. By our assumption, the type of $\Phi$ is different from $\textsf{A}_1$ and $\textsf{B}_{\ell}$, so the type of the dual system $\Phi^{\vee}$ is different from $\textsf{A}_1$ and $\textsf{C}_{\ell}$. Furthermore, the fact that $\overline{T}$ is generic means that the Galois group $\mathrm{Gal}(K_{\overline{T}}/K)$ in its action on $X_*(\overline{T})$ contains the Weyl group $W(\Phi)$ --- cf. \S3.1. Our assertion now follows directly from Theorem \ref{T:Klyachko10}.
\end{proof}

Since the proof of Theorem \ref{T:Klyachko10} is deferred to Appendix~1, we will now give an example that, on the one hand, shows a situation where the corollary can be checked by a direct computation, and on the other hand, demonstrates that the assertion can be false if the ambient group is not adjoint.

\vskip1mm

\addtocounter{thm}{1}

\noindent {\bf Example 4.3.} Let $L/K$ be a separable field extension of degree $n > 2$ such that  the Galois group $\mathrm{Gal}(M/K)$ of the minimal Galois extension $M$ of $K$ that contains $L$ is isomorphic to $S_n$. Set $\mathcal{G} = \mathrm{Gal}(M/K)$ and $\mathcal{H} = \mathrm{Gal}(M/L)$. Corresponding to the extension $L/K$, we have a maximal $K$-generic $K$-torus $\overline{T} = \mathrm{R}_{L/K}(\mathbb{G}_m)/\mathbb{G}_m$ of the adjoint group $G = \mathrm{PGL}_n$ of type $\textsf{A}_{n-1}$. The group of cocharacters $X_*(\overline{T})$ fits into the following exact sequence of $\mathcal{G}$-modules
$$
0 \to \mathbb{Z} \longrightarrow \mathbb{Z}[\mathcal{G}/\mathcal{H}] \longrightarrow X_*(\overline{T}) \to 0,
$$
leading to the exact sequence in cohomology
$$
0 = H^1(\mathcal{G} , \mathbb{Z}[\mathcal{G}/\mathcal{H}]) \longrightarrow H^1(\mathcal{G} , X_*(\overline{T})) \longrightarrow H^2(\mathcal{G} , \mathbb{Z})
\stackrel{\alpha}{\longrightarrow} H^2(\mathcal{G} , \mathbb{Z}[\mathcal{G}/\mathcal{H}]) = H^2(\mathcal{H} , \mathbb{Z}).
$$
In terms of the natural identifications $$H^2(\mathcal{G} , \mathbb{Z}) \simeq \mathrm{Hom}(\mathcal{G} , \mathbb{Q}/\mathbb{Z}) \ \ \text{and} \ \
H^2(\mathcal{H} , \mathbb{Z}) \simeq \mathrm{Hom}(\mathcal{H} , \mathbb{Q}/\mathbb{Z}),$$
the map $\alpha$ corresponds to the restriction map
$$
\mathrm{Hom}(\mathcal{G} , \mathbb{Q}/\mathbb{Z}) \longrightarrow \mathrm{Hom}(\mathcal{H} , \mathbb{Q}/\mathbb{Z}).
$$
It easily follows that $\alpha$ is injective for $n > 2$, and we obtain $H^1(\mathcal{G} , X_*(\overline{T})) = 0$, in agreement with Corollary \ref{C:triv}.

On the other hand, the norm torus $\widetilde{T} = \mathrm{R}^{(1)}_{L/K}(\mathbb{G}_m)$ is a maximal $K$-generic $K$-torus in the simply connected group $\widetilde{G} = \mathrm{SL}_n$. The co-character group $X_*(\widetilde{T})$ can be determined from  the following exact sequence of $\mathcal{G}$-modules
$$
0 \to X_*(\widetilde{T}) \longrightarrow \mathbb{Z}[\mathcal{G}/\mathcal{H}] \stackrel{\delta}{\longrightarrow} \mathbb{Z} \to 0,
$$
where $\delta$ is the augmentation map. This induces the exact sequence
$$
\mathbb{Z}[\mathcal{G}/\mathcal{H}]^{\mathcal{G}} \stackrel{\delta}{\longrightarrow} \mathbb{Z} \longrightarrow H^1(\mathcal{G} , X_*(\widetilde{T})) \longrightarrow H^1(\mathcal{G} , \mathbb{Z}[\mathcal{G}/\mathcal{H}]) = H^1(\mathcal{H} , \mathbb{Z}) = 0.
$$
It follows that $H^1(\mathcal{G} , X_*(\widetilde{T})) \simeq \mathbb{Z}/n\mathbb{Z}$; in particular, it is nontrivial (including the case $n = 2$).

\vskip1mm

We will now discuss a consequence of Corollary \ref{C:triv} for unramified cohomology that will be needed in subsequent sections. Let $\mathcal{K}$ be a field complete with respect to a discrete valuation $v$. For any algebraic extension $\mathcal{L}/\mathcal{K}$, we let $\mathcal{O}_{\mathcal{L}}$ denote the valuation ring of the unique extension of $v$ to $\mathcal{L}$.
We also denote by $\mathcal{K}^{\mathrm{ur}}$ the maximal unramified extension of $\mathcal{K}$. Suppose now that $T$ is a $\mathcal{K}$-torus whose minimal splitting field $\mathcal{L} = \mathcal{K}_T$ is unramified over $\mathcal{K}$.  It follows from Hilbert's Theorem 90 and the inflation-restriction sequence that $$H^1(\mathcal{K} , T) = H^1(\mathcal{K}^{\mathrm{ur}}/\mathcal{K} , T) = H^1(\mathcal{L}/\mathcal{K} , T),$$
and one also shows that
$$
H^1(\mathcal{K}^{\mathrm{ur}}/\mathcal{K} , T(\mathcal{O}_{\mathcal{K}^{\mathrm{ur}}})) = H^1(\mathcal{L}/\mathcal{K} , T(\mathcal{O}_{\mathcal{L}})).
$$
The subgroup of unramified cocycles $H^1(\mathcal{L}/\mathcal{K} , T)_{\{v\}} \subset H^1(\mathcal{L}/\mathcal{K} , T)$ is defined as the image of the natural
homomorphism $H^1(\mathcal{L}/\mathcal{K} , T(\mathcal{O}_{\mathcal{L}})) \to H^1(\mathcal{L}/\mathcal{K} , T)$.

\begin{prop}\label{P:all-unram}
Let $\overline{T}$ be a maximal $\mathcal{K}$-torus of an absolutely simple adjoint algebraic $\mathcal{K}$-group $\overline{G}$ of type different from
$\textsf{A}_1$ and $\textsf{B}_{\ell}$. If $\overline{T}$ is $\mathcal{K}$-generic with unramified minimal splitting field $\mathcal{L} = \mathcal{K}_T$,  then $H^1(\mathcal{L}/\mathcal{K} , \overline{T}) = H^1(\mathcal{L}/\mathcal{K} , \overline{T})_{\{v\}}$.
\end{prop}
\begin{proof}
We will view cocharacters of $\overline{T}$ as 1-parameter subgroups $\mathbb{G}_m \to \overline{T}$. Then
the map
$$
X_*(\overline{T}) \otimes_{\mathbb{Z}} \mathcal{L}^{\times} \to \overline{T}(\mathcal{L}), \ \ \chi \otimes a \mapsto \chi(a),
$$
is an isomorphism of $\mathrm{Gal}(\mathcal{L}/\mathcal{K})$-modules. Furthermore, if $\pi \in \mathcal{K}$ is a uniformizer, then since $\mathcal{L}/\mathcal{K}$ is unramified, $\pi$ remains a uniformizer in $\mathcal{L}$, and therefore we have a decomposition of $\mathrm{Gal}(\mathcal{L}/\mathcal{K})$-modules
$$
\mathcal{L}^{\times} = \langle \pi \rangle \times \mathcal{U}, \ \ \text{where} \ \ \mathcal{U} = (\mathcal{O}_{\mathcal{L}})^{\times}.
$$
It follows that
$$
\overline{T}(\mathcal{L}) \simeq X_*(\overline{T}) \times \overline{T}(\mathcal{O}_{\mathcal{L}})
$$
as $X_*(\overline{T}) \otimes \mathcal{U} \simeq \overline{T}(\mathcal{O}_{\mathcal{L}})$. In view of our assumptions on $\overline{T}$, we have
$H^1(\mathcal{L}/\mathcal{K} , X_*(\overline{T})) = 0$ by Corollary \ref{C:triv},  and the required fact follows.
\end{proof}

\vskip.5mm

\noindent {\bf Example 4.5.} Let $\ell/k$ be a finite separable extension of degree $n$ such that the minimal Galois extension $m$ of $k$ containing $\ell$ has Galois group $S_n$ over $k$. Set $\mathcal{K} = k((x))$ and $\mathcal{L} = \ell((x))$. Then the norm torus $T = \mathrm{R}^{(1)}_{\mathcal{L}/\mathcal{K}}(\mathbb{G}_m)$ is a maximal $\mathcal{K}$-torus in the simply connected group $G = \mathrm{SL}_n$. Furthermore, this torus  is $\mathcal{K}$-generic and its splitting field is unramified over $\mathcal{K}$ (with respect to the standard valuation $v$ on the field of Laurent power series). We have
$$
H^1(\mathcal{K} , T) = \mathcal{K}^{\times}/N_{\mathcal{L}/\mathcal{K}}(\mathcal{L}^{\times}) = k^{\times}/N_{\ell/k}(\ell^{\times}) \times \langle x \rangle / \langle x^n \rangle.
$$
At the same time, the unramified part $H^1(\mathcal{K} , T)_v$ is easily seen to be $k^{\times}/N_{\ell/k}(\ell^{\times})$. Thus, in this case, $H^1(\mathcal{K} , T) \neq H^1(\mathcal{K} , T)_{\{v\}}$. So, the assertion of Proposition \ref{P:all-unram} may fail if the ambient group is not adjoint.

\section{Maximal tori with unramified splitting fields}

Let $\mathcal{K}$ be a field that is complete with respect to a discrete valuation $v$, with valuation ring $\mathcal{O}$ and residue field $k$. We also fix a uniformizer $\pi \in \mathcal{K}$. The goal of this section is to establish the following result, which may be known to some experts, but which does not seem to have been recorded in the literature.
\begin{thm}\label{T:unram-split-field}
Let $G$ be a reductive algebraic $\mathcal{K}$-group. Assume that $G$ has \emph{good reduction} at $v$, i.e. there exists a reductive group scheme $\mathscr{G}$ over $\mathcal{O}$ with generic fiber $G$. Then given a maximal $\mathcal{K}$-torus $S$ of $G$ whose splitting field $\mathcal{K}_S$ is unramified over $\mathcal{K}$, there exists a maximal torus $\mathscr{S}'$ of $\mathscr{G}$ such that for its generic fiber $S'$, there exists $h \in G(\mathcal{K}^{\mathrm{ur}})$ satisfying $S' = hSh^{-1}$ and the isomorphism $\varphi \colon S \to S'$, $x \mapsto h x h^{-1}$, is defined over $\mathcal{K}$.
\end{thm}

We begin by recalling the well-known parametrization of the conjugacy classes of maximal tori in terms of Galois cohomology. So, let $G$ be a (connected) reductive algebraic group over an arbitrary field $K$. Fix a maximal $K$-torus $T$ of $G$, and let $N = N_G(T)$ denote its normalizer in $G$. Furthermore, let $W = N/T$ denote the Weyl group, $\theta \colon N \to W$  the corresponding quotient map, and $\theta^1 \colon H^1(K , N) \to H^1(K , W)$ the induced map on Galois cohomology. Given any other maximal $K$-torus $S$, we choose $g \in G(K^{\mathrm{sep}})$ so that $S = gTg^{-1}$. Then for any $\sigma \in \mathrm{Gal}(K^{\mathrm{sep}}/K)$, the element $\xi(\sigma) := g^{-1} \cdot \sigma(g)$ belongs to $N(K^{\mathrm{sep}})$, and the correspondence $\sigma \mapsto \xi(\sigma)$ is a 1-cocycle with values in $N(K^{\mathrm{sep}})$ whose cohomology class  $[\xi] \in  H^1(K , N)$ is independent of the choice of the conjugating element $g$. Furthermore, the correspondence
$$
S \ \mapsto \ [\xi]
$$
sets up a bijection between the $G(K)$-conjugacy classes of maximal $K$-tori of $G$ and the elements of $\ker(H^1(K , N) \to H^1(K , G))$. (More generally, if $T$ splits over an extension $L/K$, then the above correspondence sets up a bijection between the $G(K)$-conjugacy classes of maximal $K$-tori of $G$ that split over $L$ and the elements of $\ker(H^1(L/K , N) \to H^1(L/K , G))$.) We will need the following
version of this fact.
\begin{lemma}\label{L:sameW}
Let $S_1$ and $S_2$ be two maximal $K$-tori of $G$, and let $[\xi_1] , [\xi_2] \in \ker(H^1(K , N) \to H^1(K , G))$ be the corresponding cohomology classes. Then $\theta^1([\xi_1]) = \theta^1([\xi_2])$ if and only if there exists $h \in G(K^{\mathrm{sep}})$ such that $S_2 = hS_1h^{-1}$ and the isomorphism $\varphi \colon S_1 \to S_2$, $x \mapsto h x h^{-1}$, is defined over $K$.
\end{lemma}
\begin{proof}
Clearly, $\varphi^{-1} \circ \sigma(\varphi) \colon S_1 \to S_1$ is given by $x \mapsto (h^{-1} \sigma(h)) x (h^{-1} \sigma(h))^{-1}$, and therefore $\varphi$ is $K$-defined if and only if $s(\sigma) := h^{-1} \sigma(h) \in S_1$ for all $\sigma \in \mathrm{Gal}(K^{\mathrm{sep}}/K)$.

\vskip2mm

$\Leftarrow)$ Let $g_1 \in G(K^{\mathrm{sep}})$ be such that $S_1 = g_1 T g_1^{-1}$. Then $g_2 = hg_1$ satisfies $S_2 = g_2 T g_2^{-1}$, and the cocycles $\xi_i(\sigma) = g_i^{-1} \sigma(g_i)$, $i = 1, 2$, corresponding to $S_1$ and $S_2$, are related by
$$
\xi_2(\sigma) = g_1^{-1} s(\sigma) \sigma(g_1) = (g_1^{-1} s(\sigma) g_1) \xi_1(\sigma).
$$
Since $g_1^{-1} s(\sigma) g_1 \in T$, we have $\theta(\xi_1(\sigma)) = \theta(\xi_2(\sigma))$ for all $\sigma$, and therefore $\theta^1([\xi_1]) = \theta^1([\xi_2])$.

\vskip2mm

$\Rightarrow$) Changing $\xi_2$ to an equivalent cocycle (which amounts to a different choice of $g_2$ for which $S_2 = g_2Tg_2^{-1}$), we may assume that the elements $\xi_i(\sigma) = g_i^{-1} \sigma(g_i)$, $i = 1, 2$, satisfy $\theta(\xi_1(\sigma)) = \theta(\xi_2(\sigma))$, i.e.
$$
\xi_2(\sigma) = \xi_1(\sigma) t(\sigma) \ \ \text{with} \ \ t(\sigma) \in T,
$$
for all $\sigma \in \mathrm{Gal}(K^{\mathrm{sep}}/K)$. Set $h = g_2g_1^{-1}$. It is enough to show that $h^{-1} \sigma(h) \in S_1$ for all $\sigma$. We have
$$
h^{-1} \sigma(h) = g_1 \xi_2(\sigma) \sigma(g_1)^{-1} = g_1(\xi_1(\sigma) t(\sigma)) \sigma(g_1)^{-1} = (g_1 \xi_1(\sigma) \sigma(g_1)^{-1}) \cdot (\sigma(g_1) t(\sigma) \sigma(g_1)^{-1}) =
$$
$$
=\sigma(g_1 t(\sigma) g_1^{-1}) \in S_1(K^{\mathrm{sep}}),
$$
as required.
\end{proof}

Beginning the proof of Theorem \ref{T:unram-split-field}, we pick a maximal torus $\mathscr{T}$ of $\mathscr{G}$ (cf. \cite[Exp. IX, 7.3]{SGA3}); then its generic fiber $T$ is a maximal $\mathcal{K}$-torus of $G$ whose splitting field $\mathcal{K}_T$ is unramified over $\mathcal{K}$. Let $N = N_G(T)$ and $\mathscr{N} = N_{\mathscr{G}}(\mathscr{T})$ be the corresponding normalizers. We denote by $\mathcal{O}^{\mathrm{ur}}$ the valuation ring of the maximal unramified extension $\mathcal{K}^{\mathrm{ur}}$. Then the Weyl group $W = N/T$ can be identified with
\begin{equation}\label{E:ident}
N(\mathcal{K}^{\mathrm{ur}})/T(\mathcal{K}^{\mathrm{ur}}) = \mathscr{N}(\mathcal{O}^{\mathrm{ur}})/\mathscr{T}(\mathcal{O}^{\mathrm{ur}}).
\end{equation}
Since by assumption the torus $S$ splits over $\mathcal{K}^{\mathrm{ur}}$, it corresponds to some class $[\xi] \in \ker(H^1(\mathcal{K}^{\mathrm{ur}}/\mathcal{K} , N) \to
H^1(\mathcal{K}^{\mathrm{ur}}/\mathcal{K} , G))$. Since the elements of $\ker(H^1(\mathcal{K}^{\mathrm{ur}}/\mathcal{K} ,
\mathscr{N}(\mathcal{O}^{\mathrm{ur}})) \to H^1(\mathcal{K}^{\mathrm{ur}}/\mathcal{K} , \mathscr{G}(\mathcal{O}^{\mathrm{ur}}))$ correspond to the maximal tori
of $\mathscr{G}$, it follows from Lemma \ref{L:sameW} that it is enough to construct a class $[\xi']$ in this set that satisfies $\theta^1([\xi]) = \theta^1([\xi'])$.

\begin{lemma}\label{L:xi'}
There exists a cocycle $\xi' \in Z^1(\mathcal{K}^{\mathrm{ur}}/\mathcal{K} , \mathscr{N}(\mathcal{O}^{\mathrm{ur}}))$ such that

\vskip1mm

\noindent {\rm (1)} $\theta^1([\xi]) = \theta^1([\xi'])$;

\vskip1mm

\noindent {\rm (2)} \parbox[t]{15cm}{there exists $n \geq 1$ such that for $\mathcal{K}' = \mathcal{K}(\sqrt[n]{\pi})$, the image of $[\xi']$ in
$H^1((\mathcal{K}')^{\mathrm{ur}}/\mathcal{K}' , G)$ is trivial.}
\end{lemma}

We will now assume the lemma and complete the proof of Theorem \ref{T:unram-split-field}. In view of the validity of the Grothendieck-Serre conjecture over discrete valuation rings (cf. \S2.2), the image of $[\xi']$ is trivial in $H^1((\mathcal{K}')^{\mathrm{ur}}/\mathcal{K}' , \mathscr{G}({\mathcal{O}'}^{\mathrm{ur}}))$, where ${\mathcal{O}'}^{\mathrm{ur}}$ is the valuation ring of
$(\mathcal{K}')^{\mathrm{ur}}$. We note that $\mathcal{K}$ and $\mathcal{K}'$ have the same
residue field $k$, and that the residue of $\xi'$ is the trivial cocycle with values in $\underline{\mathscr{G}}(k^{\mathrm{sep}})$, where $\underline{\mathscr{G}}$ is the reduction of $\mathscr{G}$.  Applying Hensel's Lemma, we conclude that the class $[\xi']$ is trivial in $H^1(\mathcal{K}^{\mathrm{ur}}/\mathcal{K} , \mathscr{G}(\mathcal{O}^{\mathrm{ur}}))$, as required.

\noindent {\it Proof of Lemma \ref{L:xi'}.} Using (\ref{E:ident}), for each $\sigma \in \mathrm{Gal}(\mathcal{K}^{\mathrm{ur}}/\mathcal{K})$ we can pick $n(\sigma) \in \mathscr{N}(\mathcal{O}^{\mathrm{ur}})$ so that $\theta(\xi(\sigma)) = \theta(n(\sigma))$, i.e.
\begin{equation}\label{E:xi2}
\xi(\sigma) = n(\sigma) t(\sigma) \ \ \text{with} \ \ t(\sigma) \in T(\mathcal{K}^{\mathrm{ur}}).
\end{equation}
As in the proof of Proposition \ref{P:all-unram}, we have a canonical isomorphism of modules over $\Gamma = \mathrm{Gal}(\mathcal{K}^{\mathrm{ur}}/\mathcal{K})$: $$
X_*(T) \otimes_{\mathbb{Z}} (\mathcal{K}^{\mathrm{ur}})^{\times} \to T(\mathcal{K}^{\mathrm{ur}}), \ \ \chi \otimes a \mapsto \chi(a).
$$
Furthermore, we have the following direct product decomposition of $\Gamma$-modules:
$$
(\mathcal{K}^{\mathrm{ur}})^{\times} = \mathcal{U} \times \langle \pi \rangle,
$$
where $\mathcal{U} = (\mathcal{O}^{\mathrm{ur}})^{\times}$ is the group of units in $\mathcal{K}^{\mathrm{ur}}$. Now, set
$$
A := X_*(T) \otimes_{\mathbb{Z}} \mathcal{U} \simeq T(\mathcal{O}^{\mathrm{ur}}) \subset T(\mathcal{K}^{\mathrm{ur}}) \ \ \text{and} \ \
B := X_*(T) \otimes_{\mathbb{Z}} \langle \pi \rangle \subset T(\mathcal{K}^{\mathrm{ur}}),
$$
noting that $A$ and $B$ are invariant under the action of $\Gamma$ as well as under conjugation by elements of $N(\mathcal{K}^{\mathrm{ur}})$ and that
$$
T(\mathcal{K}^{\mathrm{ur}}) = A \times B \ \ \text{as} \ \ \Gamma\text{-modules}.
$$
So, we can write $t(\sigma)$ in (\ref{E:xi2}) as
$$
t(\sigma) = a(\sigma) b(\sigma) \ \ \text{with} \ \ a(\sigma) \in A, \ b(\sigma) \in B.
$$
Set $\xi'(\sigma) = n(\sigma) a(\sigma) \in \mathscr{N}(\mathcal{O}^{\mathrm{ur}})$. Using the cocycle condition for $\xi$ in conjunction with the fact that
$\xi(\sigma) = \xi'(\sigma) b(\sigma)$, we obtain the following relation:
$$
(\xi'(\sigma) \cdot \sigma(\xi'(\tau)))^{-1} \cdot \xi'(\sigma\tau) = (\sigma(n(\tau))^{-1} \cdot b(\sigma) \cdot \sigma(n(\tau))) \cdot \sigma(b(\tau)) \cdot b(\sigma\tau)^{-1}.
$$
Clearly, the left-hand side belongs to $\mathscr{N}(\mathcal{O}^{\mathrm{ur}})$, and the right-hand side to $B$. It follows that the left-hand side is actually
in $A$, hence both sides are equal to $1$. In other words, $\xi'$ is a cocycle that  satisfies $\theta^1([\xi]) = \theta^1([\xi'])$. Furthermore, we have
$$
b(\sigma\tau) = ((\sigma(n(\tau)))^{-1} b(\sigma) \sigma(n(\tau))) \cdot \sigma(b(\tau)).
$$
Conjugating this relation by $\xi(\sigma\tau)$ and using the fact that $\xi$ is a cocycle and that $\xi(\sigma) t \xi(\sigma)^{-1} = n(\sigma) t n(\sigma)^{-1}$ for $t \in T(\mathcal{K}^{\mathrm{ur}})$, we see that $\nu(\sigma) = n(\sigma) b(\sigma) n(\sigma)^{-1}$ defines a Galois cocycle with values in ${}_{\xi}T(\mathcal{K}^{\mathrm{ur}})$, where ${}_{\xi}T$ denotes the twist of $T$ by $\xi$. Let $\mathcal{L}$ be the minimal splitting field of ${}_{\xi}T$ (which by construction is unramified over $\mathcal{K}$), and let $n = [\mathcal{L} : \mathcal{K}]$. Set $K' = K(\sqrt[n]{\pi})$. We claim that the image of $[\nu] \in H^1(\mathcal{K}^{\mathrm{ur}}/\mathcal{K} , {}_{\xi}T)$ in $H^1(\mathcal{K}^{\mathrm{ur}}\mathcal{K}'/\mathcal{K}' , {}_{\xi}T)$ is trivial. Indeed, it follows from Hilbert's Theorem 90 that every element in the latter group is annihilated by multiplication by $n$. Now, the cocycle $\nu$ has values in
$$
{}_{\xi}B := X_*({}_{\xi}T) \otimes_{\mathbb{Z}} \langle \pi \rangle \subset {}_{\xi}T(\mathcal{L}).
$$
After base change from $\mathcal{K}$ to $\mathcal{K}'$, we can consider a similar subgroup
$$
{}_{\xi}B' = X_*({}_{\xi}T) \otimes_{\mathbb{Z}} \langle \pi^{1/n} \rangle \subset {}_{\xi}T(\mathcal{L}\mathcal{K}').
$$
Since every element of ${}_{\xi}B$ can be uniquely divided by $n$ in ${}_{\xi}B'$, there is a cocycle $\nu'$ with values in ${}_{\xi}B' \subset {}_{\xi}T(\mathcal{L}\mathcal{K}')$ such that $\nu = n \cdot  \nu'$. But then it follows from the remark above that the image of the class $[\nu]$ in $H^1(\mathcal{K}^{\mathrm{ur}}\mathcal{K}'/\mathcal{K}' , {}_{\xi}T)$ is trivial. This means that there exists $s \in {}_{\xi}T(\mathcal{L}\mathcal{K}')$
such that
$$
n(\sigma) b(\sigma) n(\sigma)^{-1} = s^{-1} \cdot (n(\sigma) \sigma(s) n(\sigma)^{-1}).
$$
Then $b(\sigma) = (n(\sigma)^{-1} s^{-1} n(\sigma)) \cdot \sigma(s)$, which implies that
$$
\xi(\sigma) = s^{-1} \cdot \xi'(\sigma) \cdot \sigma(s),
$$
for all $\sigma \in \mathrm{Gal}(K^{\mathrm{ur}} \mathcal{K}'/\mathcal{K}')$. This means that the classes $[\xi]$ and $[\xi']$ are mapped to the same element
in $H^1(K^{\mathrm{ur}}\mathcal{K}'/\mathcal{K}' , G)$, and therefore the image of $[\xi']$ is in fact trivial, as required. \hfill $\Box$

\section{Proof of Theorem \ref{T:GoodReduction1}}

Theorem \ref{T:GoodReduction1} is an easy consequence of the following result.

\begin{thm}\label{T:GR-all}
Let $\mathcal{K}$ be a field complete with respect to a discrete valuation $v$ such that the residue field $k$ is finitely generated, and let $G$ be an absolutely almost simple $\mathcal{K}$-group that has good reduction at $v$. Assume that $\mathrm{char}\: k \neq 2$ if $G$ is of type $\textsf{B}_{\ell}$
$(\ell \geq 2)$. Then any  $G' \in \gen_{\mathcal{K}}(G)$  also has good reduction at $v$.
\end{thm}

\vskip1mm

For the proof, we will consider separately the two cases where the type of $G$ is different from $\textsf{A}_1$ and $\textsf{B}_{\ell}$ $(\ell \geq 2)$ and where it is one of those types. In each case, we will characterize the existence of good reduction in terms of the presence of maximal tori with very specific properties --- see Theorems \ref{T:GR-not} and \ref{T:AB}. These characterizations will also be used in \S \ref{S:WC} for the analysis of weakly commensurable Zariski-dense subgroups. We begin with the following sufficient condition for good reduction for types different from $\textsf{A}_1$ and $\textsf{B}_{\ell}$ $(\ell \geq 2)$.
\begin{thm}\label{T:GR-not}
Let $\mathcal{K}$ be a field complete with respect to a discrete valuation $v$, and let $G$ be an absolutely almost simple algebraic
$\mathcal{K}$-group of type different from $\textsf{A}_1$ and $\textsf{B}_{\ell}$ $(\ell \geq 2)$. Assume that $G$ contains a maximal $\mathcal{K}$-torus $T$ which is $\mathcal{K}$-generic and whose minimal splitting field $\mathcal{K}_T$ is unramified over $\mathcal{K}$. Then $G$ has good reduction at $v$.
\end{thm}
\begin{proof}
Let $G_0$ be the quasi-split inner form of $G$, and let $\mathcal{L}$ be the minimal Galois extension of $\mathcal{K}$ over which $G_0$ becomes split. Being a subextension of $\mathcal{K}_T/\mathcal{K}$, the extension $\mathcal{L}/\mathcal{K}$ is unramified, and therefore $G_0$ and the corresponding adjoint group
$\overline{G}_0$ have good reduction (cf. \cite[Corollary 7.9.4]{KalPr}). Let $\overline{\mathscr{G}}_0$ be the corresponding model for $\overline{G}_0$ over the valuation ring $\mathcal{O}$ of $\mathcal{K}$. Now, let $\rho \colon G \to \overline{G}$ be the isogeny onto the adjoint group, and $\overline{T} = \rho(T)$. It follows from Steinberg's Theorem (cf. \cite[Proposition 6.19]{PR}, \cite[8.6]{BorSpr}) that there exist an embedding $\iota \colon \overline{T} \hookrightarrow \overline{G}_0$ and a 1-cocycle $\xi \in Z^1(\mathcal{K} , \overline{T})$ such that for the image $\bar{\xi}$ of $\xi$ under the natural map $Z^1(\mathcal{K} , \overline{T}) \to Z^1(\mathcal{K} , \overline{G}_0)$, the twisted group ${}_{\bar{\xi}} G_0$ is $\mathcal{K}$-isomorphic to $G$. Since $\mathcal{K}_T/\mathcal{K}$ is unramified, according to Theorem \ref{T:unram-split-field}, there exist a maximal torus $\overline{\mathscr{S}}$ of $\overline{\mathscr{G}}_0$ and an element $h \in \overline{G}_0(\mathcal{K}^{\mathrm{sep}})$ such that the generic fiber $\overline{S}$ satisfies $\overline{S} = h \overline{T} h^{-1}$ and the morphism $\varphi \colon \overline{T} \to \overline{S}$, $x \mapsto hxh^{-1}$, is defined over $\mathcal{K}$; recall that the latter is equivalent to the fact that $s(\sigma) := h^{-1} \cdot \sigma(h)$ lies in $\overline{T}(\mathcal{K}^{\mathrm{sep}})$ for all $\sigma \in \mathrm{Gal}(\mathcal{K}^{\mathrm{sep}}/\mathcal{K})$. Let $\xi'$ be the image of $\xi$ under the map $Z^1(\mathcal{K} , \overline{T}) \to Z^1(\mathcal{K} , \overline{S})$ induced by $\varphi$. We have
\begin{equation}\label{E:Scocycle}
\xi(\sigma) = h^{-1} \xi'(\sigma) h = h^{-1} (\xi'(\sigma) \cdot  (h \sigma(h)^{-1}) ) \sigma(h).
\end{equation}
Next, the equation
$$
\zeta(\sigma) := \xi'(\sigma) \cdot (h \sigma(h)^{-1}) = \xi'(\sigma) \cdot (h s(\sigma)^{-1} h^{-1}) = h( \xi(\sigma) \cdot s(\sigma)^{-1}) h^{-1}.
$$
defines a cocycle $\zeta \in Z^1(\mathcal{K} , \overline{S})$,  and we let $\bar{\zeta}$ denote its image under the map $Z^1(\mathcal{K} ,  \overline{S}) \to Z^1(\mathcal{K} , \overline{G}_0)$. It follows from (\ref{E:Scocycle}) that $\bar{\xi} = \bar{\zeta}$, and hence $G \simeq {}_{\bar{\xi}} G_0 \simeq {}_{\bar{\zeta}} G_0$.  So, it remains to show that ${}_{\bar{\zeta}} G_0$ has a reductive model
over $\mathcal{O}$.  Since $\varphi$ is defined over $\mathcal{K}$, we have
$$
\mathcal{K}_{\overline{S}} = \mathcal{K}_{\overline{T}} = \mathcal{K}_T,
$$
which is unramified over $\mathcal{K}$. In addition, $\overline{S}$ is generic over $\mathcal{K}$, so by Proposition \ref{P:all-unram}
$$
H^1(\mathcal{K} , \overline{S}) = H^1(\mathcal{K}_{\overline{S}}/\mathcal{K} , \overline{S}) = H^1(\mathcal{K}_{\overline{S}}/\mathcal{K} , \mathcal{U}),
$$
where $\mathcal{U} = \overline{S}(\mathcal{O}_{\mathcal{K}_{\overline{S}}})$ and $\mathcal{O}_{\mathcal{K}_{\overline{S}}}$ is the valuation ring of $\mathcal{K}_{\overline{S}}$. Thus, replacing $\zeta$ with an equivalent cocycle, we may assume that it has values in $\mathcal{U}$. Obviously, the inner automorphisms corresponding to the elements of $\mathcal{U}$ act on $\mathscr{G}_0 \times_{\mathcal{O}} \mathcal{O}_{\mathcal{K}_{\overline{S}}}$, and then the corresponding  twisted $\mathcal{O}$-scheme $\mathscr{G} = {}_{\bar{\zeta}} \mathscr{G}_0$ is a required reductive model for  ${}_{\bar{\zeta}} G_0 \simeq G$.
\end{proof}

\vskip5mm

\begin{prop}\label{P:tor10}
Let $\mathcal{K}$ be a field complete with respect to a discrete valuation $v$ and assume that the residue field $k = \mathcal{K}^{(v)}$ is infinite and finitely
generated. If $G$ is an absolutely almost simple algebraic $\mathcal{K}$-group that has good reduction at $v$, then $G$ possesses a maximal $\mathcal{K}$-torus $T$ that is $\mathcal{K}$-generic and whose minimal splitting field $\mathcal{K}_T$ is unramified over $\mathcal{K}$.
\end{prop}
\begin{proof}
Let $\mathscr{G}$ be a model of $G$ over $\mathcal{O}$. Then the reduction $\underline{\mathscr{G}}$ is an absolutely almost simple algebraic $k$-group of the same type as $G$. Since $k$ is infinite and finitely generated, one can find a maximal $k$-torus $\overline{\mathscr{T}}$ of $\underline{\mathscr{G}}$ that is generic over $k$ (cf. Theorem \ref{T:Exist}). Let $\mathscr{T}$ be a lift of $\overline{\mathscr{T}}$ to $\mathscr{G}$ (cf. \cite[Corollary B.3.5]{Conrad}). Then the generic fiber $T$ is a maximal $\mathcal{K}$-torus of $G$ that is $\mathcal{K}$-generic and whose splitting field $\mathcal{K}_T$ is unramified over $\mathcal{K}$.
\end{proof}

\vskip1mm

\begin{cor}\label{C:notA1Bl}
Let $G$ be an absolutely almost simple algebraic $\mathcal{K}$-group of type different from $\textsf{A}_1$ and $\textsf{B}_{\ell}$. Assume that the residue field $k$ is finitely generated and $G$ has good reduction at $v$. Then any $G' \in \gen_{\mathcal{K}}(G)$ has good reduction at $v$.
\end{cor}
\begin{proof}
We first consider the case where $k$ is a {\it finite}  field. In this case, it is well-known that the fact that $G$ has good reduction implies that $G$ is quasi-split over $\mathcal{K}$.
Let $G' \in \gen_{\mathcal{K}}(G)$. Then clearly
\begin{equation}\label{E:rank=}
\mathrm{rk}_{\mathcal{K}}\: G' = \mathrm{rk}_{\mathcal{K}}\: G.
\end{equation}
On the other hand, according to Tits' classification \cite{Tits},  the group $G'$ is quasi-split if and only if all vertices in the Tits index are distinguished. Since $G'$ is an inner twist of $G$, this is equivalent to (\ref{E:rank=}). Thus, $G'$ is $\mathcal{K}$-quasi-split, hence $\mathcal{K}$-isomorphic to $G$. In particular, it has good reduction at $v$.

Next, suppose that $k$ is {\it infinite} and finitely generated. In this case, our claim follows directly from the above results. Indeed, according to Proposition \ref{P:tor10}, the group $G$ contains a maximal $\mathcal{K}$-torus $T$ that is generic over $\mathcal{K}$ and whose splitting field $\mathcal{K}_T$ is unramified over $\mathcal{K}$. Since $G' \in \gen_{\mathcal{K}}(G)$, it contains a maximal $\mathcal{K}$-torus $T'$ isomorphic to $T$. Then of course $\mathcal{K}_{T'} = \mathcal{K}_{T}$ is unramified over $\mathcal{K}$. Furthermore, the assumption that $G'$ is an inner twist of $G$ implies that the subgroups $\theta_T(\mathrm{Gal}(\mathcal{K}_T/\mathcal{K})) \subset \mathrm{Aut}(\Phi(G , T))$ and $\theta_{T'}(\mathrm{Gal}(\mathcal{K}_{T'}/\mathcal{K})) \subset \mathrm{Aut}(\Phi(G , T'))$ have isomorphic images in the groups of symmetries of the corresponding Dynkin diagrams. Then the fact that $T$ is $\mathcal{K}$-generic implies the same for $T'$. 
Hence, $G'$ has good reduction by Theorem \ref{T:GR-not}.
\end{proof}

We now turn to the second case where the type of $G$ is either $\textsf{A}_1$ or $\textsf{B}_{\ell}$ $(\ell \geq 2)$. Let us first show that  Theorem \ref{T:GR-not} may be false in this case.

\vskip1mm

\addtocounter{thm}{1}

\noindent {\bf Example 6.5.} Let $\mathcal{K} = \mathbb{Q}(\!(x)\!)$, equipped with the standard valuation $v$.

\vskip1mm

 (a) Let $D$ be the quaternion algebra $\displaystyle \left( \frac{-1 , x}{\mathcal{K}}  \right)$ and $G = \mathrm{SL}_{1 , D}$. Set $\mathcal{L} =
\mathbb{Q}(\sqrt{-1})(\!(x)\!)$, and let $T = \mathrm{R}^{(1)}_{\mathcal{L}/\mathcal{K}}(\mathbb{G}_m)$ be the corresponding maximal $\mathcal{K}$-torus of $G$. Then $T$ is $\mathcal{K}$-generic and splits over the unramified extension $\mathcal{L}/\mathcal{K}$, but the quaternion algebra $D$ ramifies at $v$, hence $G$ does not have good reduction according to Example 2.2.

\vskip1mm

(b) Let $\ell \geq 2$ and $\mathcal{K}$ be as above. Set
$$
q_0 = x_1^2 + \cdots + x_{2\ell - 1}^2 + 2x_{2\ell}^2 \ \ \text{and} \ \ q = q_0 + x x_{2\ell + 1}^2,
$$
and let
$$
H = \mathrm{Spin}_{2\ell}(q_0) \ \subset \ \mathrm{Spin}_{2\ell+1}(q) = G.
$$
Since $H$ corresponds to a quadratic form defined over $\mathbb{Q}$, it has good reduction and contains a $\mathcal{K}$-generic maximal torus $T$ that splits over an extension of the form $L\mathcal{K}$ for some finite extension $L$ of $\mathbb{Q}$, which is obviously unramified. Now, $T$ is also a maximal torus in $G$ (because $G$ and $H$ have the same rank $\ell$), and we claim that it remains $\mathcal{K}$-generic in $G$. Indeed, the group $H$ is of type $\textsf{D}_{\ell}$, the group $G$ is of type $\textsf{B}_{\ell}$, hence $\vert W(G , T) \vert = 2 \cdot \vert W(H , T) \vert$. Since $T$ is generic in $H$, the image of $\mathrm{Gal}(\mathcal{K}_T/\mathcal{K})$ in $\mathrm{Aut}(\Phi(H , T))$ contains $W(H , T)$. Furthermore, the (signed) discriminant of $q_0$ is not
a square, so the image is not entirely contained in $W(H , T)$ (cf. \cite[Lemma 4.1]{PR-WC}). It follows that the image is $W(G , T)$, making $T$ generic in $G$. On the other hand, both the first and the second residues (cf. \cite[Ch. VI]{Lam}) of $q$ are nontrivial, and therefore no scalar multiple of $q$ can be equivalent to a diagonal quadratic form with all coefficients being units. So, $G$ does not have good reduction at $v$ by Example 2.3.

\vskip1mm

For types $\textsf{A}_1$ and $\textsf{B}_{\ell}$ we have the following modified condition for good reduction.
\begin{thm}\label{T:AB}
Let $\mathcal{K}$ be a field that is complete with respect to a discrete valuation $v$ with residue field $k$.
Suppose $G$ is an absolutely almost simple algebraic $\mathcal{K}$-group of type either $\textsf{A}_1$ or $\textsf{B}_{\ell}$ $(\ell \geq 2)$, and assume that $\mathrm{char}\: k \neq 2$ if $G$ is of type $\textsf{B}$.

\vskip2mm

\noindent {\rm (1)} \parbox[t]{16.2cm}{If $G$ contains two maximal $\mathcal{K}$-tori $T_1$ and $T_2$ that are $\mathcal{K}$-generic and whose splitting fields
$\mathcal{K}_{T_i}$ are unramified over $\mathcal{K}$ and satisfy $\mathcal{K}_{T_1} \cap \mathcal{K}_{T_2} = \mathcal{K}$, then $G$ has good reduction at $v$.}

\vskip1mm

\noindent {\rm (2)} \parbox[t]{16.2cm}{Conversely, if $k$ is infinite and finitely generated, and $G$ has good reduction at $v$, then $G$ contains two maximal $\mathcal{K}$-tori $T_1$ and $T_2$ that are $\mathcal{K}$-generic and whose splitting fields $\mathcal{K}_{T_1}$ and $\mathcal{K}_{T_2}$ are unramified over $\mathcal{K}$ and satisfy $\mathcal{K}_{T_1} \cap \mathcal{K}_{T_2} = \mathcal{K}$.}
\end{thm}
\begin{proof} Without loss of generality, we may assume that $G$ is simply connected.

\vskip1mm

(1): The argument for type $\textsf{A}_1$ is rather simple. Indeed, here $G = \mathrm{SL}_{1 , D}$, where $D$ is a quaternion algebra over $\mathcal{K}$, and we
need to show that $D$ is unramified at $v$. Furthermore, we have $T_i = \mathrm{R}^{(1)}_{\mathcal{L}_i/\mathcal{K}}(\mathbb{G}_m)$, where $\mathcal{L}_i = \mathcal{K}_{T_i}$ is a quadratic subfield of $D$, that, by assumption, is unramified. Now, if $D$ were ramified at $v$, then the residue algebra $\overline{D}$ would be a quadratic extension of $k$. Then
$$
\overline{D} = \overline{\mathcal{L}_1} = \overline{\mathcal{L}_2}.
$$
This would imply that $\mathcal{L}_1 = \mathcal{L}_2$, contradicting the fact that $\mathcal{L}_1 \cap \mathcal{L}_2 = \mathcal{K}$. Thus, $D$ is unramified, as required.

\vskip1mm

The argument for type $\textsf{B}_{\ell}$ is similar, but more technical. Here $G = \mathrm{Spin}_n(q)$, where $n = 2\ell + 1$ and $q$ is a nondegenerate quadratic form on $\mathcal{K}^n$. We will use the standard action of $G$ on the $n$-dimensional vector space. It is well-known that every maximal $\mathcal{K}$-torus $T$ of $G$ fixes an anisotropic vector $a \in \mathcal{K}^n$, and hence lies in the stabilizer $G(a)$, which can be identified with $H = \mathrm{Spin}_{n-1}(q')$, where $q'$ is the restriction of $q$ to the orthogonal complement $W = \langle a \rangle^{\perp}$; note that $H$ is a group of type $\textsf{D}_{\ell}$.

So, in our set-up, for each $i = 1, 2$, we can choose an anisotropic vector $a_i \in \mathcal{K}^n$ fixed by $T_i$, and let $H_i = \mathrm{Spin}_{n-1}(q_i)$, where $q_i$ is the restriction of $q$ to the orthogonal complement $W_i = \langle a_i \rangle^{\perp}$. Then $T_i$ is a maximal $\mathcal{K}$-torus in $H_i$, which is $\mathcal{K}$-generic in $H_i$ and has unramified splitting field. So, it follows from Theorem \ref{T:GR-not} for $\ell \geq 3$ and from Lemma
\ref{L:deg8} below  for $\ell = 2$ (we note that the order of the Weyl group for type $\textsf{B}_2$ is 8) that $H_i$ has good reduction at $v$. According to Example 2.3, this means that there exist an element $\lambda_i \in \mathcal{K}^{\times}$ and a basis $e^{(i)}_1, \ldots , e^{(i)}_{n-1}$ of $W_i$ such that in this basis, the quadratic form $\lambda_i q_i$ has the following presentation
$$
u^{(i)}_1 x_1^2 + \cdots + u^{(i)}_{n-1} x_{n-1}^2,
$$
where $u^{(i)}_j$ are units in $\mathcal{K}$ for $i = 1, 2$ and $j = 1, \ldots , n-1$. Also, by scaling $a_i$, we may assume that the values of $v(\lambda_i q(a_i))$ are either $0$ or $1$. If it is $0$ for at least one $i \in \{ 0 , 1\}$, then $G$ has good reduction at $v$ (see Example 2.3). So, let us assume that the value is $1$ for both $i = 1, 2$, i.e. $\lambda_iq(a_i) = \pi u^{(i)}_n$ where $u^{(i)}_n$ is a unit. Then
$$
\lambda_i q = u^{(i)}_1 x_1^2 + \cdots + u^{(i)}_{n-1} x_{n-1}^2 + \pi u^{(i)}_n x_{n}^2.
$$
We have $\lambda_1 q = (\lambda_1 \lambda_2^{-1}) (\lambda_2 q)$. So, taking determinants and evaluating $v$, we obtain
$$
1 \equiv nv(\lambda_1\lambda_2^{-1}) + 1(\mathrm{mod}\: 2),
$$
which, in view of the fact that $n$ is odd, implies that $v(\lambda_1 \lambda_2^{-1})$ is even. So, after scaling, we can actually assume that $\lambda_1 \lambda_2^{-1}$ is a unit. Consequently, setting $\lambda = \lambda_1$ and replacing $q$ by $\lambda q$, we may assume that
$$
q = u^{(i)}_1 x_1^2 + \cdots + u^{(i)}_{n-1} x_{n-1}^2 + \pi u^{(i)}_n x_n^2,
$$
where $u^{(i)}_j$ are all units. Then the diagonal quadratic forms $\langle \overline{u^{(i)}_1}, \ldots \overline{u^{(i)}_{n-1}} \rangle$ over the residue field $k$ (where the bar denotes taking the residue in $k$) are both the so-called first residues of $q$ (we note that the first and second residues were constructed by Springer assuming that the residue characteristic is $\neq 2$). Since the first-residue is well-defined (see \cite[Ch. VI]{Lam}),
we conclude that these residues are equivalent, and then by Hensel's Lemma, the quadratic forms $q_i = u^{(i)}_1 x_1^2 + \cdots + u^{(i)}_{n-1} x_{n-1}^2$ for $i = 1, 2$ themselves are equivalent.

Let $d_i$ be the (signed) discriminant of $q_i$. Since $T_i$ is generic in $G$, and the Weyl group of $G$ contains the Weyl group of $H_i$ (with respect to $T_i$) as a subgroup of index 2, we conclude that $H_i$ is an outer form of a split group over $\mathcal{K}$. Furthermore, the minimal Galois extension of $\mathcal{K}$ over which it becomes an inner form is $\mathcal{K}(\sqrt{d_i})$. Since the forms $q_1$ and $q_2$ are equivalent, we conclude that
$$
\mathcal{L} := \mathcal{K}(\sqrt{d_1}) = \mathcal{K}(\sqrt{d_2})
$$
is a quadratic extension of $\mathcal{K}$. However, $\mathcal{L} \subset \mathcal{K}_{T_i}$ for both $i = 1, 2$, contradicting the assumption that $\mathcal{K}_{T_1}$ and $\mathcal{K}_{T_2}$ are disjoint over $\mathcal{K}$.

\vskip2mm

(2): Let $\mathscr{G}$ be the model of $G$ over $\mathcal{O}$ with reduction $\underline{\mathscr{G}}$, which is an absolutely almost simple algebraic $k$-group of the same type as $G$. Since $k$ is infinite and finitely generated, we can find a maximal $k$-torus $\overline{\mathscr{T}}_1$ of $\underline{\mathscr{G}}$ that is generic over $k$. Next, let $\overline{\mathscr{T}}_2$ be a maximal $k$-torus of $\underline{\mathscr{G}}$ that is generic over the splitting field $k_{\overline{\mathscr{T}}_1}$ of $\overline{\mathscr{T}}_1$. Since the Dynkin diagrams of the types $\textsf{A}_1$ and $\textsf{B}_{\ell}$ do not have nontrivial automorphisms, the degrees $[k_{\overline{\mathscr{T}}_i} : k]$ for $i = 1, 2$ are equal to the order $w$ of the Weyl group. Besides, the degree $[k_{\overline{\mathscr{T}}_1} k_{\overline{\mathscr{T}}_2} : k_{\overline{\mathscr{T}}_1}]$ also equals $w$. This implies that
\begin{equation}\label{E:k-disj}
k_{\overline{\mathscr{T}}_1} \cap k_{\overline{\mathscr{T}}_2} = k.
\end{equation}
Let $\mathscr{T}_1$ and $\mathscr{T}_2$ be the lifts of $\overline{\mathscr{T}}_1$ and $\overline{\mathscr{T}}_2$ to $\mathscr{G}$, and let $T_1$ and $T_2$ be the corresponding generic fibers. Then $T_1$ and $T_2$ are maximal $\mathcal{K}$-tori of $G$ that are generic over $\mathcal{K}$ and whose splitting fields $K_{T_i}$ are unramified extensions of $\mathcal{K}$ with the residue fields $k_{\overline{\mathscr{T}}_i}$ for $i = 1, 2$. Then (\ref{E:k-disj}) implies that
$\mathcal{K}_{T_1} \cap \mathcal{K}_{T_2} = \mathcal{K}$, as required.
\end{proof}

\vskip.5mm

We will now prove the statement about good reduction of spinor groups of 4-dimensional quadratic forms that was used in the above argument. We recall that given a nondegenerate quadratic form $q$ over $\mathcal{K}$ of dimension four, the spinor group $\mathrm{Spin}_4(q)$ is isomorphic to $G = \mathrm{R}_{\mathcal{L}/\mathcal{K}}(H)$, where $\mathcal{L}$ is a 2-dimensional \'{e}tale $\mathcal{K}$-algebra and $H = \mathrm{SL}_{1 , \mathcal{D}}$, where $\mathcal{D}$ is a central quaternion $\mathcal{K}$-algebra.
\begin{lemma}\label{L:deg8}
Let $G$ be as above. If $G$ possesses a maximal $\mathcal{K}$-torus such that $\mathcal{K}_T/\mathcal{K}$ is an unramified extension of degree 8, then $G$ has good reduction.
\end{lemma}
\begin{proof}
If $\mathcal{L} = \mathcal{K} \times \mathcal{K}$, then $[\mathcal{K}_T : \mathcal{K}] \leq 4$ for any $\mathcal{K}$-torus $T$ of $G$. Thus, in our situation,
$\mathcal{L}/\mathcal{K}$ is a quadratic field extension. It is enough to show that $\mathcal{L}/\mathcal{K}$ is unramified and $H_{\mathcal{L}} = H \times_{\mathcal{K}} \mathcal{L}$ has good reduction. Indeed, if $\mathscr{H}$ is a reductive $\mathcal{O}_{\mathcal{L}}$-model for $H_{\mathcal{L}}$, then
$\mathscr{G} = \mathrm{R}_{\mathcal{O}_{\mathcal{L}}/\mathcal{O}_{\mathcal{K}}}(\mathscr{H})$ would be a reductive $\mathcal{O}_{\mathcal{K}}$-model for $G$. Since $\mathcal{L} \subset \mathcal{K}_T$, we immediately obtain that $\mathcal{L}/\mathcal{K}$ is unramified. Furthermore,
$$
G \times_{\mathcal{K}} \mathcal{L} \simeq H_{\mathcal{L}} \times H_{\mathcal{L}}.
$$
In terms of this $\mathcal{L}$-isomorphism, let $T \times_{\mathcal{K}} \mathcal{L} \simeq T_1 \times T_2$. We have $[\mathcal{L}_T : \mathcal{L}] = 4$, which means that $T_1$ and $T_2$ are nonisomorphic $\mathcal{L}$-tori of $H_{\mathcal{L}}$ having unramified splitting fields. So, the fact that $H_{\mathcal{L}}$ has good reduction follows from the first part of the proof of Theorem \ref{T:AB}.
\end{proof}

\begin{cor}\label{C:GR-AB}
Let $G$ be an absolutely almost simple algebraic $\mathcal{K}$-group of type either $\textsf{A}_1$ or $\textsf{B}_{\ell}$ that has good reduction at $v$. Assume that $k$ is finitely generated and $\mathrm{char}\: k \neq 2$ if $G$ is of type $\textsf{B}$. Then any  $G' \in \gen_{\mathcal{K}}(G)$ also has good reduction at $v$.
\end{cor}
\begin{proof}
As in the proof of Corollary \ref{C:notA1Bl}, we consider two cases. First, if the residue field is finite, the fact that an absolutely almost $\mathcal{K}$-group $G$ of type $\textsf{A}_1$ or $\textsf{B}_{\ell}$ has good reduction means that it actually splits over $\mathcal{K}$. Then $G'$ also splits, and hence has good reduction. Next, suppose the residue field is infinite and finitely generated. Then
by Theorem \ref{T:AB}(2), the group $G$ possesses two maximal $\mathcal{K}$-tori $T_1$ and $T_2$ that are generic over $\mathcal{K}$ and whose splitting fields are unramified over $\mathcal{K}$ and satisfy $\mathcal{K}_{T_1} \cap \mathcal{K}_{T_2} = \mathcal{K}$. On the other hand, $G'$ contains maximal $\mathcal{K}$-tori $T'_1$ and $T'_2$ that are $\mathcal{K}$-isomorphic to $T_1$ and $T_2$, respectively. Clearly, $T'_1$ and $T'_2$ have properties analogous to those of $T_1$ and $T_2$, so $G'$ has good reduction by Theorem \ref{T:AB}(1).
\end{proof}
\vskip5mm

Now, Theorem \ref{T:GR-all}  follows from Corollaries \ref{C:notA1Bl} and \ref{C:GR-AB}. Furthermore, to prove the first assertion of Theorem \ref{T:GoodReduction1}, one needs to use Corollary \ref{C:gen}  and then apply Theorem \ref{T:GR-all}. To prove the second assertion, we let $\overline{\mathscr{T}}$ be a maximal $k^{(v)}$-torus of the reduction $\underline{G}^{(v)}$. Let $\mathscr{G}$ be the model of $G \times_k k_v$ over $\mathcal{O}_v$ that yields the reduction $\underline{G}^{(v)}$. Then according to \cite[Corollary B.3.5]{Conrad}, the torus $\overline{\mathscr{T}}$ lifts to a maximal torus $\mathscr{T}$ of $\mathscr{G}$; let $T$ be the corresponding generic fiber, which is a maximal $k_v$-torus of $G \times_k k_v$. As we already mentioned, $G' \in \gen_{k_v}(G)$, so there exists a maximal $k_v$-torus $T'$ of $G'$ that is isomorphic to $T$. We have already established that $G' \times_k k_v$ has a model $\mathscr{G}'$ over $\mathcal{O}_v$, and using Theorem \ref{T:unram-split-field} we may assume that $T'$ is the generic fiber of a torus $\mathscr{T}'$ of $\mathscr{G}'$.
According to Proposition~\ref{P:UnMod}, the fact that $T \simeq T'$ over $k_v$ implies that $\mathscr{T} \simeq \mathscr{T}'$ over $\mathcal{O}_v$. Then the reduction $\overline{\mathscr{T}'}$ is a maximal $k_v$-torus of the reduction $(\underline{G}')^{(v)}$ that is isomorphic to $\overline{T}$. A symmetric argument shows that every maximal $k_v$-torus of $(\underline{G}')^{(v)}$ is isomorphic to a maximal $k_v$-torus of $\underline{G}^{(v)}$. Thus, $(\underline{G}')^{(v)} \in \gen_{k_v}(\underline{G}^{(v)})$.

\vskip2mm

\begin{cor}
Let $G$ be an absolutely almost simple algebraic group over an infinite finitely generated field $k$, and let $V$ be a divisorial set of places of $k$. Assume that $\mathrm{char}\: k \neq 2$ if $G$ is of type $\textsf{B}$. Then there exists a finite subset $S \subset V$ such that every $G' \in \gen_k(G)$ has good reduction at all $v \in V \setminus S$.
\end{cor}
\begin{proof}
First, we note that the residue field $k^{(v)}$ is finitely generated for all $v \in V$, so we can apply our previous results. Clearly, we can find a finite subset $S \subset V$ such that $G$ has good reduction at all $v \in V \setminus S$. Besides, if $G$ is of type $\textsf{B}$, then by our assumption $\mathrm{char}\: k \neq 2$ and we can include in $S$ all $v \in V$ such that $v(2) \neq 0$. But then according to Theorem \ref{T:GoodReduction1}, every $G' \in \gen_k(G)$ also has good reduction at all $v \in V$.
\end{proof}

This corollary shows that the truth of the Finiteness Conjecture for forms $G$ with good reduction and all divisorial sets of places of the given finitely generated field $k$ would imply the finiteness of $\gen_k(G)$.

\section{The behavior of the genus under a purely transcendental base change: proof of Theorem \ref{T:transc}}\label{S:PureTr}

In order to set the stage for the proof of Theorem \ref{T:transc}, we would first like to present an analogous result for division algebras.

\vskip1mm

\noindent {\bf \ref{S:PureTr}.1. The genus of a division algebra.} We recall that two finite-dimensional central division algebras $D_1$ and $D_2$ over a field $K$ are said to have the {\it same maximal subfields} if they have the same degree $n$ and satisfy the following property: a degree $n$ extension $P/K$ admits a $K$-embedding $P \hookrightarrow D_1$ if and only if it admits a $K$-embedding $P \hookrightarrow D_2$. Given a finite-dimensional central division algebra $D$ over $K$, one defines its genus $\gen(D)$ as the set of classes $[D'] \in \Br(K)$ in the Brauer group corresponding to central division $K$-algebras $D'$ that have the same maximal subfields as $D$. This concept has been analyzed in detail in \cite{CRR-Bull}, \cite{CRR3}, \cite{CRR-Isr}, \cite{RR-manuscr}, and other publications. Our goal in the present subsection is to prove the following.

\begin{prop}\label{P:k(x)}
Let $D$ be a central division algebra of degree $n$ over a field $k$, and assume that $n$ is prime to $\mathrm{char}\: k$. Set $K = k(x)$. Then
every element of $\gen(D \otimes_k K)$ is represented by a division algebra of the form $D' \otimes_k K$, where $D'$ is a central division algebra over $k$ with
$[D'] \in \gen(D)$.
\end{prop}

\begin{proof}
For any $n$ that is prime to $\mathrm{char}\: k$, we have the following exact sequence that goes back to Faddeev (cf. \cite[Example 9.2]{GMS}):
$$
0 \to {}_n\Br(k) \longrightarrow {}_n\Br(K) \stackrel{\rho}{\longrightarrow} \bigoplus_p \mathrm{Hom}(\mathrm{Gal}((K^{(p)})^{\mathrm{sep}}/K^{(p)}) ,
\mathbb{Z}/n\mathbb{Z}),
$$
where $p$ runs through all monic irreducible polynomials in $k[x]$, $\rho$ is the direct sum of the corresponding residue maps
$$
\rho^{(p)} \colon {}_n\mathrm{Br}(K) \longrightarrow \mathrm{Hom}(\mathrm{Gal}((K^{(p)})^{\mathrm{sep}}/K^{(p)}) ,
\mathbb{Z}/n\mathbb{Z}),
$$
and $K^{(p)} = k[x]/(p(x))$ is the residue field at $p$. Let $\Delta \in \gen(D \otimes_k K)$. Since the algebra $D \otimes_k K$ is unramified at all $p$, the latter implies that $\Delta$ is also unramified at all $p$ (cf. \cite{CRR-Bull}, \cite{RR-manuscr}), i.e. $\rho([\Delta]) = 0$. So, it follows from the above exact sequence that $\Delta$ is of the form $\Delta = D' \otimes_k K$ for some central division $k$-algebra $D'$ of degree $n$. It remains to show that $D' \in \gen(D)$. For this, we let $v$ denote the discrete valuation of $K$ corresponding to the polynomial $p(x) = x$; then the completion $K_v$ is $k(\!(x)\!)$ and the residue field $K^{(v)}$ is $k$. Since $D' \otimes_k K \in \gen(D\otimes_k K)$, it follows from \cite[Lemma 2.1]{RR-manuscr} that the degree $n$ division algebras $\mathcal{D} = D \otimes_k K_v$ and $\mathcal{D}' = D' \otimes_k K_v$ have the same maximal subfields. Since $\mathcal{D}$ and $\mathcal{D}'$ are division algebras, the valuation $v$ extends to valuations $w$ and $w'$ of these algebras, and we let $\overline{\mathcal{D}}$ and $\overline{\mathcal{D}'}$ denote the corresponding residue algebras. A standard argument shows that the fact that $\mathcal{D}$ and $\mathcal{D}'$ have the same maximal subfields implies that the residue algebras also have the same maximal subfields. Since
$\overline{\mathcal{D}} \simeq D$ and $\overline{\mathcal{D}'} \simeq D'$, we see that $D' \in \gen(D)$, as required.
\end{proof}

Using the proposition repeatedly, we obtain a similar statement for the field of rational functions $K = k(x_1, \ldots , x_m)$ in any number of variables. Our next goal is to prove Theorem \ref{T:transc} that extends the proposition to absolutely almost simple algebraic groups.

\vskip2mm

\noindent {\bf \ref{S:PureTr}.2. Proof of Theorem \ref{T:transc}.}
The argument relies on the following fundamental fact.

\begin{thm}\label{T:RagRam}
{\rm (Raghunathan, Ramanathan \cite{RagRam})} Let $G$ be a connected reductive algebraic group over a field $k$, and let $\mathbb{A}^1_k = \mathrm{Spec}\: k[x]$ be the affine line over $k$. Let  $B \to \mathbb{A}^1_k$  be a principal $G$-bundle on $\mathbb{A}^1_k$ such that the bundle $B \times_{\mathbb{A}^1_k} \mathbb{A}^1_{k^{\mathrm{sep}}}$ on $\mathbb{A}^1_{k^{\mathrm{sep}}} = \mathrm{Spec}\: k^{\mathrm{sep}}[x]$, where $k^{\mathrm{sep}}$ is a separable closure of $k$, is trivial. Then $B$ is constant, i.e. there exists a principal $G$-bundle $B_0 \to \mathrm{Spec}\: k$ such that $B = B_0 \times_{\mathrm{Spec}\: k} \mathbb{A}^1_k$.
\end{thm}

An alternative proof of this theorem was given in \cite{GilleA1}. Later in \cite{CGP}, the theorem was extended to reductive, but not necessarily connected, groups; see also \cite{ARR} for a new proof over fields of characteristic zero that uses buildings.  In cohomological language, the theorem means that the natural map from $H^1(k , G)$ to $\mathrm{Ker}\left(\he(\mathbb{A}^1_k , G) \to \he(\mathbb{A}^1_{k^{\mathrm{sep}}} , G)\right)$ is a surjection (in fact, a bijection). We will use this interpretation to prove the following.
\begin{prop}\label{P:A1}
Let $G$ be a semisimple algebraic group over a field $k$, let $K = k(x)$, and let $V$ be the set of discrete valuations of $K$ associated with the monic irreducible polynomials $p(x) \in k[x]$. If $H$  is an inner $K$-form of $G \times_k K$ that has good reduction at all $v \in V$ and satisfies $H \times_K k^{\mathrm{sep}}(x) \simeq G \times_k k^{\mathrm{sep}}(x)$, then $H \simeq H_0 \times_k K$ for some $k$-form $H_0$ of $G$.
\end{prop}
\begin{proof}
The $k$-group $G$ is an inner twist of a quasi-split $k$-group $G_0$. In terms of proving the proposition, we can replace $G$ by $G_0$, and hence assume
that $G$ itself is quasi-split. Let $\overline{G}$ be the corresponding adjoint group. Since $G$ is quasi-split, there exists a $k$-defined finite subgroup $\Sigma \subset \mathrm{Aut}(G)$ such that $\mathrm{Aut}(G) = \overline{G} \rtimes \Sigma$. Let $v \in V$, and denote by $\mathcal{O}_v$ the valuation ring of the completion $K_v$. Set $\mathscr{G} = G \times_k \mathcal{O}_v$ and $\overline{\mathscr{G}} = \overline{G} \times_k \mathcal{O}_v$. We say that a cohomology class in $H^1(K , \overline{G})$ is {\it unramified} at $v \in V$ if its image  under the restriction map $H^1(K , \overline{G}) \to H^1(K_v , \overline{G})$ belongs to $\mathrm{Im}(\he^1(\mathcal{O}_v , \overline{\mathscr{G}}) \to H^1(K_v , \overline{G}))$.

Suppose now that $H$ is an inner twist of $G \times_k K$ that has good reduction at all $v \in V$. Fix  a cocycle $\xi \in Z^1(K , \overline{G})$ such that $H = {}_{\xi}(G \times_k K).$ We will first show that
\begin{equation}\label{E:etale}
[\xi] \in \mathrm{Im}\: \beta_k, \ \ \text{where} \ \  \beta_k \colon \he^1(\mathbb{A}^1_k , \overline{G}) \to H^1(K  , \overline{G})
\end{equation}
is the map induced by passage to the generic point.
According to Proposition \ref{P:Harder}, it is enough to show that $[\xi]$ is unramified at all $v \in V$. So, fix $v \in V$. By our assumption, there exists a reductive group $\mathcal{O}_v$-scheme $\mathscr{H}$ with generic fiber $H$. This scheme is necessarily an inner form of $\mathscr{G}$, so we can find a cocycle $\xi' \in Z_{\mathrm{\acute{e}t}}^1(\mathcal{O}_v , \overline{\mathscr{G}})$ such that $\mathscr{H} = {}_{\xi'} \mathscr{G}$. Passing to the generic point, we obtain $H \times_K K_v \simeq
{}_{\xi'} (G \times_k K_v)$. This means that for the image $\xi_v$ of $\xi$ under the restriction map $Z^1(K , \overline{G}) \to Z^1(K_v , \overline{G})$, the cohomology classes $[\xi_v]$ and $[\xi']$ have the same image in $H^1(K_v , \mathrm{Aut}(G))$. Thus, there exists $g \in \mathrm{Aut}(G)(K_v^{\mathrm{sep}})$ such that $\xi_v(\sigma) = g \xi'(\sigma) \sigma(g)^{-1}$ for all $\sigma \in \mathrm{Gal}(K_v^{\mathrm{sep}}/K_v)$. We can write $g = hs$ with $h \in \overline{G}(K_v^{\mathrm{sep}})$ and $s \in \Sigma(K_v^{\mathrm{sep}}) = \Sigma(k^{\mathrm{sep}})$, and then  define $\xi'' \in Z^1(K_v , \overline{G})$ by
$\xi''(\sigma) := s \xi'(\sigma) \sigma(s)^{-1}$. Clearly, $[\xi'']
= [\xi_v]$ in $H^1(K_v , \overline{G})$, and by construction $[\xi'']$ lies in the image of $\he^1(\mathcal{O}_v , \overline{\mathscr{G}}) \to H^1(K_v , \overline{G})$. Thus, $[\xi]$ is unramified at $v$.

Now, using (\ref{E:etale}), pick $[\zeta] \in \he^1(\mathbb{A}^1_k , \overline{G})$ such that $[\xi] = \beta_k([\zeta])$. To prove the proposition, it is enough to show that $[\xi]$ is the image of some $[\xi_0] \in H^1(k , \overline{G})$, as then one can take $H_0 = {}_{\xi_0} G$. As we discussed above, this would follow from Theorem \ref{T:RagRam} if we could show that $[\zeta] \in \mathrm{Ker}\left(\he^1(\mathbb{A}^1_k , \overline{G}) \to \he^1(\mathbb{A}^1_{k^{\mathrm{sep}}} , \overline{G}) \right)$. We have the following commutative diagram
$$
\begin{CD}
\he^1(\mathbb{A}^1_k , \overline{G}) @ > \beta_k >>  H^1(K , \overline{G}) \\
@ V \gamma_1 VV @ VV \gamma_2 V \\
\he^1(\mathbb{A}^1_{k^{\mathrm{sep}}} , \overline{G}) @ > \beta_{k^{\mathrm{sep}}} >> H^1(k^{\mathrm{sep}}(x) , \overline{G})
\end{CD}.
$$
The fact that $H \times_K k^{\mathrm{sep}}(x) \simeq G \times_K k^{\mathrm{sep}}(x)$ means that $[\xi] \in \mathrm{Ker}(H^1(K , \overline{G}) \to H^1(k^{\mathrm{sep}}(x) , \mathrm{Aut}(\overline{G}))$. Since $\overline{G}$ splits over $k^{\mathrm{sep}}$, the map $H^1(k^{\mathrm{sep}}(x) , \overline{G}) \to H^1(k^{\mathrm{sep}}(x) , \mathrm{Aut}(\overline{G}))$ has trivial kernel, and therfore we conclude that actually  $[\xi] \in \mathrm{Ker}\: \gamma_2$. So, it follows from the diagram that in order to show that $[\zeta] \in \mathrm{Ker}\: \gamma_1$, it suffices to prove that $\mathrm{Ker}\: \beta_{k^{\mathrm{sep}}}$ is trivial. According to \cite{Nisn2}, there is a bijection between $\mathrm{Ker}\: \beta_{k^{\mathrm{sep}}}$ and the  double coset space $$\mathrm{cl}(\overline{G}, k^{\mathrm{sep}}(x), V^s) := \overline{G}(\mathbf{A}^{\infty}(V^s)) \backslash \overline{G}(\mathbf{A}(V^s)) / \overline{G}(k^{\mathrm{sep}}(x))$$ where $V^s$ is the set of discrete valuations of $k^{\mathrm{sep}}(x)$ associated with the closed points of $\mathbb{A}^1_{k^{\mathrm{sep}}}$, with $\overline{G}(\mathbf{A}(V^s))$   and $\overline{G}(\mathbf{A}^{\infty}(V^s))$ denoting the group of rational ad\`eles of $\overline{G}$ associated with $V^s$ and its subgroup of integral ad\`eles (cf. \cite[\S 4]{CRR-Isr}). Fix a maximal $k$-defined torus $T$ of $\overline{G}$. Then $T$ splits over $k^{\mathrm{sep}}$ and a standard argument using strong approximation for the opposite maximal unipotent subgroups associated with $T$ (cf. {\it loc. cit.}) shows that every double coset in $\mathrm{cl}(\overline{G}, k^{\mathrm{sep}}(x), V^s)$ has a representative in $T(\mathbf{A}(V^s))$. On the other hand, for the multiplicative group $S = \mathbb{G}_m$, the double coset space $\mathrm{cl}(T, k^{\mathrm{sep}}(x), V^s)$ can be identified with the Picard group of $\mathbb{A}^1_{k^{\mathrm{sep}}}$, which is trivial. Since $T$ is $k^{\mathrm{sep}}$-split, we obtain that $\mathrm{cl}(T, k^{\mathrm{sep}}(x), V^s)$ reduces to a single element. Thus, $\mathrm{cl}(\overline{G}, k^{\mathrm{sep}}(x), V^s)$ also reduces to a single element, and the injectivity of $\beta_{k^{\mathrm{sep}}}$ follows. This completes the proof of the proposition.
\end{proof}

\vskip.1mm

It is now easy to complete the proof of Theorem \ref{T:transc}. Let $H \in \gen_K(G \times_k K)$, where $K = k(x)$.  Since $G \times_k K$ has good reduction at all $v \in V$, we see from Theorem \ref{T:GoodReduction1} that the same is true for $H$. Let $T$ be any maximal $k$-torus of $G$. Then $H$ has a maximal $K$-torus isomorphic to $T \times_k K$, which splits over $k^{\mathrm{sep}}(x)$. Thus, both $G \times_k K$ and $H$ split over $k^{\mathrm{sep}}(x)$, hence
$G \times_k k^{\mathrm{sep}}(x) \simeq H \times_K k^{\mathrm{sep}}(x)$. Since $H$ is an inner twist of $G \times_k K$, we can apply Proposition \ref{P:A1} to conclude that $H = H_0 \times_k K$ for some inner $k$-form $H_0$ of $G$. Let $v$ be the valuation of $K$ associated with $x$. Then the reductions of $G \times_k K$ and $H$ at $v$ coincide with $G$ and $H_0$, respectively. Consequently, Theorem \ref{T:GoodReduction1} yields $H_0 \in \gen_k(G)$. (In fact, applying this argument to all valuations, we see that $H_0 \times_k \ell \in \gen_{\ell}(G \times_k \ell)$ for {\it every} finite simple extension $\ell/k$.)

\section{Killing the genus by a purely transcendental extension}\label{S:killing}

\noindent As in the previous section, we will first explain the phenomenon of ``killing the genus" in the case of division algebras.

\vskip1mm

\noindent {\bf \ref{S:killing}.1. Killing the genus of a division algebra.} It turns out that Proposition \ref{P:k(x)} can be significantly strengthened as follows.
\begin{thm}\label{T:GenDA}
Let $D$ be a central division algebra of degree $n$ over a field $k$, and assume that $n$ is prime to $\mathrm{char}\: k$. Set $K = k(x_1, \ldots , x_{n-1})$.
Then $\gen(D \otimes_k K)$ consists of (the Brauer classes of) central division $K$-algebras of the form $D' \otimes_k K$, where $D'$ is a central division $k$-algebra of degree $n$ such that the classes $[D]$ and $[D']$ generate the same subgroup of $\Br(k)$.
\end{thm}

We already know from Proposition \ref{P:k(x)} and the subsequent remark that every element of $\gen(D \otimes_k K)$ is represented by a division algebra of the form $D' \otimes_k K$ for some central division algebra $D'$ over $k$ of degree $n$. In order to show that the classes $[D]$ and $[D']$ generate the same subgroup of $\Br(k)$, we will eventually use Amitsur's Theorem \cite{Amit}. However, its application requires some preparation. We refer the reader to \cite[Ch. 13]{Salt2} for basic facts about Severi-Brauer varieties and their function fields.
\begin{lemma}\label{L:SB1}
Let $D$ be a central division algebra of degree $n$ over a field $k$, and let $F_D$ be the function field of the corresponding Severi-Brauer variety $SB(D)$. Then there exist elements $x_1, \ldots , x_{n-1} \in F_D$ that are algebraically independent over $k$ and such that $F_D/k(x_1, \ldots , x_{n-1})$ is an extension of degree $n$. \end{lemma}
\begin{proof}
Let $W \subset D$ be a $k$-subspace of dimension $m$. If we fix a $k$-basis $w_1, \ldots , w_m \in W$, then there exists a homogeneous polynomial $\nu_W \in k[t_1, \ldots , t_m]$ such that
$$
\nu_W(\alpha_1, \ldots , \alpha_m) = \mathrm{Nrd}_{D/K}(\alpha_1 w_1 + \cdots + \alpha_m w_m) \ \ \text{for all} \ \ \alpha_1, \ldots , \alpha_m \in k.
$$
Set $Z_W$ to be the subvariety of the projective space $\mathbb{P}(W)$ defined by the equation $\nu_W = 0$. It was shown by E.~Matzri \cite{Matzri} that for a Zariski-dense set of subspaces $W$ in the Grassmannian $\mathrm{Gr}(n+1 , D)$, the variety $Z_W$ is absolutely irreducible and birationally $k$-isomorphic to $SB(D)$. For the purpose of proving our lemma, we pick one such $(n+1)$-dimensional subspace $W \subset D$ and fix a basis $w_1, \ldots , w_{n+1}$. Pick two distinct indices $i,j \in \{1, \ldots , n+1 \}$, set
$$
p(T) = \nu_W(t_1, \ldots , t_{i-1}, T, t_{i+1}, \ldots , t_{j-1}, 1, t_{j+1}, \ldots , t_{n+1}),
$$
and then re-denote the variables $t_1, \ldots , t_{i-1}, t_{i+1}, \ldots , t_{j-1}, t_{j+1}, \ldots , t_{n+1}$ as $x_1, \ldots , x_{n-1}$. Then $F_D$ is isomorphic to the extension of $k(x_1, \ldots , x_{n-1})$ obtained by adjoining a root of $p(T)$. On the other hand, $p(T)$ is irreducible over $k(x_1, \ldots , x_{n-1})$ and its leading term is $\mathrm{Nrd}_{D/K}(w_i)T^n$, demonstrating that $\deg p = n$ and completing the argument.
\end{proof}

\vskip2mm
\noindent {\it Proof of Theorem \ref{T:GenDA}.} As in the lemma, we denote by $F_D$ the function field of the Severi-Brauer variety $SB(D)$, and pick algebraically independent elements $x_1, \ldots , x_{n-1} \in F_D$ so that $F_D$ is a degree $n$ extension of $K = k(x_1, \ldots , x_{n-1})$.  According to Amitsur's theorem \cite{Amit}, the kernel of the base change map
$\Br(k) \to \Br(F_D)$ coincides with the cyclic subgroup $\langle [D] \rangle \subset \Br(k)$. In particular,
$$
D \otimes_k F_D \simeq M_n(F_D) \simeq (D \otimes_k K) \otimes_K F_D.
$$
Since $[F_D : K] = n$, the latter means that $F_D$ is $K$-isomorphic to a maximal subfield of $D \otimes_k K$. By our assumption, $D \otimes_k K$ and $D'
\otimes_k K$ have the same maximal subfields, so $F_D$ admits a $K$-embedding as a maximal subfield into $D' \otimes_k K$. It follows that
$$
D' \otimes_k F_D \simeq (D' \otimes_k K) \otimes_K F_D \simeq M_n(F_D).
$$
Then  Amitsur's theorem yields the inclusion $[D'] \in \langle [D] \rangle$. A symmetric argument shows that $[D] \in \langle [D'] \rangle$, which completes
the argument. \hfill $\Box$

\vskip3mm

Thus, no matter what the genus $\gen(D)$ is originally, after a suitable purely transcendental base change $K/k$, the genus $\gen(D \otimes_k K)$ becomes finite, and in fact minimal possible. We call this phenomenon ``killing the genus by a purely transcendental extension." Later in this section, we will prove Theorems \ref{T:Kill1} and \ref{T:Kill2} that reveal a similar phenomenon for the norm one groups $\mathrm{SL}_{1 , A}$ of central simple algebras $A$ and groups of type $\textsf{G}_2$, after which we will discuss possible generalizations. But first, we would like to continue our discussion of this phenomenon in the context of division algebras. As an immediate consequence of Theorem \ref{T:GenDA}, we have
\begin{cor}\label{C:Quat}
Let $D$ be a quaternion division algebra over a field $k$ of characteristic $\neq 2$, and let $K = k(x)$. Then $$\gen(D \otimes_k K) = \{ [D \otimes_k K] \}.$$
\end{cor}

The above proof of Theorem \ref{T:GenDA} for quaternions yields the following statement:

\vskip2mm

\noindent $(\bullet)$ \parbox[t]{16cm}{\it Let $D_1$ and $D_2$ be two central quaternion division algebras over a field $k$ of characteristic $\neq 2$, and
let $K = k(x)$. If $D_1 \otimes_k K$ and $D_2 \otimes_k K$ are in the same genus, then $D_1 \simeq D_2$ over $k$.}

\vskip2mm

\noindent It turns out that $(\bullet)$ remains valid if the field of rational functions $K = k(x)$  is replaced by the function field of any
absolutely irreducible curve over $k$ having a $k$-rational point.
\begin{prop}\label{P:curve}
Let $D_1$ and $D_2$ be two central quaternion division algebras over a field $k$ of characteristic $\neq 2$, and let $C$ be a smooth geometrically integral curve over $k$ with $C(k) \neq \emptyset$. If for the function field $K = k(C)$, the algebras $D_1 \otimes_k K$ and $D_2 \otimes_k K$ are in the same genus (as $K$-algebras), then $D_1 \simeq D_2$ over $k$.
\end{prop}
(We note that since $C(k) \neq \emptyset$, the algebras $D_1 \otimes_k K$ and $D_2 \otimes_k K$ are division algebras over $K$. Indeed, let $P \in C(k)$.
Since $P$ is nonsingular, we can consider the corresponding valuation $v$ of $K$, and then the completion $K_v$ can be identified with the field $k(\!(t)\!)$ of
formal Laurent series. Then the algebras $D_1 \otimes_k K_v$ and $D_2 \otimes_k K_v$ are obviously division algebras, so the algebras $D_1 \otimes_k K$ and $D_2 \otimes_k K$ are also division algebras.)
\begin{proof}
Without loss of generality, we may assume that $C$ is projective. Fix a rational point $P \in C(k)$, and consider the divisor $\Delta_n = nP$ on $C$ for $n > 0$. It follows from the Riemann-Roch Theorem (see, for example, \cite{Serre-AG} for the statement and  relevant notations) that the dimension $\ell(\Delta_n) = \dim_k L(\Delta_n)$ of the space $L(\Delta_n)$ associated with $\Delta_n$ for $n \gg 0$ is given by
$$
\ell(\Delta_n) = n + 1 - g,
$$
where $g$ is the genus of $C$. Thus, we can find an {\it odd} $n \geq 1$ prime to $\mathrm{char}\: k$  such that there exists $f \in L(\Delta_n) \setminus L(\Delta_{n-1})$. Then the divisor of poles of the principal divisor $(f)$ is $\Delta_n$, hence has degree precisely $n$. Thinking of $f$ as a morphism $C \to \mathbb{P}^1_k$, we conclude that the degree of this map is $n$, which means that $K$ is a degree $n$ extension of the field of rational functions $k(x)$. On the other hand, the function field $F_{D_1}$ of the Severi-Brauer variety of $D_1$ can be viewed as a quadratic extension of $k(x)$. (More precisely, $K$ and $F_{D_1}$ can be embedded into an algebraic closure of the field $k(x)$ so that the images of these embeddings, for which we keep the same notations, have degrees $n$ and $2$ over $k(x)$, respectively.) We have
$$
D_1 \otimes_k F_{D_1} \simeq M_2(F_{D_1}) \simeq (D_1 \otimes_k k(x)) \otimes_{k(x)} F_{D_1},
$$
implying that $F_{D_1}$ admits a $k(x)$-embedding into $D_1 \otimes_k k(x)$ as a maximal subfield (just as in the proof of Theorem \ref{T:GenDA}). Then the composition $F_{D_1}K \simeq F_{D_1} \otimes_{k(x)} K$ admits a $K$-embedding into $(D_1 \otimes_{k} k(x)) \otimes_{k(x)} K \simeq D_1 \otimes_k K$ as a maximal subfield. By our assumption, the algebras $D_1 \otimes_k K$ and $D_2 \otimes_k K$ are in the same genus, so there is a $K$-embedding $F_{D_1}K \hookrightarrow D_2 \otimes_{k(x)} K$. It follows that
$$
(D_2 \otimes_k F_{D_1}) \otimes_{F_{D_1}} F_{D_1}K \simeq  D_2 \otimes_k F_{D_1}K  \simeq (D_2 \otimes_k K) \otimes_K F_{D_1}K \simeq M_2(F_{D_1}K).
$$
Thus, the degree $n$ extension $F_{D_1}K/F_{D_1}$ splits the algebra $D_2 \otimes_k F_{D_1}$. Since $n$ is odd, we conclude that $D_2 \otimes_k F_{D_1} \simeq
M_2(F_{D_1})$. By Amitsur's theorem, this means that the quaternion division algebras $D_1$ and $D_2$ are isomorphic.
\end{proof}

\vskip2mm

\addtocounter{thm}{2}

\noindent {\bf Remark 8.5.} The assumption in Proposition \ref{P:curve} that $C$ has a $k$-rational point cannot be omitted. Indeed, let $D_1$ and $D_2$ be two nonisomorphic quaternion division algebras having a common subfield (e.g., one can take $D_1 = (-1 , 3)$ and $D_2 = (-1 , 7)$ over $k = \Q$). Then $D_1 \otimes_K D_2 \simeq M_2(D)$ for the quaternion division algebra $D = (-1 , 21)$. Let $C$ be the Severi-Brauer variety for $D$ (which is a conic without a rational point), and $K = k(C)$. Since $K$ splits $D$, the $K$-algebras $D_1 \otimes_k K$ and $D_2 \otimes_k K$ are isomorphic, hence belong to the same genus. However, by construction, $D_1$ and $D_2$ are not isomorphic as $k$-algebras.

\vskip2mm

The result established in Corollary \ref{C:Quat} prompts the following

\vskip2mm

\noindent {\bf Question 8.6.} {\it Does there exist a central quaternion division algebra $D$ over the field of rational functions $K = k(x)$ over some
field $k$ having \emph{nontrivial} genus $\gen(D)$?}

\vskip2mm

\noindent {\bf 8.2. Proof of Theorem \ref{T:Kill1}.} While the argument involves some of the same considerations as the proof of Theorem \ref{T:GenDA}, it also contains several new elements. In order to apply Amitsur's theorem, we need to match certain maximal subfields of the algebras obtained by a purely transcendental base change, and not just the corresponding maximal tori in the associated norm one groups. It is well-known that given a central division algebra $D$ of degree $n$ over a field $k$, every maximal $k$-torus $T$ of $G = \mathrm{SL}_{1 , D}$ is the norm  torus $\mathrm{R}^{(1)}_{F/k}(\mathbb{G}_m)$ for some maximal separable subfield $F \subset D$. The problem is that in general, given two separable degree $n$ extensions $F_1$ and $F_2$ of $K$, the fact that the corresponding norm tori are $K$-isomorphic, may {\it not} imply that the extensions are isomorphic.\footnote{To construct such an example, it is enough to find a finite group $G$ having two nonconjugate subgroups $H_1$ and $H_2$ such that the permutation lattices
$\Z[G/H_1]$ and $\Z[G/H_2]$ are isomorphic as $\Z[G]$-modules. We learned from correspondence with R.~Guralnick and D.~Saltman that the following example of this situation was found by L.~Scott \cite{Scott}: $G = \mathrm{PSL}_2(F_{29})$ and $H_1 , H_2$ are nonconjugate
subgroups isomorphic to $A_5$.} However, as the following lemma shows, this complication does not arise in the case of generic tori.
We recall that a field extension $F/k$ of degree $n$ is called {\it generic} if it is separable and for its normal closure $\tilde{F}$, the Galois group $\Ga(\tilde{F}/k)$ is isomorphic to the symmetric group $S_n$.
\begin{lemma}\label{L:Iso}
Let $F_1$ and $F_2$ be two degree $n$ extensions of a field $k$, and let $T_i = \mathrm{R}^{(1)}_{F_1/k}(\mathbb{G}_m)$ $(i = 1, 2)$ be the corresponding norm  tori. If at least one of the extensions $F_i$ is generic over $k$ and $T_1 \simeq T_2$ as $k$-tori, then $F_1 \simeq F_2$ over $k$.
\end{lemma}
\begin{proof}
It is well-known that the minimal splitting field of $T_i$ is the normal closure of $F_i$ over $K$, which we will denote by $\tilde{F}_i$. Since $T_1$ and $T_2$ are $k$-isomorphic, we have $\tilde{F}_1 = \tilde{F}_2 =: \tilde{F}$, and then by our assumption, the Galois group $G = \Ga(\tilde{F}/K)$ is isomorphic to $S_n$. Let $H_i = \mathrm{Gal}(\tilde{F}/F_i)$. To prove that $F_1 \simeq F_2$, it is enough to show that the subgroups $H_1$ and $H_2$ are conjugate in $G$. When $n \neq 6$, this follows from the elementary fact in group theory that, in this case, $S_n$ has only one conjugacy class of subgroups of index $n$; in other words, every subgroup of index $n$ is the stabilizer of some point --- see \cite[Kapitel II, Satz 5.5]{Huppert}. 

For $n = 6$, it is well-known that the group of outer automorphisms of $G$ has order 2. Furthermore, if $\sigma$ is an outer automorphism of $G$ and $H \subset G$ is a fixed subgroup of index $n$, then it follows from \cite[Kapitel II, Satz 5.5]{Huppert}
that any subgroup of index $n$ is conjugate to either $H$ or $\sigma(H)$. To prove that $H_1$ and $H_2$ are conjugate in this case as well, we observe that  a $k$-isomorphism between $T_1$ and $T_2$  yields an isomorphism of the character groups $X(T_1) \simeq X(T_2)$ as $\Z[G]$-modules, hence an isomorphism of the vector spaces $W_1 = X(T_1) \otimes_{\Z} \Q$ and $W_2 = X(T_2) \otimes_{\Z} \Q$ as $\Q[G]$-modules. The representation of $G$ afforded by $W_i$ can be described as the permutation representation on $\Q[G/H_i]$ ``minus" the trivial representation (we will refer to this representation as the {\it standard representation} associated with the subgroup $H_i$). So, to complete the argument, it remains to show that the standard representations $\rho \colon G \to \mathrm{GL}(W)$ and $\rho' \colon G \to \mathrm{GL}(W')$ associated with $H$ and $\sigma(H)$ are {\it not} equivalent. Since $\rho' = \rho \circ \sigma$, this follows from the explicit description of the character of the standard representation and the fact that $\sigma$ switches the conjugacy classes of 3-cycles and the products of two disjoint 3-cycles.
\end{proof}

Next, we need the following strengthening of Lemma \ref{L:SB1} that includes the genericity condition.
\begin{prop}\label{P:SB2}
Let $A$ be a central simple algebra of degree $n$ over a finitely generated field $k$, and let $F_A$ be the function field of the Severi-Brauer variety $SB(A)$. Then there exist elements $x_1, \ldots , x_{n-1} \in F_A$ that are algebraically independent over $k$ and such that $F_A/k(x_1, \ldots , x_{n-1})$ is a \emph{generic} field extension of degree $n$.
\end{prop}

Before giving the proof, we first discuss the following auxiliary construction. Let $F/k$ be a separable field extension of degree $n$. Fix a basis $\omega_1 = 1, \omega_2, \ldots , \omega_n$ of $F$ over $k$, and let $t_1, \ldots , t_n$ be variables. Set
\begin{equation}\label{E:Polyn1}
\varphi_{F/k}(t_1, \ldots , t_n) = \prod_{\sigma} (\sigma(\omega_1)t_1 + \cdots + \sigma(\omega_n)t_n),
\end{equation}
where the product is taken over all distinct embeddings $\sigma \colon F \hookrightarrow \overline{k}$. Clearly, $\varphi_{F/k}$ is a homogeneous polynomial
of degree $n$ in $t_1, \ldots , t_n$ with coefficients in $k$. When  $k$ is infinite, $\varphi_{F/k}$ is uniquely characterized by the condition
\begin{equation}\label{E:Polyn2}
\varphi_{F/k}(\alpha_1, \ldots , \alpha_n) = \mathrm{N}_{F/k}(\alpha_1 \omega_1 + \cdots + \alpha_n \omega_n) \ \ \text{for all} \ \ \alpha_1, \ldots , \alpha_n \in k.
\end{equation}
Consider the polynomial $f_{F/k}(T) = \varphi_{F/k}(T, t_2, \ldots , t_{n-1}, 1)$ over $L := k(t_2, \ldots , t_{n-1})$.
\begin{lemma}\label{L:SplF}
Keeping the preceding notations, let $\tilde{F}$ be the normal closure of $F$ over $k$. Then the splitting field $E$ of $f_{F/k}(T)$ over $L$ coincides with $\tilde{F}L = \tilde{F}(t_2, \ldots , t_{n-1})$, and therefore $\Ga(E/L) \simeq \Ga(\tilde{F}/k)$.
\end{lemma}
Indeed, it follows from (\ref{E:Polyn1}) that
$$
f_{F/k}(T) = \prod_{\sigma} (T + (\sigma(\omega_2)t_2 + \cdots + \sigma(\omega_{n-1})t_{n-1} + \sigma(\omega_n))).
$$
This shows that $E \subset \tilde{F}L$. On the other hand, if $\tau \in \Ga(\tilde{F}L/L) = \Ga(\tilde{F}/k)$ fixes the element $(\sigma(\omega_2)t_2 + \cdots + \sigma(\omega_{n-1})t_{n-1} + \sigma(\omega_n))$, then it fixes all the elements $\sigma(\omega_2), \ldots , \sigma(\omega_n)$. It follows that if $\tau$ fixes $E$ then $\tau = \mathrm{id}$. So, $E = \tilde{F}L$, hence $\Ga(E/L) = \Ga(\tilde{F}/k)$.

\vskip2mm

\noindent {\it Proof of Proposition \ref{P:SB2}.} The argument is a refinement of
the proof of Lemma \ref{L:SB1}, and we will freely use the notations introduced therein. Recall that the key point in realizing $F_D$ as a degree $n$ extension of the rational function field $F = k(x_1, \ldots , x_{n-1})$ was the fact that $SB(D)$ is birationally isomorphic to $Z_W$ for a suitable choice of an $(n+1)$-dimensional subspace $W \subset D$. While this fact remains valid without any changes for any central simple algebra $A$, in order to ensure that the extension $F_A/F$ is generic, we need to specialize the choice of $W$. First, since $k$ is finitely generated, $A$ contains a maximal subfield $P$ that is a generic extension of $k$ (this follows immediately, for example, from Theorem \ref{T:Exist}).  It was shown by Saltman \cite[4.2(c)]{Salt1}, \cite[13.28]{Salt2} that for a Zariski-dense set of $a \in A$, the space $W := P + ka$ is $(n+1)$-dimensional and the corresponding variety $Z_W$ is birationally isomorphic to $SB(A)$. Fix one such $a$. Pick a basis $w_1 = 1, \ldots , w_n$ of $P/k$; then $w_1, \ldots , w_n, w_{n+1} := a$ is a basis of $W$. Take $i = 1$, $j = n$ as in the proof of Lemma \ref{L:SB1} and consider the corresponding polynomial
$$
p_W(T) = \nu_W(T, x_1, \ldots,  x_{n-2}, 1, x_{n-1}),
$$
noting that $p_W$ is monic and  has coefficients in the ring $R := k[x_1, \ldots , x_{n-1}]$. As in Lemma \ref{L:SB1}, the polynomial $p_W(T)$ is irreducible over $F = k(x_1, \ldots , x_{n-1})$ and the extension $F_A/F$ is obtained by adjoining a root of $p_W$. In order to prove that the extension is generic, we will use specialization. Let $E$ be the splitting field of $p_W$ and $G = \Ga(E/F)$ be the corresponding Galois group. Furthermore, let $S$ be the integral closure of $R$ in $E$, and let $\mathfrak{p}$ be the prime ideal of $R$ generated by $x_{n-1}$. Since the restriction of the reduced norm map $\mathrm{Nrd}_{A/k}$ to $P$ coincides with the usual norm map $\mathrm{N}_{P/k}$, we see from (\ref{E:Polyn2}) that $\nu_W(t_1, \ldots , t_n, t_{n+1})(\mathrm{mod}\: t_{n+1})$ coincides with $\varphi_{P/k}(t_1, \ldots , t_n)$, from which it follows that $p_W(T)(\mathrm{mod}\: \mathfrak{p})$ coincides with $f_{P/k}(T)$. In particular, since $f_{P/k}(T)$ is separable, so is $p_W(T)$. Fix a prime ideal $\mathfrak{P} \subset S$ lying above $\mathfrak{p}$, and let $G(\mathfrak{P})$ be its decomposition group.  Then according to \cite[Ch. VII, Proposition 2.5]{Lang-Alg}, there is a natural surjective homomorphism of $G(\mathfrak{P})$ to the automorphism group $H$ of the field of fractions of $S/\mathfrak{P}$ over $L$ (which is the field of fractions of $R/\mathfrak{p}$). On the other hand, it follows from our construction and Lemma \ref{L:SplF} that the Galois group of the splitting field of $f_{P/k}(T)$ is the symmetric group $S_n$, so $H$ admits a surjection onto $S_n$. Thus, a subgroup of $G$ admits a surjection onto $S_n$, and therefore $\vert G \vert \geqslant n!$. However, $G$ is the Galois group of the splitting field of a  separable polynomial of degree $n$, hence must be isomorphic to a subgroup of $S_n$. Thus, $G \simeq S_n$, as required. \hfill $\Box$

\vskip3mm
It would be interesting to determine if the conclusion of the proposition remains valid without assuming that $k$ is finitely generated.

\vskip1mm

We can now complete the proof of Theorem \ref{T:Kill1} by imitating the proof of Theorem \ref{T:GenDA}. So, let $G = \mathrm{SL}_{1, A}$, where $A$ is a central simple  algebra of degree $n$ over a finitely generated field $k$.  Set $F  = k(x_1, \ldots , x_{n-1})$, and suppose that $G' \in \gen_F(G \times_k F)$. Using Theorem \ref{T:transc} repeatedly, we see that $G' = H \times_k F$ for some $H \in \gen_k(G)$. Since $H$ is an inner twist of $G$, we have $H = \mathrm{SL}_{1 , B}$ for some central simple algebra $B$ of degree $n$ over $k$. It remains to show that the classes $[A] , [B] \in \Br(k)$ generate the same subgroup.

Using Proposition \ref{P:SB2}, we present the function field $F_A$ of the Severi-Brauer variety $SB(A)$ as a degree $n$ generic extension of $F$, and then arguing as in the proof of Theorem \ref{T:GenDA}, we conclude that $F_A$ is $F$-isomorphic to a maximal \'etale subalgebra of $A \otimes_k F$.  Let $T =
\mathrm{R}^{(1)}_{F_A/F}(\mathbb{G}_m)$ be the corresponding maximal $F$-torus of $G \times_k F$. By our assumption, $H \times_k F \in \gen_F(G \times_k F)$, so $T$
is $F$-isomorphic to a maximal $F$-torus $T'$ of $G'$; the latter is the norm torus $\mathrm{R}^{(1)}_{E/F}(\mathbb{G}_m)$ for some maximal \'etale subalgebra $E$ of $B
\otimes_k F$, which in fact is a field extension as $T'$ is $F$-anisotropic. Since the field extension $F_A/F$ is generic by construction, we can use Lemma \ref{L:Iso} to conclude that $F_A$ is $F$-isomorphic to $E$; in other words, $F_A$ admits an $F$-embedding into $B \otimes_k F$. As in the proof of Theorem \ref{T:GenDA}, we observe that then $F_A$ splits $B$, so invoking Amitsur's Theorem, we see that $[B] \in \langle [A] \rangle$. The inclusion $[A] \in \langle [B] \rangle$ is established by a symmetric argument. \hfill $\Box$

\vskip2mm

\noindent {\bf 8.3. Proof of Theorem \ref{T:Kill2}.} We recall that an algebraic group $G$ of type $\textsf{G}_2$ over a field $k$ of characteristic $\neq 2$
can be realized as the automorphism group of an octonian algebra $\mathbb{O}(a, b, c)$ corresponding to a triple $(a, b, c) \in k^{\times}  \times k^{\times} \times k^{\times}$. The norm form $q$ of $\mathbb{O}(a, b, c)$ is the Pfister form $\ll a, b, c \gg$ in standard notations; we will write it as
$$
q(x_0, x_1, \ldots , x_7) = x_0^2 + q'(x_1, \ldots , x_7) \ \ \text{where} \ \ q'(x_1, \ldots , x_7) = -ax_1^2 - bx_2^2 + \cdots .
$$
The following facts are well-known:

\vskip2mm

\noindent (1) \parbox[t]{15cm}{for a field extension $F/k$, the group $G$ is either split or anisotropic over $F$, cf. \cite{Tits};}

\vskip1mm

\noindent (2) \parbox[t]{15cm}{Two $K$-groups $G_1$ and $G_2$ of type $\textsf{G}_2$ with associated norm forms $q_1$ and $q_2$ are $F$-isomorphic if and only if $q_1$ and $q_2$ are equivalent over $F$, cf. \cite[Proposition 33.19]{Invol};}

\vskip1mm

\noindent (3) \parbox[t]{15cm}{$G$ is split over $F$ if and only if $q$ is hyperbolic (equivalently, isotropic) over $F$ --- this follows from (2).}

\vskip2mm

\noindent It is enough to show that if $G_1$ and $G_2$ are two $k$-groups of type $\textsf{G}_2$ such that for $P := k(x_1, \ldots , x_6)$ the groups $\mathscr{G}_1 :=G_1 \times_k P$ and $\mathscr{G}_2 :=G_2 \times_k P$ are in the same genus, then $G_1 \simeq G_2$ over $k$. We may assume that $G_1$ and $G_2$ are anisotropic over $k$, and let $q_1$ and $q_2$ be the corresponding norm forms. Set
$$
L = k(x_1, \ldots , x_6)\left(\sqrt{-q'_1(x_1, \ldots , x_6, 1)} \right) = P\left(\sqrt{-q'_1(x_1, \ldots , x_6, 1)} \right)
$$
(the ``homogeneous function field" of $q_1$ in the terminology of \cite{Lam}). Then $q_1$ represents zero over $L$, so $G_1$ splits over $L$. A standard argument shows that $\mathscr{G}_1$ contains a maximal $P$-torus $\mathscr{T}$ of the form $\mathscr{T} = \mathrm{R}^{(1)}_{L/P}(\mathbb{G}_m) \times \mathrm{R}^{(1)}_{L/P}(\mathbb{G}_m)$ cf. (\cite[Lemma 6.17]{PR}).
By our assumption, $\mathscr{G}_1$ and $\mathscr{G}_2$ are in the same genus, and in particular, $\mathscr{T}$ is $P$-isomorphic to a maximal $P$-torus of $\mathscr{G}_2$, implying that $\mathscr{G}_2$ becomes split over $L$. Thus,  the 3-Pfister form $q_2$ becomes split over the function field of the 3-Pfister form $q_1$, and therefore the forms $q_1$ and $q_2$ are equivalent over $k$
(cf. \cite[Ch. X, Corollary 4.10]{Lam}). So, the groups $G_1$ and $G_2$ are $k$-isomorphic, as required.

\vskip2mm

\noindent {\bf 8.4. Motivic genus.} The following variation of the notion of  the genus, proposed by A.S.~Merkurjev, provides a different perspective on the above results. He defined the {\it motivic genus} $\gen_{m}(G)$ of an absolutely almost simple algebraic $k$-group $G$ to be the set of $k$-isomorphism classes of (inner) $k$-forms $G'$ of $G$ such that $G' \times_k F  \in \gen_F(G \times_k F)$ for \emph{all} field extensions $F/k$. Then Theorem \ref{T:Kill1} implies that for $G = \mathrm{SL}_{1 , A}$, where $A$ is a central simple algebra of degree $n$, the motivic genus is always finite of size $\leq n-1$, and reduces to a single element if $A$ has exponent two. In addition, by Theorem \ref{T:Kill2}, the motivic genus of a group of type $\textsf{G}_2$ also reduces to a single element. Furthermore, according to a result of Izhboldin \cite{Izhb}, for given non-degenerate quadratic forms $q$ and $q'$ of {\it odd} dimension over a field $k$ of characteristic $\neq 2$, the condition

\vskip1mm

\noindent $(\dagger)$ {\it $q$ and $q'$ have the same Witt index over any extension $F/k$}

\vskip1mm

\noindent implies that $q$ and $q'$ are scalar multiples of each other (this conclusion being false for even-dimensional forms). It follows that $\vert \gen_m(G) \vert = 1$ for $G = \mathrm{Spin}_n(q)$ with $n$ odd. We note that condition $(\dagger)$ is equivalent to the fact that the motives of $q$ and $q'$ in the category of Chow motives are isomorphic (cf. Vishik \cite{Vish1}, \cite[Theorem 4.18]{Vish2}, and also Karpenko \cite{Karp}), which prompted the choice of terminology for this version of the genus. One can expect the motivic genus to be finite for all absolutely almost simple groups, of size bounded by a constant depending only on the type of the group (at least over fields of ``good'' characteristic, but not necessarily finitely generated). On the other hand, Conjecture~1.7 asserts that the genus gets reduced to the motivic genus (i.e., becomes as small as possible) after a suitable purely transcendental extension of the base field.

\vskip5mm

\section{Weakly commensurable Zariski-dense subgroups}\label{S:WC}

The goal of this section is to prove Theorem \ref{T:WC-GR} that relates the presence of a finitely generated Zariski-dense subgroup weakly commensurable to a given one with good reduction. First, let us  fix a model $\mathfrak{X} = \mathrm{Spec}\: A$ for $k$,  i.e. an affine integral normal scheme of finite type over $\mathbb{Z}$ with function field $k$, and let $V$ denote the set of discrete valuations of $k$ associated with the prime divisors on $\mathfrak{X}$. We will consider separately the two cases where  $\dim \mathfrak{X} = 1$ and $\dim \mathfrak{X} > 1$, respectively.

\vskip1mm

\noindent {\bf \ref{S:WC}.1. Proof of Theorem \ref{T:WC-GR} in the case $\dim \mathfrak{X} = 1$.} In this case, $k$ is a number field (recall that $\mathrm{char}\: k = 0$), and $V$ consists of almost all nonarchimedean valuations of $k$.  The proof here is an adaptation of the argument developed in \cite[\S 5]{PR-Appl} for a different, although related, purpose. By \cite[Theorem 6.7]{PR}, one can find a finite subset $S_1 \subset V$ such that $G$ is quasi-split over the completion $k_v$ for all $v \in V \setminus S_1$. Furthermore, let $\ell/k$ be the minimal Galois extension over which $G$ becomes an inner form of the split group, and choose a finite subset $S_2 \subset V$ so that $\ell/k$ is unramified at all $v \in V \setminus S_2$. Finally, since $k$ coincides with the trace field $k_{\Gamma}$, it follows from the Strong Approximation Theorem of Weisfeiler \cite{Weis} that there exists a finite subset $S_3 \subset V$ such that the closure of $\Gamma$ in $G(k_v)$ in the $v$-adic topology is open for all $v \in V \setminus S_3$. Set $S(\Gamma) = S_1 \cup S_2 \cup S_3$. The fact that this set is as required in Theorem \ref{T:WC-GR} is an immediate consequence of the following.
\begin{prop}\label{P:quasi-split}
Let $G'$ be an absolutely almost simple $k$-group such that there exists a finitely generated Zariski-dense subgroup $\Gamma' \subset G'(k)$ that is weakly commensurable to $\Gamma$. Then $G'$ is quasi-split over $k_v$, and consequently has good reduction, for all $v \in V \setminus S(\Gamma).$
\end{prop}
\begin{proof}
Let $\ell'$ be the minimal Galois extension of $k$ over which $G'$ becomes an inner form of the split group. By Theorem \ref{T:basic}, the existence of $\Gamma'$ that is weakly commensurable to $\Gamma$ implies that either $G$ and $G'$ have the same type or one of them is of type $\textsf{B}_{\ell}$ and the other of type $\textsf{C}_{\ell}$, and also that $\ell = \ell'$. The latter means that when $G$ and $G'$ are of the same type, then the corresponding adjoint groups $\overline{G}'$ and $\overline{G}$ are inner twists of each other
over $k$, hence over $k_v$. In order to prove that $G'$ is quasi-split over $k_v$, it is enough to show that
\begin{equation}\label{E:rank}
\mathrm{rk}_{k_v}\: G' \geqslant \mathrm{rk}_{k_v}\: G.
\end{equation}
Indeed, when one of the groups is of type $\textsf{B}_{\ell}$ and the other of type $\textsf{C}_{\ell}$, we see that $G$ is $k_v$-split, and then the inequality shows that $G'$ is also $k_v$-split. Next, suppose that $G$ and $G'$ are of the same type. As we pointed out above, $\overline{G}'$ is an inner twist of $\overline{G}$, and therefore the $*$-actions of the absolute Galois group of $k_v$ on the Tits indices of $G$ and $G'$ are identical
(cf. \cite[Lemma 4.1(a)]{PR-WC}). It is well-known that the relative rank of a semisimple group equals the number of distinguished orbits under the $*$-action on its Tits index. By our construction, $G$ is quasi-split over $k_v$, so all $*$-orbits are distinguished. Then (\ref{E:rank}) implies that all $*$-orbits in the Tits index of $G'$ are also distinguished, so $G'$ is $k_v$-quasi-split.

Now, to prove (\ref{E:rank}), we first use Theorem \ref{T:ExGenElts2} to find a regular semisimple element $\gamma$ so that the corresponding torus $T = C_G(\gamma)^{\circ}$ is generic over $k$ and contains a maximal $k_v$-split torus of $G$ over $k_v$. By our assumption, $\gamma$ is weakly commensurable to some semisimple element $\gamma' \in \Gamma'$ of infinite order. Let $T'$ be a maximal $k$-torus of $G'$ containing $\gamma'$. Then it follows from Proposition \ref{P:Isog} that there exists a $k$-defined isogeny $T \to T'$.
Since $\mathrm{rk}_{k_v}\: T = \mathrm{rk}_{k_v}\: G$ by construction, the required inequality (\ref{E:rank}) follows.
Thus, $G'$ is quasi-split over $k_v$. Since $\ell' = \ell$, we obtain that $\ell'$ is unramified at $v$, and therefore $G'$ has good reduction at $v$ (cf. \cite[Corollary 7.9.4]{KalPr}). \end{proof}

\noindent {\bf \ref{S:WC}.2. Zariski-density of reductions when $\dim \mathfrak{X} > 1$.} In this case, the residue fields $k^{(v)}$ of the valuations $v \in V$ are
{\it infinite} finitely generated fields. The goal of this subsection is to establish a result about the Zariski-density of the reductions of $\Gamma$ modulo almost all $v \in V$ that will play a crucial role in the proof of Theorem \ref{T:WC-GR}. Fix a faithful $k$-defined representation $G \hookrightarrow \mathrm{GL}_n$. By shrinking $V$, we may assume without loss of generality that for all $v \in V$, the reduction $\Gu^{(v)}$ associated with this realization is a connected absolutely almost simple algebraic group over the residue field $k^{(v)}$. We denote by $\rho_v \colon G(\mathscr{O}_{k,v}) \to
\Gu^{(v)}(k^{(v)})$ the corresponding reduction map (where $\mathscr{O}_{k,v}$ is the valuation ring of $v$ in $k$).
\begin{prop}\label{P:R-ZD}
Let $\Gamma \subset G(k)$ be a finitely generated Zariski-dense subgroup. Then for almost all $v \in V$, we have the inclusion $\Gamma \subset G(\mathscr{O}_v)$ and the reduction $\Gamma^{(v)} := \rho_v(\Gamma)$ is Zariski-dense in $\Gu^{(v)}$.
\end{prop}
\begin{proof}
The first assertion is obvious. To prove the second one, we observe that since $G$ is absolutely almost simple and $k$ has characteristic zero, the adjoint representation $r \colon G \to \mathrm{GL}(\mathfrak{g})$ on the Lie algebra $\mathfrak{g}=L(G)$ is (absolutely) irreducible. Being Zariski-dense in $G$, the subgroup $\Gamma$ also acts on $\mathfrak{g}$ absolutely irreducibly. Then by Burnside's Lemma (cf. \cite[7.3]{Lam1}), the image $r(\Gamma)$ spans $\mathrm{End}(\mathfrak{g})$. Fix a basis $a_1, \ldots , a_m$ $(m = \dim \mathfrak{g})$ of $\mathfrak{g}(k) \subset M_n(k)$ over $k$, and let $e_{ij}$ $(i,j = 1, \ldots , m)$ be the corresponding standard basis of $\mathrm{End}_k(\mathfrak{g}(k))$. Then we can find elements $\gamma_1, \ldots , \gamma_d \in \Gamma$ such that for all $i,j = 1, \ldots , m$ there are expressions
\begin{equation}\label{E:basis}
e_{ij} = \sum_{\ell = 1}^d \alpha^{\ell}_{ij} r(\gamma_{\ell}) \ \ \text{with} \ \ \alpha^{\ell}_{ij} \in k.
\end{equation}
Then for almost all $v \in V$, the following properties hold:

\vskip2mm

\noindent $\bullet$ \parbox[t]{16cm}{the elements $a_1, \ldots , a_m$ belong to $\mathfrak{g}(\mathscr{O}_{k,v}) =
\mathfrak{g}(k) \cap M_n(\mathscr{O}_{k,v})$ and their reductions $\bar{a}_1, \ldots , \bar{a}_m \in M_n(K^{(v)})$ form
a $k^{(v)}$-basis of $\underline{\mathfrak{g}}^{(v)}(k^{(v)})$, where $\underline{\mathfrak{g}}^{(v)}$ is the Lie algebra
of the reduction $\Gu^{(v)}$;}

\vskip2mm

\noindent $\bullet$ all coefficients $\alpha^{\ell}_{ij}$ belong to $\mathscr{O}_{k,v}$.

\vskip2mm

\noindent We note that for any such $v$, the endomorphisms $e_{ij}$ leave $\mathfrak{g}(\mathscr{O}_{k,v})$ invariant, and their reductions $\bar{e}_{ij}$ form the standard basis of $\mathrm{End}_{k^{(v)}}(\underline{\mathfrak{g}}^{(v)}(k^{(v)}))$ associated with the basis $\bar{a}_1, \ldots , \bar{a}_m$. Reducing (\ref{E:basis}), we obtain the relations
$$
\bar{e}_{ij} = \sum_{\ell = 1}^d \bar{\alpha}^{\ell}_{ij} r(\rho_v(\gamma_{\ell})),
$$
where $\bar{\alpha}^{\ell}_{ij}$ denotes the image of $\alpha^{\ell}_{ij}$ in $k^{(v)}$. These relations show that $\Gamma^{(v)}$ acts on $\underline{\mathfrak{g}}^{(v)}$ absolutely irreducibly. Letting $H$ denote the Zariski-closure of $\Gamma^{(v)}$ in $\Gu^{(v)}$,
we observe that the Lie algebra $L(H)$ is a $\Gamma^{(v)}$-invariant subspace of $\underline{\mathfrak{g}}^{(v)}$, and therefore there
are only two possibilities: $L(H) = \underline{\mathfrak{g}}^{(v)}$ or $L(H) = \{ 0 \}$. In the first case, $H = \Gu^{(v)}$,
i.e. $\Gamma^{(v)}$ is Zariski-dense, as desired. In the second case, $H$, hence $\Gamma^{(v)}$, is finite. So, to complete the proof of the proposition we need to show that $\Gamma^{(v)}$ is {\it infinite} for almost all $v$. For this, we will consider two cases.

\vskip2mm

\noindent {\sc Case 1:} $\mathrm{char}\: k^{(v)} = 0.$ It follows from Jordan's Theorem (cf. \cite{Collins}) that there exists an integer $j > 0$ (depending on $n$) such that every finite subgroup $\Phi \subset \mathrm{GL}_n(F)$, where $F$ is any field of characteristic zero, contains an abelian normal subgroup of index dividing $j$, and then the commutator subgroup $[\Phi^{(j)} , \Phi^{(j)}]$ of the subgroup $\Phi^{(j)}$ generated by the $j$th powers is trivial. Now, since $\Gamma$ is Zariski-dense in $G \subset \mathrm{GL}_n$, it follows that
$\Delta := [\Gamma^{(j)} , \Gamma^{(j)}]$ is also Zariski-dense. In particular, we can find a {\it nontrivial} element $\delta \in \Delta$; then for almost all $v$ the reduction $\rho_v(\delta)$ is nontrivial. Then the group $\Gamma^{(v)}$ must be  {\it infinite}. Indeed, otherwise Jordan's Theorem would yield that
$$
[(\Gamma^{(v)})^{(j)} , (\Gamma^{(v)})^{(j)}] = \rho_v(\Delta)
$$
is trivial, which is not the case by our construction.

\vskip2mm

\noindent {\sc Case 2:} $\mathrm{char}\: k^{(v)} > 0.$ We will show that there exists $\gamma \in \Gamma$ such that $\rho_v(\gamma)$ has {\it infinite} order for almost all $v$ at hand. For this, it is enough to make sure that the trace $\mathrm{tr}\: (r(\rho_v(\gamma)))$ is not a sum of roots of unity. Let $k_0$ be the algebraic closure of $\Q$ in $k$; we note that $[k_0 : \Q] < \infty$ as $k$ is finitely generated, and $k \neq k_0$ since $\dim \mathfrak{X} > 1$.  So, as $k$ is the trace field of $\Gamma$, we can find $\gamma \in \Gamma$ so that $f := \mathrm{tr}\: (r(\gamma)) \notin k_0.$ We will show that this $\gamma$ is as required. First, we observe that for any $v$ with $\mathrm{char}\: k^{(v)} > 0$, the restriction $v_0$ of $v$ to $k_0$ is nontrivial, and the residue field $k_0^{(v_0)}$ embeds into $k^{(v)}$. Furthermore, the residue $\bar{f}$ coincides with $\mathrm{tr}\: (r(\rho_v(\gamma)))$. Now, it is enough to show that $\bar{f}$ is not algebraic over $k_0^{(v_0)}$, which follows from the fact that the following two properties hold for almost all $v$:

\vskip1mm

\noindent (1) $k_0^{(v_0)}$ is algebraically closed in $k^{(v)}$;

\noindent (2) $\bar{f} \notin k_0^{(v_0)}$.

\vskip1mm

\noindent The proof is based on the following well-known fact: {\it Let $X$ be an irreducible algebraic variety over a field $K$; then $X$ is absolutely
irreducible if and only if $K$ is algebraically closed in the field of rational functions $K(X)$} (see, for example, \cite[Proposition 5.50]{GW}). To apply this fact in our situation, we observe that since the model $\mathfrak{X} = \mathrm{Spec}\: A$ is normal, it can be viewed as a scheme over the ring of integers $\mathcal{O}_0$ of $k_0$. Besides, the $k_0$-variety $X = \mathfrak{X} \times_{\mathcal{O}_0} k_0$ is absolutely irreducible. Then it follows from the classical theorem of Bertini-Noether (cf. \cite[Proposition 10.4.2]{FrJar}) that for almost all $v$, the reduction
$$
\underline{\mathfrak{X}}^{(v_0)} = \mathfrak{X} \otimes_{\mathcal{O}_0} k^{(v_0)}
$$
is an absolutely irreducible variety over $k_0^{(v_0)}$ whose field of rational functions coincides with $k^{(v)}$. Applying the statement mentioned above, we obtain property (1). Furthermore, since $f$ is not constant on $X$, we can find two points $x_1 , x_2 \in X(\overline{k_0})$ such that $f(x_1) \neq f(x_2)$. Then for almost all $v$, the points admit the reductions $\bar{x}_1$, $\bar{x}_2$ with respect to an extension of $v_0$, and for the reduction $\bar{f} \in k^{(v)}$, we have $\bar{f}(\bar{x}_1) \neq \bar{f}(\bar{x}_2)$. This means that $\bar{f} \notin k_0^{(v_0)}$, verifying property (2) and completing the argument. \end{proof}

\vskip1mm

\noindent {\bf \ref{S:WC}.3. Proof of Theorem \ref{T:WC-GR} in the case $\dim \mathfrak{X} > 1$.} It follows from Proposition \ref{P:R-ZD} that there exists a finite subset $S(\Gamma) \subset V$ such that for any $v \in V \setminus S(\Gamma)$ the following two conditions hold:

\vskip2mm

\noindent (a) \parbox[t]{15.5cm}{$G$ has good reduction at $v$, so that the $\underline{G}^{(v)}$ is a connected absolutely almost simple group;}

\vskip1mm

\noindent (b) \parbox[t]{15.5cm}{$\Gamma \subset G(\mathscr{O}_{k, v})$, and for the reduction map $\rho_v \colon G(\mathscr{O}_{k,v}) \to \Gu^{(v)}(k^{(v)})$, the image $\Gamma^{(v)} = \rho_v(\Gamma)$ is Zariski-dense in $\Gu^{(v)}$.}

\vskip2mm

\noindent Suppose now that $G'$ is an absolutely almost simple algebraic $k$-group such that there exists a finitely generated Zariski-dense subgroup $\Gamma' \subset G'(k)$ that is weakly commensurable to $\Gamma$. As before, we denote by $\ell$ (resp., $\ell'$) the minimal Galois extension of $k$ over which $G$ (resp., $G'$) becomes an inner form of the split group. By Theorem \ref{T:basic}, the Weyl groups of $G$ and $G'$ have the same order $w$  and  $\ell = \ell'$. Fix an extension $u$ of $v$ to $\ell$; it follows from (a) that the extension $\ell/k$ is unramified at $v$.

Since $\Gamma^{(v)}$ is (finitely generated and) Zariski-dense in $\Gu^{(v)}$, by Theorem \ref{T:ExGenElts} there exists a regular semisimple element $\bar{\gamma} \in \Gamma^{(v)}$ that is generic over $\ell^{(u)}$. Write $\bar{\gamma} = \rho_v(\gamma)$ with $\gamma \in \Gamma$.
\begin{lemma}\label{L:elt}
{\rm (1)} $\gamma$ is a regular semisimple element of infinite order.

\noindent {\rm (2)} \parbox[t]{15.9cm}{The maximal $k$-torus $T = C_G(\gamma)^{\circ}$ is generic over $k_v$ and the extension $k_T/k$ is unramified at $v$.}
\end{lemma}
\begin{proof}
(1): Let $c(t)$ be the characteristic polynomial of $\mathrm{Ad}\: \gamma$. Then its reduction $\bar{c}(t)$ is the characteristic polynomial of
$\mathrm{Ad}\: \bar{\gamma}$. Since $\bar{\gamma}$ is regular semisimple, the multiplicity of $1$ as a root of $\bar{c}(t)$ equals $r = \mathrm{rk}\:
\underline{G}^{(v)} = \mathrm{rk}\: G$. So, the multiplicity of $1$ as a root of $c(t)$ is $\leq r$, implying that it is in fact precisely $r$ and hence the element $\gamma$ is regular and semisimple. (To see the latter, let us consider the Jordan decomposition if $\gamma = \gamma_s \gamma_u$; then $c(t)$ is also  the characteristic polynomial of $\mathrm{Ad}\: \gamma_s$. On the other hand, since $\mathrm{char}\: k = 0$, the assumption $\gamma_u \neq e$ would imply the existence of a nontrivial nilpotent element in the Lie algebra centralized by $\gamma_s$. But this would clearly make the multiplicity of 1 as a root of $c(t)$ greater than  $r$, a contradiction.)

Furthermore, since $\bar{\gamma}$ has infinite order, so does $\gamma$.

\vskip1mm

(2): Let $E = (k_v)_T$ be the splitting field of $T$ over $k_v$. Clearly, $E$ contains $\ell_u$, and we let $\tilde{u}$ denote the extension of $u$ to $E$.
Since $\ell/k$ is unramified at $v$, it follows from \cite[Lemma 4.1]{PR-WC} that it is enough to prove that
\begin{equation}\label{E:=}
[E : \ell_u] =: w = [E^{(\tilde{u})} : \ell^{(u)}].
\end{equation}
Since $E$ contains the splitting field of $c(t)$, the residue field $E^{(\tilde{u})}$ contains the splitting field of $\bar{c}(t)$. By our construction, the $k^{(v)}$-torus $\bar{T}:= C_{\underline{G}^{(v)}}(\bar{\gamma})$ is generic over $\ell^{(u)}$, so $[(k^{(v)})_{\bar{T}} : \ell^{(u)}] = w$. On the other
hand, the roots of $\bar{c}(t)$ include the values at $\bar{\gamma}$ of all roots $\alpha \in \Phi(\underline{G}^{(v)} , \bar{T})$. So, it follows from Lemma \ref{L:generate} that $(k^{(v)})_{\bar{T}} = E^{(\tilde{u})}$. Thus,
$$
w = [E : \ell_u] \geq [E^{(\tilde{u})} : \ell^{(u)}] = [(k^{(v)})_{\bar{T}} : \ell^{(u)}] = w,
$$
and (\ref{E:=}) follows.
\end{proof}

First, assume that the type of $G$ is different from $\textsf{A}_1$ and $\textsf{B}_{\ell}$, and let $v \in V \setminus S(\Gamma)$. It follows
from Lemma \ref{L:elt} and the discussion preceding it that one can find a regular semisimple element $\gamma \in \Gamma$ of
infinite order such that the $k$-torus $T = C_G(\gamma)^{\circ}$ is generic over $\ell_u$ and the splitting field $k_T$ is
unramified at $v$. By our assumption, $\gamma$ is weakly commensurable to some semisimple element $\gamma' \in \Gamma'$ of infinite
order. Let $T'$ be a maximal $k$-torus of $G'$ containing $\gamma'$. According to Corollary \ref{C:Isog}, the $k$-tori $T$ and $T'$ are
isogenous over $k$. It follows that $T'$ is generic over $\ell_u = {\ell'}_u$, hence over $k_v$. Then $G'$ has good reduction at $v$ by Theorem \ref{T:GR-not}.

Next, let $G$ be of one of the types $\textsf{A}_1$ or $\textsf{B}_{\ell}$. We then first pick a regular semisimple element $\bar{\gamma}_1 \in \Gamma^{(v)}$ of infinite order that is generic over $k^{(v)}$ (note that here $\ell = k$), and let $\bar{T}_1 = C_{\underline{G}^{(v)}}(\bar{\gamma}_1)^{\circ}$ denote the corresponding $k^{(v)}$-torus. We then pick a regular semisimple element $\bar{\gamma}_2 \in \Gamma^{(v)}$ of infinite order that is generic over the splitting field $(k^{(v)})_{\bar{T}_1}$, and let $\bar{T}_2 = C_{\underline{G}^{(v)}}(\bar{\gamma}_2)^{\circ}$. Note that the Dynkin diagrams of the types at hand do not have nontrivial automorphisms, so
$$
[(k^{(v)})_{\bar{T}_1} (k^{(v)})_{\bar{T}_2} : (k^{(v)})_{\bar{T}_1}] = w = [(k^{(v)})_{\bar{T}_2} : k^{(v)}],
$$
which implies that
\begin{equation}\label{E:disjoint}
(k^{(v)})_{\bar{T}_1} \cap (k^{(v)})_{\bar{T}_2} = k^{(v)}.
\end{equation}
Now, pick $\gamma_i \in \Gamma$ so that $\rho_v(\gamma_i) = \bar{\gamma}_i$, and let $T_i = C_G(\gamma_i)^{\circ}$, for $i = 1, 2$. Also, let $c_i(t)$ be the characteristic polynomial of $\mathrm{Ad}\: \gamma_i$. Then the reduction $\bar{c}_i(t)$ is the characteristic polynomial of $\mathrm{Ad}\: \bar{\gamma}_i$. Since $T_i$ (resp., $\bar{T}_i$) is $k_v$- (resp., $k^{(v)}$-)generic, it follows from Lemma \ref{L:generate} that $(k_v)_{T_i}$ (resp., $(k^{(v)})_{\bar{T}_i})$ coincides with the splitting field of $c_i$ (resp., $\bar{c}_i$), and therefore $(k^{(v)})_{\bar{T}_i}$ is precisely the residue field of $(k_v)_{T_i}$. We also know from Lemma \ref{L:elt} that the extension $(k_v)_{T_i}/k_v$ is unramified. Then (\ref{E:disjoint}) yields that $(k_v)_{T_1} \cap (k_v)_{T_2} = k_v$. By our assumption, the elements $\gamma_1 , \gamma_2$ are weakly commensurable to semisimple elements  $\gamma'_1 , \gamma'_2 \in \Gamma'$ of infinite order, respectively. Let $T'_i$ be a maximal $k$-torus of $G'$ containing $\gamma'_i$. Then by Corollary \ref{C:Isog}, the torus $T_i$ is $k$-isogenous to the torus $T'_i$ for $i = 1, 2$. It follows that $T'_i$ is generic over $k_v$, the extension $(k_v)_{T'_i}/k_v$ is unramified for $i = 1, 2$, and
$$
(k_v)_{T'_1} \cap (k_v)_{T'_2} = k_v.
$$
So, $G'$ has good reduction by Theorem \ref{T:AB}.

\vskip2mm

\noindent {\bf \ref{S:WC}.4. Subgroups with the same exceptional set.} A smaller Zariski-dense subgroup $\Delta \subset \Gamma$ may, a priori, require a larger exceptional set $S(\Delta)$ in Theorem \ref{T:WC-GR}. We will show in this subsection, however, that our construction of $S(\Gamma)$ produces an exceptional set that works for many subgroups $\Delta \subset \Gamma$ that are  smaller than $\Gamma$. To describe precisely the possibilities for $\Delta$, we will need the following definition.

\addtocounter{thm}{1}

\vskip2mm

\noindent {\bf Definition 9.4.} Let $\Gamma$ be an abstract group. The following subgroups will be called \emph{principal standard subgroups} of $\Gamma$:

\vskip2mm

\noindent (1) the commutator subgroup $[\Gamma , \Gamma]$;

\vskip1mm

\noindent (2) the subgroups $\Gamma^{(n)}$ generated by the $n$th powers $\gamma^n$ of elements $\gamma \in \Gamma$, for some $n \geqslant 1$.

\vskip2mm

\noindent Furthermore, a subgroup $\Delta$ of $\Gamma$ is called \emph{standard} if there exists a (finite) chain of subgroups $$\Delta = \Gamma_m \subset \Gamma_{m-1} \subset \cdots \subset \Gamma_1 \subset \Gamma_0 = \Gamma$$ such that $\Gamma_{i+1}$ contains a principal standard subgroup of $\Gamma_i$.

\vskip2mm

We note that all standard subgroups of a (finitely generated) Zariski-dense subgroup of a (connected) semisimple algebraic group are automatically Zariski-dense, although they may not be finitely generated. However, typically this does not create any additional problems in the analysis of weak commensurability since the statements dealing with the existence of generic elements having special properties remain valid for those Zariski-dense subgroups
that are contained in a finitely generated subgroup of $G(k)$. Our goal in this subsection is to present the following strengthening of Theorem~\ref{T:WC-GR}.
\begin{thm}\label{T:StSub}
Let $G$ be an absolutely almost simple algebraic group over a finitely generated field $k$ of characteristic zero, and let $V$ be a divisorial set of places of $k$. Given a Zariski-dense subgroup $\Gamma \subset G(k)$ with trace field $k$, there exists a finite subset $S(\Gamma) \subset V$ such that any absolutely almost simple algebraic $K$-group $G'$ with the property that there exists a finitely generated Zariski-dense subgroup $\Gamma' \subset G'(K)$ that is weakly commensurable to some standard subgroup $\Delta \subset \Gamma$, has good reduction at all $v \in V \setminus S(\Gamma)$.
\end{thm}
\begin{proof}
It turns out that the set $V(\Gamma)$ we have constructed in the proof of Theorem \ref{T:WC-GR} works in this more general setting. To see this, we will revisit our construction separately in the cases where $\dim \mathfrak{X} = 1$ and $\dim \mathfrak{X} > 1$. In the first case, the exceptional set $S(\Gamma)$ was constructed in subsection 9.1 as the union $S_1 \cup S_2 \cup S_3$ in the notations introduced therein. It follows from the definitions that the finite sets $S_1$ and $S_2$ are independent of $\Gamma$. Now, recall that the finite set $S_3$ is chosen so that the closure of $\Gamma$ in $G(k_v)$ is open for all $v \in V \setminus S_3$. It is easy to see, however, that every principal standard subgroup, and hence any standard subgroup, of an open subgroup of $G(k_v)$ is itself open. This implies that if the closure of $\Gamma$ in $G(k_v)$ is open, then so is the closure of any standard subgroup $\Delta \subset \Gamma$. In other words,
if the set $S_3$ is chosen as in section 9.1 for $\Gamma$, then it also ensures the required property (i.e., the openness of the closure) for any standard subgroup $\Delta \subset \Gamma$. Thus, the set $S(\Gamma) = S_1 \cup S_2 \cup S_3$ will serve as an exceptional set for $\Delta$ as well.

Recall that the exceptional set in the case $\dim \mathfrak{X} > 1$ was actually constructed in subsection 9.3 as $V(\Gamma) = S_1 \cup S_2$,  where $S_1$ consists of those $v \in V$ for which condition (a) fails, and $S_2$ consists of those $v \in V \setminus S_1$ for which condition (b) fails. Clearly, $S_1$ is independent of $\Gamma$. On the other hand, $S_2$ is chosen to be disjoint from $S_1$ so that for $v \in V \setminus (S_1 \cup S_2)$, the image $\rho_v(\Gamma)$ under the reduction map is a Zariski-dense subgroup of the connected absolutely almost simple algebraic $k^{(v)}$-group $\Gu^{(v)}$. Then $\rho_v(\Delta)$ is also Zariski-dense in $\Gu^{(v)}$ for any standard subgroup $\Delta \subset \Gamma$. This means that if we choose the finite set $S_2$ for $\Gamma$,
then condition (b) will hold true for any standard subgroup $\Delta \subset \Gamma$ and for any $v \in V \setminus (S_1 \cup S_2)$. Thus $S(\Gamma) = S_1 \cup S_2$ can be taken for an exceptional set for $\Delta$, completing the proof.
\end{proof}

\vskip2mm

\section{Application to lattices and length-commensurable Riemann surfaces}\label{S:lattice}

As we already mentioned in the introduction, there is a conjecture that predicts the existence of only finitely many possibilities for the algebraic hull of a finitely generated Zariski-dense subgroup that is weakly commensurable to a given one (\cite[Conjecture 6.1]{R-ICM}); this conjecture is a crucial element in the so-called ``eigenvalue rigidity."

\vskip2mm

\noindent {\bf Conjecture \ref{S:lattice}.1.} {\it Let $G_1$ and $G_2$ be absolutely simple (adjoint) algebraic groups over a field $F$ of characteristic zero, and let $\Gamma_1 \subset G_1(F)$ be a finitely generated Zariski-dense subgroup with trace field $k_{\Gamma_1} =: k$. Then there exists a \underline{\emph{finite}} collection $\mathscr{G}_1^{(2)}, \ldots , \mathscr{G}_r^{(2)}$ of $F/k$-forms of $G_2$ such that any finitely generated Zariski-dense subgroup $\Gamma_2 \subset G_2(F)$ that is weakly commensurable to $\Gamma_1$ is conjugate in $G_2(F)$ to a subgroup of one of the $\mathscr{G}_i^{(2)}(k)$ $(\subset G_2(F))$ for $i = 1, \ldots , r$.}

\vskip2mm

Due to Theorem \ref{T:WC-GR}, this conjecture would follow from the Finiteness Conjecture for forms with good reduction, and hence is valid in those cases where the Finiteness Conjecture has been established.  For example,  since the truth of the Finiteness Conjecture is known for inner forms of type $\textsf{A}_n$ over all finitely generated fields of characteristic zero (cf. \cite{CRR3}), we see that given a central simple algebra $A$ over a finitely generated field $k$ with $\mathrm{char}\: k = 0$ and a finitely generated Zariski-dense subgroup $\Gamma \subset G(k)$, where $G = \mathrm{SL}_{1,A}$, with trace field $k_{\Gamma} = k$, there are only finitely many isomorphism classes of central division $k$-algebras $A'$ such that there exists a finitely generated Zariski-dense subgroup $\Gamma' \subset G'(k)$, where $G' = \mathrm{SL}_{1,A'}$, that is weakly commensurable to $\Gamma$. Other available results on the Finiteness Conjecture (cf. \cite{RR-survey}) lead to a variety of cases where Conjecture 10.1 is known. We will not, however, provide a complete list of these cases here, but rather focus on the case of  finitely generated Zariski-dense subgroups  weakly commensurable to lattices (arithmetic or not), where our result (given in Theorem \ref{T:lattice} below)
may have applications to locally symmetric spaces.

\vskip2mm

\noindent {\bf \ref{S:lattice}.1. Conjecture 10.1 for lattices.} We refer the reader to \cite{Marg}, \cite{Ragh} for basic facts about lattices in semisimple Lie groups. In the case of simple Lie groups, we have the following finiteness result for weakly commensurable lattices (arithmetic or not).

\addtocounter{thm}{1}

\begin{thm}\label{T:lattice}
Let $G_1$ be an absolutely simple (adjoint) real algebraic group, and let
$\Gamma_1 \subset G_1(\R)$ be a lattice with trace field $k = k_{\Gamma_1}$. Given an
absolutely simple (adjoint) algebraic group $G_2$ over an extension $F$ of $k$, there exists a finite collection $\mathscr{G}_1^{(2)}, \ldots , \mathscr{G}_r^{(2)}$ of $F/k$-forms of $G_2$ such that a finitely generated Zariski-dense subgroup $\Gamma_2 \subset G_2(F)$ that is weakly commensurable to $\Gamma_1$ is necessarily $G_2(F)$-conjugate to a subgroup of one of the $\mathscr{G}^{(2)}_i(k)$'s $(\subset G_2(F))$.
\end{thm}

\vskip1mm

(Here we assume that each $F/k$-form $\mathscr{G}^{(2)}_i$ comes with a {\it fixed} $F$-isomorphism $\phi_i \colon \mathscr{G}^{(2)}_i \times_k F \to G_2$, which then defines an embedding of groups $\mathscr{G}^{(2)}_i(k) \hookrightarrow \mathscr{G}^{(2)}_i(F) \hookrightarrow G_2(F)$.)

\vskip1mm

\begin{proof}
Since $\Gamma_1$ is finitely generated, the field $k$ is also finitely generated.  Let $V$ be a divisorial set of places of $k$. Next, for the algebraic hull  $\mathscr{G}^{(1)}$  of $\Gamma_1$, we have the inclusion $\Gamma_1 \subset \mathscr{G}^{(1)}(k)$. Now, let $\Gamma_2 \subset G_2(F)$ be an arbitrary finitely generated Zariski-dense subgroup weakly commensurable to $\Gamma_1$. Then the trace field $k_{\Gamma_2}$ coincides with $k$ (cf. Theorem \ref{T:basic}(2)). Let $\mathscr{G}_{\Gamma_2}$ be the algebraic hull of $\Gamma_2$, so that $\Gamma_2 \subset \mathscr{G}_{\Gamma_2}(k)$. According to Theorem \ref{T:WC-GR}, we can find a finite subset $S(\Gamma_1) \subset V$ such that all such $\mathscr{G}_{\Gamma_2}$'s have good reduction at any $v \in V \setminus S(\Gamma_1)$. We now recall that unless $G_1 = \mathrm{PGL}_2$, the trace field $k$ is a field of algebraic numbers (cf. \cite[7.67 and 7.68]{Ragh}), and that if $G_1$ is isomorphic to $\mathrm{PGL}_2$, then so is $G_2$ (cf. Theorem \ref{T:basic}(1)). Since the Finiteness Conjecture for forms with good reduction has already been established in the cases where either $k$ is a number field (cf. \cite[Proposition 5.2]{RR-survey}) or the group is isogenous to $\mathrm{PGL}_2$ (cf. \cite[Theorem 7.6]{RR-survey}), it follows that there exists a {\it finite} collection $\overline{\mathscr{G}}^{(2)}_1, \ldots , \overline{\mathscr{G}}^{(2)}_{\bar{r}}$ of $F/k$-forms of $G_2$ with the following property: Given a finitely generated Zariski-dense subgroup $\Gamma_2 \subset G_2(F)$, there exist an $i \in \{1, 2, \ldots , \bar{r}\}$ and a $k$-isomorphism $\varphi_{\Gamma_2 , i} \colon \mathscr{G}_{\Gamma_2} \to \overline{\mathscr{G}}^{(2)}_i$, and then of course $\varphi_{\Gamma_2 , i}(\Gamma_2) \subset \overline{\mathscr{G}}^{(2)}_i(k)$. On the other hand, we have $F$-isomorphisms
$$
\iota_{\Gamma_2} \colon \mathscr{G}_{\Gamma_2} \times_k F \to G_2 \ \ \text{and} \ \ \bar{\iota}_i \colon \overline{\mathscr{G}}^{(2)}_i \times_k F \to G_2.
$$
Then $\sigma_{\Gamma_2 , i} := \bar{\iota}_i \circ \varphi_{\Gamma_2 , i} \circ \iota_{\Gamma_2}^{-1}$ is an $F$-automorphism of $G_2$, which, in terms of the embeddings of the groups of $k$-rational points given by $\iota_{\Gamma_2}$ and $\iota_i$, has the property $\sigma_{\Gamma_2 , i}(\Gamma_2) \subset \mathscr{G}^{(2)}_i(k)$. If $\sigma_{\Gamma_2 , i}$ is inner, then it is conjugation by an element of $G_2(F)$ as $G_2$ is adjoint, giving the required fact. To handle the general case, we need to expand the collection of forms $\overline{\mathscr{G}}^{(2)}_i$ and (fixed) $F$-isomorphisms $\bar{\iota}_i \colon  \overline{\mathscr{G}}_i \to G_2$, $i = 1, \ldots , \bar{r}$. Namely, let $\theta_j$ ($j = 1, \ldots , t$) be a system of representatives of  the cosets $\mathrm{Aut}(G_2)(F)/\mathrm{Int}(G_2)(F)$, where $\mathrm{Aut}(G_2)(F)$ is  the group of $F$-defined automorphisms of the algebraic $F$-group $G_2$ and $\mathrm{Int}(G_2)(F)$ is the subgroup of inner automorphisms. Then for $i = 1, \ldots , \bar{r}$ and $j = 1, \ldots , t$, we set
$$
\mathscr{G}^{(2)}_{i,j} = \overline{\mathscr{G}}_i \ \ \text{and} \ \ \iota_{i,j} = \theta_j^{-1} \circ \iota_i.
$$
We have already seen that given a finitely generated Zariski-dense subgroup $\Gamma_2 \subset G_2(F)$, we can find $i \in \{1, \ldots , \bar{r} \}$ such that $\sigma_{\Gamma_2 , i}(\Gamma_2) \subset \overline{\mathscr{G}}_i(K)$. Furthermore, we can find $j \in \{1, \ldots , t\}$ so that $\sigma_{\Gamma_2 , i} = \theta_j \circ \tau_{\Gamma_2, i, j}$, with $\tau_{\Gamma_2, i, j}$ inner. Then
$$
\tau_{\Gamma_2, i, j}(\Gamma_2) \subset \iota_{i,j}(\mathscr{G}_{i,j}(k)),
$$
as required.
\end{proof}

\noindent {\bf \ref{S:lattice}.2. An application to length-commensurable Riemann surfaces.} Let $\mathbb{H} = \{ x + iy \ \vert \  y  > 0\}$ be the complex upper half-plane equipped with the standard hyperbolic metric $ds^2 = \frac{1}{y^2}\left( dx^2 + dy^2 \right)$. The action of $\mathrm{SL}_2(\R)$ on $\mathbb{H}$ by fractional linear transformations is transitive and isometric. Furthermore, the stabilizer of $i \in \mathbb{H}$ is the special orthogonal group $\mathrm{SO}_2(\R)$, allowing us to identify $\mathbb{H}$ with the symmetric space $\mathrm{SL}_2(\R)/\mathrm{SO}_2(\R)$. Let $\pi \colon \mathrm{SL}_2 \to \mathrm{PSL}_2$ be the canonical isogeny. Given a discrete subgroup $\Gamma \subset \mathrm{SL}_2(\R)$ containing $\{ \pm I \}$ and having torsion-free
image $\pi(\Gamma) \subset \mathrm{PSL}_2(\R)$, the quotient $M = \Gamma \backslash \mathbb{H}$ is a Riemann surface. It is well-known that every compact Riemann surface of genus $> 1$ is of this form. However, in this subsection, we will be interested in more general Riemann surfaces, where $\Gamma$ is only assumed to be finitely generated and Zariski-dense. It was demonstrated in \cite{MacReid} that some properties of $M$ can be understood in terms of the associated quaternion algebra $A_{\Gamma}$, which is constructed as follows.

Let $\Gamma^{(2)}$ denote the subgroup generated by the squares of all elements, and let $A_{\Gamma}$ be the $\Q$-subalgebra of $M_2(\R)$ generated by $\Gamma^{(2)}$. One shows that $A_{\Gamma}$ is a quaternion algebra (although not necessarily a division algebra) with center $k_{\Gamma} = \Q(\mathrm{tr}\: \gamma \ \vert \ \gamma \in \Gamma^{(2)})$ (trace field) --- cf. \cite[Ch. 3]{MacReid}. If $\Gamma_1$ and $\Gamma_2$ are commensurable, then $A_{\Gamma_1} = A_{\Gamma_2}$; in other words, $A_{\Gamma}$ is an invariant of the commensurability class of $\Gamma$. Moreover, if $\Gamma$ is an {\it arithmetic} Fuchsian group, then $k_{\Gamma}$ is a number field and $A_{\Gamma}$ is {\it the} quaternion algebra involved in the description of $\Gamma$ (cf. \cite[\S 8.2]{MacReid}). It follows that if $\Gamma_1$ and $\Gamma_2$ are arithmetic and the algebras $A_{\Gamma_1}$ and $A_{\Gamma_2}$ are isomorphic, then $\Gamma_1$ is commensurable with a conjugate of $\Gamma_2$, and hence the corresponding Riemann surfaces are commensurable, i.e. have a common finite-sheeted cover. The algebra $A_{\Gamma}$ no longer determines the commensurability class of $\Gamma$ if the latter is not arithmetic, but it nevertheless remains an important invariant of the commensurability class.

In differential geometry, one attaches to a Riemannian manifold $M$ various spectra; in particular, the (weak) length spectrum $L(M)$ is defined as the set of the lengths of closed geodesics in $M$. Two Riemannian manifolds $M_1$ and $M_2$ are called {\it length-commensurable} if $\Q \cdot L(M_1) = \Q \cdot L(M_2)$. For arithmetic Riemann surfaces, length-commensurability implies commensurability (cf. \cite{Reid}). ``Most" Riemann surfaces, however, are {\it not} arithmetic, and their investigation presents many challenges. In those cases where we are unable to characterize the commensurability class in terms of the length spectrum, we would like to understand at least the properties of the associated quaternion algebras. As we will see momentarily, the fact that $M_1 = \Gamma_1 \backslash \mathbb{H}$ and $M_2 = \Gamma_2 \backslash \mathbb{H}$ are length-commensurable implies that the trace fields are equal: $k_{\Gamma_1} = k_{\Gamma_2}$, i.e. the corresponding algebras $A_{\Gamma_1}$ and $A_{\Gamma_2}$ have a common center. We can now state the following finiteness result for families of length-commensurable surfaces.
\begin{thm}\label{T:RS}
Let $M_i = \Gamma_i \backslash \mathbb{H}$ $(i \in I)$ be a family of length-commensurable Riemann surfaces, with
$\Gamma_i \subset \mathrm{SL}_2(\R)$ Zariski-dense. Then the associated quaternion algebras $A_{\Gamma_i}$ $(i \in I)$
belong to \emph{finitely many} isomorphism classes (over the common center).
\end{thm}
\begin{proof}
We first recall that closed geodesics in $M = \Gamma \backslash \mathbb{H}$ correspond to hyperbolic elements in $\Gamma$ different from $\pm I$, which are precisely the semisimple elements of $\Gamma$ having infinite order. Furthermore, the length of the closed geodesic $c_{\gamma}$ that corresponds to an element $\gamma \in \Gamma$ which is conjugate to $\left( \begin{array}{cc} t_{\gamma} & 0 \\ 0 & t_{\gamma}^{-1} \end{array}  \right)$ is given by
$$
\ell(c_{\gamma}) = \frac{2}{n_{\gamma}} \cdot \vert \log \vert t_{\gamma} \vert \vert,
$$
where $n_{\gamma}$ is a certain integer (``winding number"). It follows that
\begin{equation}\label{E:RLS}
\Q \cdot L(M) = \Q \cdot \{ \: \log \vert t_{\gamma} \vert \ \vert \ \gamma \in \Gamma \ \ \text{semisimple of infinite order} \: \}.
\end{equation}
Suppose now that two Riemann surfaces $M_i = \Gamma_i \backslash \mathbb{H}$ $(i = 1, 2)$ are length-commensurable. This means that for any
semisimple
element $\gamma_1 \in \Gamma_1$ of infinite order, there exists a semisimple element $\gamma_2 \in \Gamma_2$ of infinite order such that $\ell(c_{\gamma_1}) /
\ell(c_{\gamma_2}) \in \Q^{\times}$, and vice versa. This translates into the relation
$$
t_{\gamma_1}^{n_1} = t_{\gamma_2}^{n_2} \neq 1 \ \ \text{for some} \ \ n_1 , n_2 \in \mathbb{Z},
$$
which implies that the subgroups $\Gamma_1$ and $\Gamma_2$ are weakly commensurable. Then applying Theorem \ref{T:basic}(2), we conclude that their trace fields are the same: $k_{\Gamma_1} = k_{\Gamma_2}$ (we note that the definitions of the trace field given earlier and in the current subsection produce the same result). Fix one subgroup $\Gamma_1$ and set $k = k_{\Gamma_1}$; then $\Gamma_1^{(2)} \subset G_1(k)$, where $G_1 = \mathrm{SL}_{1 , A_{\Gamma_1}}$.  Since the group $\Gamma_1$ is finitely generated, the field $k$ is finitely generated, and we let $V$ denote a divisorial set of discrete valuations of $k$. Now let $\Gamma_i$ $(i \in I)$ be any other subgroup in the family. Then $k_{\Gamma_i} = k$ and $\Gamma_i^{(2)} \subset G_i(k)$, where $G_i = \mathrm{SL}_{1,A_{\Gamma_i}}$. Since $\Gamma_j$ for any $j \in I$ is finitely generated, we have $[\Gamma_j : \Gamma_j^{(2)}] < \infty$, so the weak commensurability of $\Gamma_1$ and $\Gamma_i$ implies that of $\Gamma_1^{(2)}$
and $\Gamma_i^{(2)}$. Thus, it follows from Theorem \ref{T:WC-GR} that there exists a finite subset $S(\Gamma_1) \subset V$ such that all $G_i$ $(i \in I)$ have good reduction at every $v \in V \setminus S(\Gamma_1)$. Here the groups $G_i$ are all of type $\textsf{A}_1$, and since the Finiteness Conjecture for forms with good reduction of this type over fields of characteristic $\neq 2$ has already been established (cf. \cite[Theorem 7.6]{RR-survey}), we conclude that they belong to finitely many isomorphism classes. Consequently, the quaternion algebras $A_{\Gamma_i}$ also belong to finitely many isomorphism classes.
\end{proof}

\vskip.1mm

Let us point out that
this theorem is one of the first examples of the use in differential geometry of new techniques from arithmetic geometry that involve the notion of  good reduction.

\section{The genus problem and good reduction for groups of type $\textsf{F}_4$}\label{S:F4}

In this section we will prove Theorems 1.10-1.13 that address some aspects of the genus problem and the Finiteness Conjecture for simple algebraic groups of type $\textsf{F}_4$. Our considerations will rely on properties of cohomological invariants, which we have assembled in Appendix 2. We therefore recommend that the reader consult Appendix 2 before continuing with this section.
We note that while Theorems 1.10-1.12 deal only with forms that have trivial $g_3$-invariant, Theorem 1.13 shows that truth of the Finiteness Conjecture yields
the properness of the map $\phi$, which is expected to classify all forms of type $\textsf{F}_4$ (cf. Appendix 2).

\vskip1mm

\noindent {\it Proof of Theorem 1.10.} Let $k_0$ be a number field, set $k = k_0(x)$, and let $V$ be the set of discrete valuations of $k$ associated with the closed points of $\mathbb{A}^1_{k_0}$. Let $G$ be a $k$-group of type $\textsf{F}_4$ that splits over a quadratic extension of $k$. We need to show that any $G' \in \gen_k(G)$ is $k$-isomorphic to $G$. According to Proposition A2.3, the group $G$ possesses a maximal $k$-torus $T$ that splits over a quadratic extension of $k$. Since $G$ and $G'$ lie in the same genus, $T$ is $k$-isomorphic to a maximal $k$-torus of $G'$, implying that $G'$ splits over the same quadratic extension of $k$ as $G$, and therefore the invariant $g_3(G')$ is trivial. Then it follows from Theorem A2.1 that in order to prove the isomorphism $G \simeq G'$ over $k$, it is enough to show that
$$
f_3(G) = f_3(G')  \ \ \text{and} \ \ f_5(G) = f_5(G').
$$
Recall that for any $i \geq 1$ and any $v \in V$, we have a residue map
$$
\rho^i_v \colon H^i(k , \mathbb{Z}/2\mathbb{Z}) \longrightarrow H^{i-1}(k^{(v)} , \mathbb{Z}/2\mathbb{Z}).
$$
These maps enable us to construct, in each degree $i \geq 1$, the following exact sequence that in the case $i = 2$ goes back to Faddeev:
\begin{equation}\label{E:Faddeev}
0 \to H^i(k_0 , \mathbb{Z}/2\mathbb{Z}) \longrightarrow H^i(k , \mathbb{Z}/2\mathbb{Z}) \stackrel{\oplus \rho^i_v}{\longrightarrow} \bigoplus_{v \in V} H^{i-1}(k^{(v)} , \mathbb{Z}/2\mathbb{Z}) \to 0
\end{equation}
(cf. \cite[Theorem 9.3]{GMS}). In order to prove that $f_i(G) = f_i(G')$ for $i = 3, 5$, we will prove that
\begin{equation}\label{E:res=}
\rho^i_v(f_i(G)) = \rho^i_v(f_i(G')) \ \ \text{for all} \ \ v \in V.
\end{equation}
Assuming this, we obtain  from (\ref{E:Faddeev}) that
$$
f_i(G') = f_i(G) + \alpha_i \ \ \text{for some} \ \ \alpha_i \in H^i(k_0 , \mathbb{Z}/2\mathbb{Z}).
$$
To complete the proof, one shows by a specialization argument that $\alpha_i = 0$ for $i = 3, 5$.  More precisely, the classes $f_i(G)$ and $f_i(G')$ are represented by symbols, and we can choose $x_0 \in k_0$ so that for the valuation $v_0$ of $k$ corresponding to $x - x_0$, all factors of these symbols are units with respect to $v_0.$ Then $k^{(v_0)} = k_0$, the groups $G$ and $G'$ have good reduction at $v_0$ (cf. Proposition A2.7 in Appendix 2), and the specializations of these symbols in $H^i(k_0 , \mathbb{Z}/2\mathbb{Z})$ coincide with the invariants $f_i(\underline{G}^{(v_0)})$ and $f_i((\underline{G}')^{(v_0)})$ of the corresponding reductions (see Theorem A2.5 and the remark at the end of subsection A2.2). Since $G$ and $G'$ are in the same genus, their reductions are also in the same genus --- see Theorem \ref{T:GoodReduction1}. But the genus of a group of type $\textsf{F}_4$ over a number field reduces to a single element by \cite[Theorem 7.5]{PR-WC},
so $\underline{G}^{(v_0)} \simeq (\underline{G}')^{(v_0)}$ and therefore $f_i(\underline{G}^{(v_0)}) = f_i((\underline{G}')^{(v_0)})$. On the other hand,
$$
f_i((\underline{G}')^{(v_0)}) = f_i(\underline{G}^{(v_0)}) + \alpha_i.
$$
Thus, $\alpha_i = 0$ and $f_i(G) = f_i(G')$, as required.

\vskip1mm

In order to prove (\ref{E:res=}), we will use the following.
\begin{lemma}\label{L:splitA}
Suppose $\mathscr{K}$ is an infinite field of characteristic $\neq$ 2 or 3. Let $G$ be a $\mathscr{K}$-group of type $\textsf{F}_4$ that splits over a quadratic extension of $\mathscr{K}$, and let $G' \in \gen_{\mathscr{K}}(G)$. Then every extension $\mathscr{L}/\mathscr{K}$ of degree $\leq 2$ that splits $f_3(G)$ (resp., $f_5(G)$) also splits $f_3(G')$ (resp., $f_5(G)$).
\end{lemma}
\begin{proof}
Let $\sigma$ be a generator of $\mathrm{Gal}(\mathscr{L}/\mathscr{K})$. Since $\mathscr{L}$ splits $f_3(G)$ (resp., $f_5(G)$), the group $G$ is $\mathscr{L}$-split (resp., $\mathscr{L}$-isotropic), cf. section A2.1 of Appendix 2. Let $B$ (resp., $P$) be an $\mathscr{L}$-defined Borel subgroup (resp., a proper $\mathscr{L}$-defined parabolic subgroup). Then the group $H = B \cap B^{\sigma}$ (resp., $H = P \cap P^{\sigma}$) is $\mathscr{K}$-defined. Let $T$ be a maximal $\mathscr{K}$-torus of $H$ that is also a maximal torus of $G$ (cf. \cite[14.13]{Borel}). Then $T$ is $\mathscr{L}$-split (resp., $\mathscr{L}$-isotropic). Being in the same genus as $G$, the group $G'$ contains a maximal $\mathscr{K}$-torus isomorphic to $T$. It follows that $G'$ is also $\mathscr{L}$-split (resp., $\mathscr{L}$-isotropic), and therefore $\mathscr{L}$ splits $f_3(G')$ (resp., $f_5(G')$).
\end{proof}

We will now prove (\ref{E:res=}) for $i = 3$. According to (\ref{E:invar}) in subsection A2.1, we can write
$$
f_3(G) = (a) \cup (b) \cup (c) \ \ \text{and} \ \ f_3(G') = (a') \cup (b') \cup (c'),
$$
where $(t) \in H^1(k , \Z/2Z)$ denotes the cohomology class corresponding to $t {k^{\times}}^2$ under the canonical isomorphism $H^1(k , \Z/2\Z) \simeq k^{\times}/{k^{\times}}^2$ (other notations are explained in A2.1).
If all values $v(a), \ldots , v(c')$ are even, then
\begin{equation}\label{E:zero}
\rho_v(f_3(G)) = 0 = \rho_v(f_3(G')).
\end{equation}
Next, suppose that the values $v(a), v(b)$, and $v(c)$ are all even, but among the values $v(a'), v(b'), v(c')$, there is at least one that is odd;
suppose, for example, that $v(c')$ is odd. We will apply Lemma \ref{L:splitA} with $\mathscr{K} = k_v$ and $\mathscr{L} = \mathscr{K}(\sqrt{c'})$,
noting that $G' \in \gen_{\mathscr{K}}(G)$ by Corollary \ref{C:gen}. Thus, since $\mathscr{L}$ splits $f_3(G')$ (over $\mathscr{K}$) by
Lemma \ref{L:splitA}, it also splits $f_3(G)$. Now,
it follows from Hensel's Lemma that the unramified cohomology group $H^i(\mathscr{K} , \mathbb{Z}/2\mathbb{Z})_v$, defined as the kernel of the residue map
$H^i(\mathscr{K} , \mathbb{Z}/2\mathbb{Z}) \to H^{i-1}(\mathscr{K}^{(v)} , \mathbb{Z}/2\mathbb{Z})$, is canonically isomorphic to $H^i(\mathscr{K}^{(v)} , \mathbb{Z}/2\mathbb{Z})$ (cf. \cite[Proposition 7.7]{GMS}). We may assume without loss of generality that $v(a) = v(b) = v(c) = 0$, so that the symbol $(a) \cup (b) \cup (c)$ is unramified. Since it splits over $\mathscr{L}$,  the symbol $(\bar{a}) \cup (\bar{b}) \cup (\bar{c})$ is trivial in $H^3(\mathscr{L}^{(v)} , \mathbb{Z}/2\mathbb{Z}) = H^3(\mathscr{K}^{(v)} , \mathbb{Z}/2\mathbb{Z})$. It follows that $(a) \cup (b) \cup (c)$ is trivial in $H^3(\mathscr{K} , \mathbb{Z}/2\mathbb{Z})$, which implies that $G$ is split. Then $G'$ is also split, and we again obtain (\ref{E:zero}).

It remains to consider the case where each set $\{v(a), v(b), v(c)\}$ and $\{v(a'), v(b'), v(c')\}$  contains at least one odd value. Without loss of generality, we may assume that
$$
v(a) = v(a') = v(b) = v(b') = 0 \ \ \text{and} \ \ v(c) = v(c') = 1.
$$
Then
$$
\rho_v(f_3(G)) = (\bar{a}) \cup (\bar{b}) \ \ \text{and} \ \ \rho_v(f_3(G')) = (\bar{a}') \cup (\bar{b}').
$$
Identifying $H^2(k^{(v)} , \mathbb{Z}/2\mathbb{Z}) = H^2(k^{(v)} , \mu_2)$ with the 2-torsion subgroup ${}_2\mathrm{Br}(k^{(v)})$ of the Brauer group, we see that the residues are represented, respectively, by the classes of the quaternion algebras
$$
\overline{D} = \left( \frac{\bar{a} , \bar{b}}{k^{(v)}}  \right)   \ \ \text{and} \ \ \overline{D}' = \left( \frac{\bar{a}' , \bar{b}'}{k^{(v)}}  \right).
$$
Since the genus of a quaternion division algebra over a number field reduces to a single element (cf. \cite{CRR-Bull}, \cite{RR-manuscr}), in order to prove that $\overline{D} \simeq \overline{D}'$, which would yield (\ref{E:res=}) in this case, one needs to prove that an extension $\ell/k^{(v)}$ of degree $\leq 2$ splits $\overline{D}$ if and only if it
splits $\overline{D}'$. Suppose that $\ell$ splits $\overline{D}$, and let $\mathscr{L}$ denote the unramified extension of $\mathscr{K} = k_v$ with residue field $\ell$. It follows from Hensel's Lemma that $\mathscr{L}$ splits $\displaystyle D = \left( \frac{a , b}{\mathscr{K}}  \right)$, and therefore also splits  $f_3(G)$. By Lemma \ref{L:splitA}, the extension $\mathscr{L}$ also splits $f_3(G')$, and therefore its residue field $\ell$ splits $\rho^3_v(f_3(G')) = (\bar{a}') \cup (\bar{b}')$, i.e. splits $\overline{D}'$. By symmetry, every extension $\ell/k^{(v)}$   of degree $\leq 2$ that splits $\overline{D}'$ also splits $\overline{D}$, completing the argument.

Next, we will prove (\ref{E:res=}) for $i = 5$. Write
$$
f_5(G) = (a) \cup (b) \cup (c) \cup (d) \cup (e) \ \ \text{and} \ \ f_5(G') = (a') \cup (b') \cup (c') \cup (d') \cup (e').
$$
As above, in the following two cases: 1) all values $v(a), \ldots , v(e')$ are even, and 2) all values $v(a), \ldots , v(e)$ are even and among the values
$v(a'), \ldots , v(e')$ there is at least one that is odd, one proves that
$$
\rho_v(f_5(G)) = 0 = \rho_v(f_5(G'))
$$
by repeating basically the same argument. The remaining situations reduce to the case where
$$
v(a) = \cdots = v(d) = v(a') = \cdots = v(d') = 0 \ \ \text{and} \ \ v(e) = v(e') = 1.
$$
Then
$$
\rho_v(f_5(G)) = (\bar{a}) \cup \cdots \cup (\bar{d}) \ \ \text{and} \ \ \rho_v(f_5(G')) = (\bar{a}') \cup \cdots \cup (\bar{d}').
$$
Since $k^{(v)}$ is a number field, by the Poitou-Tate theorem (cf. \cite[8.6.13(ii)]{NSW}, \cite[Ch. II, \S 6, Theorem B]{Serre-GC}), for any $j \geq 3$, we have an isomorphism
$$
H^j(k^{(v)} , \mathbb{Z}/2\mathbb{Z}) \longrightarrow \prod_{w \in W} H^j((k^{(v)})_w , \mathbb{Z}/2\mathbb{Z}),
$$
where $W$ is the set of all archimedean places of $k^{(v)}$. Furthermore, the group $H^j((k^{(v)})_w , \mathbb{Z}/2\mathbb{Z})$ is trivial if $w$ is complex, and has order 2 if $w$ is real; in the latter case, any symbol $(a_1) \cup \cdots \cup (a_j)$ with all $a_i$'s negative in $k^{(v)}$ gives the nontrivial element, while such a symbol in which at least one $a_i$ is positive gives the trivial element. Thus, if (\ref{E:res=}) fails, then there exists a real place $w$ of $k^{(v)}$ such that, say, $\rho_v(f_5(G))$ gives the trivial element and $\rho_v(f_5(G'))$ the nontrivial element of $H^4((k^{(v)})_w , \mathbb{Z}/2\mathbb{Z})$. This means that among $\bar{a}, \ldots , \bar{d}$ at least one element, say $\bar{d}$, is positive in $(k^{(v)})_w$, while all elements $\bar{a}', \ldots , \bar{d}'$ are negative. Consider the extension $\mathscr{L} = \mathscr{K}(\sqrt{d})$ of $\mathscr{K} = k$. Then $\mathscr{L}$ obviously splits $f_5(G)$. Let $\tilde{v}$ be an extension of $v$ to $\mathscr{L}$. Then
$$
\mathscr{L}^{(\tilde{v})} = k^{(v)}(\sqrt{\bar{d}}) \subset (k^{(v)})_w.
$$
It follows that $(\mathscr{L}^{(\tilde{v})})_{\tilde{w}}$ for $\tilde{w} \vert w$ does not split the image of $\rho_v(f_5(G'))$ in $H^4((k^{(v)})_w , \mathbb{Z}/2\mathbb{Z})$, hence $\mathscr{L}^{(\tilde{v})}$ does not split $\rho_v(f_5(G'))$, and $\mathscr{L}$ does not split $f_5(G')$. This contradicts Lemma
\ref{L:splitA}. \hfill $\Box$

\vskip1mm

\noindent {\it Proof of Theorem 1.12.} Let $G_0$ be the $k$-split group of type $\textsf{F}_4$. According to Proposition A2.3, one can view $\mathscr{I}$ as a subset of the set $H^1(k , G_0)_{g_3=0}$ of cohomology classes having trivial $g_3$-invariant, and then by Theorem A2.1, the restriction to $\mathscr{I}$ of the map
$$
\psi \colon H^1(k , G_0) \stackrel{(f_3 , f_5)}{\longrightarrow} H^3(k , \mathbb{Z}/2\mathbb{Z}) \times H^5(k , \mathbb{Z}/2\mathbb{Z})
$$
is injective. On the other hand, it follows from Theorem A2.6 that the $f_3$- and $f_5$-invariants of the forms from $\mathscr{I}$ are $V$-unramified, i.e.
$$
\psi(\mathscr{J}) \subset H^3(k , \mathbb{Z}/2\mathbb{Z})_V \times H^5(k , \mathbb{Z}/2\mathbb{Z})_V.
$$
Furthermore, Proposition 4.2 and Corollary 6.2 in \cite{CRR-Spinor} yield the finiteness of the groups $H^3(k , \mathbb{Z}/2\mathbb{Z})_V$ and $H^5(k , \mathbb{Z}/2\mathbb{Z})_V$ in the case where $k$ is a 2-dimensional global field, while Theorem 5.1(b) in \cite{RR-tori1} provides their finiteness over
a purely transcendental extension $k = k_0(x , y)$ of transcendence degree 2 of a number field $k_0$. In both cases, we obtain the finiteness of $\mathscr{I}$. \hfill $\Box$ 

\vskip.5mm

\noindent {\bf Remark 11.2.} Let $k$ be a 2-dimensional global field with a divisorial set of places $V$. Fix a Killing-Cartan type $\tau \in \{ \textsf{A}_{\ell} \}_{\ell = 1}^{\infty} \cup \cdots \cup \{ \textsf{G}_2\}$, and consider the set $\mathscr{Q}_V(\tau)$ of $k$-isomorphism classes of absolutely almost simple $k$-groups of type $\tau$ that split over a quadratic extension of $k$ and have good reduction at all $v \in V$. The results of \cite{CRR-Spinor} imply that $\mathscr{Q}_V(\tau)$ is finite if $\tau$ is one of the types $\textsf{A}_{\ell}$, $\textsf{B}_{\ell}$, $\textsf{C}_{\ell}$ $(\ell \geq 1)$ or $\textsf{G}_2$. Our Theorem 1.12 yields the finiteness of $\mathscr{Q}_V(\tau)$ for $\tau = \mathscr{F}_4$. The finiteness of $\mathscr{Q}_V(\tau)$ for $\tau = \textsf{D}_{\ell}$ $(\ell \geq 4)$ was recently established in \cite{IR2}. So, over 2-dimensional global fields, it remains to investigate the finiteness of $\mathscr{Q}_V(\tau)$ for $\tau \in \{ \textsf{E}_6, \textsf{E}_7, \textsf{E}_8\}$. Of course, this question can be viewed in the context of the more general problem of classifying absolutely almost simple algebraic groups that split over a quadratic extension of the base field, considered by Weisfeiler \cite{Weis-quadr}, in terms of certain quadratic/Hermitian forms. It should also be mentioned that the finiteness of $\mathscr{Q}_V(\tau)$ (for all $V$) has immediate consequences for the finiteness of the genus of absolutely almost simple algebraic $k$-groups of type $\tau$ that split over a quadratic extension of $k$ --- see the derivation of Theorem 1.11 from Theorem~1.12 above.

\vskip.5mm

\addtocounter{thm}{1}

\noindent {\it Proof of Theorem 1.11.} If $G$ splits over a quadratic extension $\ell/k$, then it has a maximal $k$-torus $T$ that splits over $\ell$. Then any $G' \in \gen_k(G)$ also splits over $\ell$. On the other hand, according to Corollary \ref{C:genus-GR}, there exists a divisorial set $V$ of discrete valuations of $k$ such that every group $G' \in \gen_k(G)$ has good reduction at all $v \in V$. The finiteness of $\gen_k(G)$ now follows from Theorem 1.12. $\Box$

\vskip1mm


The remainder of this section will be devoted to the proof of our last result concerning
the properness of the map $\phi$ (cf. subsection A2.1 of Appendix 2 for the relevant definitions).


\noindent {\it Proof of Theorem \ref{T:F4phi}.} We need to show that for any $k$-group $G$ of type $\mathsf{F}_4$, the fiber $\phi^{-1}(\phi(G))$ is finite. Choose a divisorial set of places $V$ of $k$ such that

\vskip1mm

$\bullet$ $\mathrm{char}\: k^{(v)} \neq 2, 3$ for all $v \in V$;

\vskip1mm

$\bullet$ $G$ has good reduction and the invariant $g_3(G)$ is unramified at all $v \in V$.

\vskip1mm

\noindent (We note that if $G$ has good reduction at $v$ and $\mathrm{char}\: k^{(v)} \neq 3$, then $g_3(G)$ is {\it automatically} unramified --- see Proposition A2.11, but this fact is not used in the argument.) Since we are assuming that the Finiteness Conjecture holds for all $k$-forms of type $\textsf{F}_4$ with respect to $V$, it is enough to show that every $G' \in \phi^{-1}(\phi(G))$ has good reduction at all $v \in V$. We will derive this fact from the following two propositions.

\begin{prop}\label{P:spl-field}
Assume that $\mathrm{char}\: k_v^{(v)} \neq 2$.
If a $k$-form $G'$ of type $\textsf{F}_4$ does not have good reduction at $v \in V$, then $G' \times_k k_v$ either splits over an unramified Galois extension $\ell/k_v$ of degree $2^a$ $(a \geq 0)$ or contains a maximal $k_v$-torus that is anisotropic over $k_v$ and splits over an unramified cubic Galois extension $\ell/k_v$.
\end{prop}

\vskip2mm

\begin{prop}\label{P:degree3}
Let $v$ be a discrete valuation of $k$ such that $\mathrm{char}\: k^{(v)} \neq 2, 3$, and let $G'$ be a simple $k$-group of type $\textsf{F}_4$ such that $G' \times_k k_v$ has a maximal $k_v$-torus that is anisotropic over $k_v$ and splits over an unramified cubic Galois extension $\ell/k_v$. If $G'$ does not have good reduction at $v$, then the invariant $g_3(G')$ is ramified at $v$.
\end{prop}



\vskip1mm

Granting these facts, we will now complete the proof of Theorem \ref{T:F4phi}. Assume that $G' \in \phi^{-1}(\phi(G))$ does not have good reduction at some $v \in V$. According to Proposition \ref{P:spl-field}, the group $G' \times_k k_v$ either splits over an unramified Galois extension $\ell/k_v$ of degree $2^a$ $(a \geq 0)$ or contains a maximal $k_v$-torus that is anisotropic over $k_v$ and splits over an unramified cubic Galois extension $\ell/k_v$.
In the first case, a standard restriction-corestriction argument yields $2^a \cdot g_3(G' \times_k k_v) = 0$, hence $g_3(G' \times_k k_v) = 0$ as we always have $3 \cdot g_3(G' \times_k k_v) = 0$. Thus, $G' \times_k k_v$ corresponds to a cohomology class  in $H^1(k_v , G_0)_{g_3 = 0}$.  But according to Theorem A2.1, the restriction of $\phi$ to $H^1(k_v , G_0)_{g_3=0}$ is injective. Thus $G' \times_k k_v \simeq G \times_k k_v$, contradicting the fact that $G$ has good reduction at $v$ and $G'$ does not.

In the second case, it follows from Proposition \ref{P:degree3} that the invariant $g_3(G' \times_k k_v)$ is ramified, while by construction the invariant $g_3(G \times_k k_v)$ is unramified. This contradicts the fact that $\phi(G) = \phi(G')$, hence $g_3(G \times_k k_v) = g_3(G' \times_k k_v)$. \hfill $\Box$

\vskip1mm

The proof of Proposition \ref{P:spl-field} relies heavily on Bruhat-Tits theory, for which we refer the reader to \cite{BT1}, \cite{BT2}, and \cite{BT3} (see also \cite{KalPr} for a modern exposition). As above, let $G_0$ be the $k$-split group of type $\mathsf{F}_4$, $T_0$ be a maximal $k_v$-split torus of $G_0$, and $\Phi = \Phi(G_0 , T_0)$ be the root system of $G_0$ with respect to $T_0$. Fix a system of simple roots $\Pi = \{\alpha_1, \ldots , \alpha_4\}$ and let $\tilde{\alpha}$ denote the maximal root.  We then have the following extended Dynkin diagram

\vskip.5mm

$$
\begin{picture}(120,10)
\put(00,00){\line(1,0){30}}
\put(30,00){\line(1,0){30}}
\put(60,1.1){\line(1,0){30}}
\put(60,-1.2){\line(1,0){30}}
\put(75,-2.5){$>$}
\put(90,00){\line(1,0){30}}
\put(00,0){\circle*{3}}
\put(30,0){\circle*{3}}
\put(60,0){\circle*{3}}
\put(90,0){\circle*{3}}
\put(120,0){\circle*{3}}
\put(-8,8){$-\tilde{\alpha}$}
\put(22,8){$\alpha_1$}
\put(52,8){$\alpha_2$}
\put(82,8){$\alpha_3$}
\put(112,8){$\alpha_4$}
\end{picture},
$$

\vskip2mm

\noindent whose set of vertices will be denoted $\tilde{\Pi}$. To each (non-empty) subset $\Omega \subset \tilde{\Pi}$, Bruhat-Tits theory associates a smooth group scheme $\mathscr{G}_{\Omega}$ over the valuation ring $\mathcal{O}_v$ of $k_v$ with the following properties:

\vskip2mm

$\bullet$ \parbox[t]{14cm}{the generic fiber $\mathscr{G}_{\Omega} \times_{\mathcal{O}_v} k_v$ is isomorphic to $G_0 \times_k k_v$;}

\vskip1mm

$\bullet$ \parbox[t]{14cm}{the closed fiber $\underline{\mathscr{G}}_{\Omega}^{(v)} = \mathscr{G}_{\Omega} \times_{\mathcal{O}_v} k_v^{(v)}$ is connected (because $G_0$ is simply connected);}

\vskip1mm

$\bullet$ \parbox[t]{14cm}{the unipotent radical $U_{\Omega}$ of $\underline{\mathscr{G}}_{\Omega}^{(v)}$ is defined and split over $k^{(v)}$, and $\underline{\mathscr{G}}_{\Omega}^{(v)}$ has a unique Levi subgroup $L_{\Omega}$ that contains the reduction $\underline{T}^{(v)}$, and hence is $k^{(v)}_v$-split.}

\vskip2mm

\noindent We note that the Dynkin diagram of the semisimple part of $L_{\Omega}$ is obtained from the extended Dynkin diagram of $G_0$ by deleting the vertices belonging to $\Omega$ and the edges having at least one endpoint in $\Omega$. Thus, $L_{\Omega}$ is a $k_v^{(v)}$-split reductive group with central torus of dimension $\vert \Omega \vert - 1$ and semisimple part (= commutator subgroup) $H_{\Omega}$ of rank $5 - \vert \Omega \vert$.

Since $\mathscr{G}_{\Omega}$ is smooth, the natural map
$$
\lambda_{\Omega} \colon H^1(\mathcal{O}_v , \mathscr{G}_{\Omega}) \longrightarrow H^1(k_v^{(v)} , \underline{\mathscr{G}}_{\Omega}^{(v)}), \ \ \xi \mapsto \bar{\xi},  $$
given by reduction is bijective by Hensel's Lemma (cf. \cite[Section 3.4, Lemme 2(2)]{BT3}). Furthermore, since $\underline{\mathscr{G}}_{\Omega}^{(v)}$ has a Levi decomposition and its unipotent radical is split, we have
\begin{equation}\label{E:levi}
H^1(k_v^{(v)} , \underline{\mathscr{G}}_{\Omega}^{(v)}) = H^1(k_v^{(v)} , L_{\Omega}).
\end{equation}
We say that a class $[\xi] \in H^1(k_v^{(v)} , L_{\Omega})$ is {\it anisotropic} if the semisimple part of the twisted group ${}_{\xi} L_{\Omega}$ is $k_v^{(v)}$-anisotropic. The set of all anisotropic classes will be denoted $H^1(k_v^{(v)} , L_{\Omega})_{\mathrm{an}}$, and its inverse image under $\lambda_{\Omega}$ will be denoted $H^1(\mathcal{O}_v , \mathscr{G}_{\Omega})_{\mathrm{an}}$.
\begin{thm}\label{T:BT}
{\rm (cf. \cite[Theorem 3.12]{BT3})} The natural map
$$
\coprod_{\Omega} H^1(\mathcal{O}_v , \mathscr{G}_{\Omega})_{\mathrm{an}} \to H^1(k_v , G_0)
$$
is a bijection.
\end{thm}

\noindent {\bf Remark 11.6.} It should be noted that this result was established in {\it loc. cit.} assuming that the residue field $k_v^{(v)}$ is {\it perfect}, which is not always the case in our situation. In this regard, we observe that since $\mathrm{char}\: k_v^{(v)} \neq 2, 3$, any simple algebraic $k_v$-group $G$ of type $\textsf{F}_4$ splits over the maximal unramified extension $k_v^{\mathrm{ur}}$. Indeed, in this case, the cohomological invariants $f_3, f_5$, and $g_3$ vanish, so the desired fact follows from the triviality of the kernel of $\phi$ (cf. Corollary A2.2). This implies that Bruhat-Tits theory in the sense of Prasad \cite{Prasad-unram} is available for $G$ over $k_v$. Another consequence is that
$$
H^1(k_v , G_0) = H^1(k_v^{\mathrm{ur}}/k_v , G_0).
$$
Then according to Theorem 3.8 in \cite{Gille-Gos}, the assertion of Theorem \ref{T:BT} remains valid for a $k_v$-split group $G_0$ of type $\textsf{F}_4$ whenever $\mathrm{char}\: k_v^{(v)} \neq 2, 3$.

\vskip2mm

{\it Proof of Proposition \ref{P:spl-field}}. Let $G'$ be a simple $k_v$-group of type $\textsf{F}_4$ that does not have good reduction at $v$. Write $G'$ as a twist ${}_{\xi'} G_0$ for some cocycle $\xi' \in Z^1(k_v , G_0)$. According to Theorem \ref{T:BT} and subsequent remarks, there exists a subset $\Omega \subset \tilde{\Pi}$ such that the class $[\xi']$ is the image under the natural map $\nu_{\Omega} \colon H^1(\mathcal{O}_v , \mathscr{G}_{\Omega}) \to H^1(k_v , G_0)$ of some class  from $H^1(\mathcal{O}_v , \mathscr{G}_{\Omega})_{\mathrm{an}}$, and we still use $\xi'$ to denote a cocycle representing this class.
We note that for $\Omega = \{ -\tilde{\alpha} \}$, the group $\mathscr{G}_{\Omega}$ coincides with the split $\mathcal{O}_v$-group scheme $\mathscr{G}_0$ of type $\textsf{F}_4$, and the classes in the image of $\nu_{\Omega}$ correspond precisely to the $k_v$-forms of type $\textsf{F}_4$ that have good reduction. Since by assumption $G'$ does {\it not} have good reduction, we conclude that $\Omega \neq \{ -\tilde{\alpha} \}$. So, our task is to show that {\it in all other cases}, the group $G'$ either splits over an unramified Galois extension $\ell/k_v$ of degree $2^a$ or contains a maximal $k_v$-torus that is anisotropic over $k_v$ and splits over an unramified cubic Galois extension $\ell/k_v$. Viewing $\xi'$ as an element of $Z^1(\mathcal{O}_v , \mathscr{G}_{\Omega})$, we can consider the twist $\mathscr{G}' = {}_{\xi'} \mathscr{G}_{\Omega}$, which is a smooth group scheme over $\mathcal{O}_v$ with generic fiber $G'$ and closed fiber ${}_{\bar{\xi}'} \underline{\mathscr{G}}_{\Omega}^{(v)}$ in the above notations. We recall that if $\ell$ is an unramified extension $k_v$ with residue field $\ell^{(v)}$, then every $\ell^{(v)}$-split torus of $\mathscr{G}' \times_{\mathcal{O}_v} \ell^{(v)}$ can be lifted to a split torus of $\mathscr{G} \times_{\mathcal{O}_v} \mathcal{O}(\ell)$, where $\mathcal{O}(\ell)$ is the valuation ring of $\ell$ (see \cite[Corollary B.3.5]{Conrad}), implying that $\mathrm{rk}_{\ell}\: G' \geq \mathrm{rk}_{\ell^{(v)}}\, (\mathscr{G}' \times_{\mathcal{O}_v} \ell^{(v)})$.

In view of (\ref{E:levi}), we may assume that $\bar{\xi}'$ has values in $L_{\Omega}$, so that we can consider the twisted groups $L'_{\Omega} := {}_{\bar{\xi}'} L_{\Omega}$ and $H'_{\Omega} := {}_{\bar{\xi}'} H_{\Omega}$. We observe that all absolutely simple components of $H'_{\Omega}$ are defined over $k_v^{(v)}$ and are inner forms. Since the central tori in $L_{\Omega}$ and $L'_{\Omega}$ are $k_v^{(v)}$-isomorphic, hence split, it is clear that if $\vert \Omega \vert \geq 3$, then $\mathrm{rk}_{k_v}\: G' \geq \mathrm{rk}_{k_v^{(v)}} \: \mathscr{G}' \geq 2$. In this case, it follows from the classification of forms of type $\textsf{F}_4$  (cf. \cite{Tits}) that $G'$ is $k_v$-split, hence has good reduction, a contradiction. Now, if $\vert \Omega \vert = 2$, then $L'_{\Omega}$ has a 1-dimensional central $k_v^{(v)}$-split torus. If  $H'_{\Omega}$ is not absolutely almost simple then it has a component of type $\textsf{A}_1$ which splits over a separable quadratic extension $\bar{\ell}/k_v^{(v)}$. Letting $\ell$ denote the unramified extension of $k_v$ with residue field $\ell^{(v)} = \bar{\ell}$,  we find as above that $\mathrm{rk}_{\ell}\: G' \geq 2$, and hence $G'$ splits over $\ell$.  Now, let $H'_{\Omega}$ be absolutely almost simple. Then it is of one of the following types: $\textsf{A}_3$, $\textsf{B}_3$, or $\textsf{C}_3$. Since every simple group of type $\textsf{B}_r$ is isogenous to the special orthogonal group of a nondegenerate quadratic form (cf. \cite[Proposition 2.20]{PR}), it obviously splits over a Galois extension of the form $\bar{\ell} = k_v^{(v)}(\sqrt{a_1}, \ldots , \sqrt{a_t})$. Furthermore, every simple group of type $\textsf{C}_r$ is isogenous to the special unitary group of a nondegenerate hermitian form over a central division algebra with a symplectic involution of the first kind (cf. \cite[Proposition 2.19]{PR}). By Merkurjev's theorem \cite{Merk}, the algebra splits over a Galois extension of the same shape $\bar{\ell} = k_v^{(v)}(\sqrt{a_1}, \ldots , \sqrt{a_t})$, and this extension also splits the group.
Picking such an extension if  $H'_{\Omega}$ has type $\textsf{B}_3$ or $\textsf{C}_3$, and letting $\ell$ be the unramified extension of $k_v$ with residue field $\ell^{(v)} = \bar{\ell}$ (which is automatically a Galois extension of $k_v$ of degree $2^a$) we will have that $\mathrm{rk}_{\ell}\: G' \geq 3$, implying that $G'$ splits over $\ell$.

In the remaining case, $H'_{\Omega}$ is an inner form of type $\textsf{A}_3$, i.e. a group of the form $\mathrm{SL}_{1,A}$ for some central simple $k_v^{(v)}$-algebra $A$ of degree 4.  Then it follows from the theorem of Merkurjev-Suslin \cite{MerkSus} that this group splits over a Galois extension $\bar{\ell}/k_v^{(v)}$ of degree $2^a$ (note that one needs to adjoin $\sqrt{-1}$ to the base field before applying the Merkurjev-Suslin theorem), and then arguing as above we find that the unramified extension $\ell$ of $k_v$ with residue field $\ell^{(v)} = \bar{\ell}$ is as required.

Finally, we consider the case $\vert \Omega \vert = 1$. Here the possible types of $H_{\Omega}$ (with the exception of $\textsf{F}_4$ itself are $\textsf{A}_1 \times \textsf{C}_3$, $\textsf{A}_2 \times \textsf{A}_2$, and $\textsf{A}_1 \times \textsf{B}_3$. The types $\textsf{A}_1 \times \textsf{C}_3$ and $\textsf{A}_1 \times \textsf{B}_3$ are handled just as above, so it remains to consider the type $\textsf{A}_2 \times \textsf{A}_2$. Let $H_1$ and $H_2$ be the almost simple components of $H'_{\Omega}$; then $H_i$ is isogenous to a group of the form $\mathrm{SL}_{1 , D_i}$, where $D_i$ is a central simple $k_v^{(v)}$-algebra of degree 3 over $k_v^{(v)}$. Picking a separable maximal subfield $\bar{\ell}$ of $D_1$ and arguing as above, we see that $G'$ splits over the unramified extension $\ell$ of $k_v$ with residue field $\bar{\ell}$.  We note that in this case, the group $G'$ is automatically $k_v$-anisotropic. Indeed, it cannot be $k_v$-split because of bad reduction, so the only other $k_v$-isotropic possibility would have a simple group $R$ of type $\textsf{C}_3$ as its anisotropic kernel. In the situation at hand, $G'$ splits over an extension $\ell/k_v$ of degree $3$, and then $R$ must also split over $\ell$. But for the groups of type $\textsf{C}$, this implies that $R$ splits over $k_v$, a contradiction.

Next,
by Wedderburn's theorem (cf. \cite[Theorem 19.2]{Invol}), $D_1$ contains a maximal subfield $\bar{\ell}$ that is a cubic Galois extension of the center. We note that, as follows from Hilbert's Theorem 90, the $k_v^{(v)}$-tori
$\bar{\mathscr{S}} = \mathrm{R}^{(1)}_{\bar{\ell}/k_v^{(v)}}(\mathbb{G}_m)$ and $\bar{\mathscr{S}}/\mu_3 = \mathrm{R}_{\bar{\ell}/k^{(v)}}(\mathbb{G}_m)/\mathbb{G}_m$ are isomorphic. So, whether $H_1$ is simply connected or adjoint, it always contains a torus isomorphic to $\bar{\mathscr{S}}$ as a maximal torus, and we keep the notation $\bar{\mathscr{S}}$ for this torus. By \cite[Corollary B.3.5]{Conrad}, the torus $\bar{\mathscr{S}}$ lifts to a torus $\mathscr{S}$ of $\mathscr{G}_{\Omega}$. Let $S \subset G'$ be the generic fiber of $\mathscr{S}$. Then $S$ splits over the unramified extension $\ell$ of $k_v$ with residue field $\bar{\ell}$. It follows that $\mathrm{rk}_{\ell}\: G' \geq 2$, and therefore $G'$ actually splits over $\ell$. Now, we consider the centralizer $C =
C_{G'}(S)$ and write it as an almost direct product $SR$, where $R$ is a reductive group of rank $2$. The complete list of possibilities for $R$ is as follows:

\vskip.5mm

$\bullet$ a 2-dimensional torus;

$\bullet$ an almost direct product of a 1-dimensional torus and an almost simple group of type $\textsf{A}_1$;

$\bullet$ a semisimple group of one of the following types: $\textsf{A}_1 \times \textsf{A}_1$, $\textsf{A}_2$, $\textsf{B}_2$, or $\textsf{C}_2$.

\vskip1.5mm

\noindent Since $S$ and $G'$ split over $\ell$, so does $R$. Taking into account that $G'$ is $k_v$-anisotropic and going through the above list, we find
that $R$ can only be either a 2-dimensional torus or an inner form of type $\textsf{A}_2$. In the first case, the almost direct product $SR$ is a required torus that is anisotropic over $k_v$ and splits over $\ell$. In the second case, $R$ is $k_v$-isogenous to a group of the form $\mathrm{SL}_{1 , D}$, where $D$ is a cubic division algebra (whose residue must coincide with $D_2$). Since $\ell$ splits $R$, it is isomorphic to a maximal subfield of $D$, implying that $R$ contains a maximal $k_v$-torus $S'$ of the form $\mathrm{R}_{\ell/k}^{(1)}(\mathbb{G}_m)$. Then $SS'$ is a required maximal torus. \hfill $\Box$

\vskip2mm

Next, in preparation for the proof of Proposition \ref{P:degree3}, we will set up the necessary notations and state one result (Proposition \ref{P:F4-1}) that will be proved in subsection A2.3 of Appendix 2. Let $G_0$ be a simple $k$-split group of type $\textsf{F}_4$, and let $\ell/k$ be a cubic Galois extension. We will now construct a special maximal $k$-torus
$T$ of $G_0$ whose cohomology classes yield all $k$-forms of $G_0$ that contain a maximal $k$-torus that is anisotropic over $k$ and splits over $\ell$. Fix a maximal $k$-split torus $T_0$ of $G_0$, and let $\Phi = \Phi(G_0 , T_0)$ be the corresponding root system. We will continue using the above labeling of the roots of the extended Dynkin diagram.  Then the subsets $\{-\tilde{\alpha} , \alpha_1 \}$ and $\{ \alpha_3 , \alpha_4 \}$ correspond to $k$-subgroups $R_1$ and $R_2$ of $G_0$ that are isomorphic to $\mathrm{SL}_3$ and whose intersection is $\mu_3$ (which is the center of both
$R_1$ and $R_2$). Let $S = \mathrm{R}^{(1)}_{\ell/k}(\mathbb{G}_m)$. We consider the embeddings $\iota_i \colon S \hookrightarrow R_i$ $(i = 1, 2)$ afforded by the regular representation of $\ell$ over $k$, and let
$$
T = \iota_1(S) \cdot \iota_2(S)  \subset R_1 \cdot R_2.
$$
We have an exact sequence
$$
1 \to \mu_3 \stackrel{\alpha}{\longrightarrow} S \times S \stackrel{\beta}{\longrightarrow} T \to 1,
$$
where $\alpha(s) = (s , s^{-1})$ and $\beta(s_1 , s_2) = \iota_1(s_1)\iota_2(s_2)$. Set
$$
\tilde{S}_1 = \{(s , s^{-1}) \ \vert \ s \in S \} \ \ \text{and} \ \ \tilde{S}_2 = \{ (s , 1) \ \vert \ s \in S \},
$$
and let $S_i = \beta(\tilde{S}_i)$ for $i = 1, 2$. Then $T = S_1 \times S_2$, so we can consider the homomorphisms $\gamma_i \colon H^1(k , T) \to H^1(k , S_i)$ $(i = 1, 2)$ given by the projections. Finally, let $\delta \colon H^1(k , S_1) \to H^2(k , \mu_3)$ be the coboundary map associated with the exact sequence
$$
1 \to \mu_3 \longrightarrow \tilde{S}_1 \stackrel{\beta}{\longrightarrow} S_1 \to 1.
$$

\addtocounter{thm}{1}

\begin{prop}\label{P:F4-1}
{\rm (1)} Every $k$-group $G$ of type $\textsf{F}_4$ that contains a maximal $k$-torus that is $k$-anisotropic and splits over $\ell$ is of the form $G = {}_{\xi}G_0$ with $\xi = \nu(\zeta)$ for some $\zeta \in Z^1(k , T)$, where $\nu \colon Z^1(k , T) \to Z^1(k , G_0)$ is the natural map.

\vskip1mm

\noindent {\rm (2)} In the above notations, $g_3(G) = \delta(\gamma_1([\zeta])) \cup \gamma_2([\zeta])$.
\end{prop}

The proof will be given in section A2.3.

\vskip2mm

{\it Proof of Proposition \ref{P:degree3}.} Applying Proposition \ref{P:F4-1} over $k_v$, we can write $G' \times_k k_v = {}_{\xi} (G_0 \times_k k_v)$, where $\xi = \nu(\zeta)$ and $\zeta \in Z^1(k_v , T)$. Suppose that $G'$ does not have good reduction at $v$. We need to show that the $g_3$-invariant $g_3(G' \times_k k_v) \in H^3(k_v , \mathbb{Z}/3\mathbb{Z})$ is ramified. By Proposition \ref{P:F4-1}, we have $g_3(G' \times_k k_v) = \delta(\gamma_1([\zeta])) \cup \gamma_2([\zeta])$. If $\delta(\gamma_1([\zeta]))$ is unramified, then by Lemma A2.10 the class of $\gamma_1([\zeta])$ in $H^1(k_v , S_1)$ lies in the image of the map\footnote{We let $\mathscr{S}_1$, $\mathscr{S}_2$, and $\mathscr{T}$ denote the $\mathcal{O}_v$-tori with generic fibers $S_1$, $S_2$, and $T$, respectively.} $H^1(\mathcal{O}_v , \mathscr{S}_1) \to H^1(k_v , S_1)$. If, in addition, $$\gamma_2([\zeta]) \in H^1(k_v , S_2) = k_v^{\times} / N_{\ell_w/k_v}(\ell_w^{\times})$$ is unramified, then it lies in the image of the map $H^1(\mathcal{O}_v , \mathscr{S}_2) \to H^1(k_v , S_2)$. This implies that there exists $\zeta_0 \in Z^1(\mathcal{O}_v , \mathscr{T})$ such that the image of $[\zeta_0]$ in $H^1(k_v , T)$ coincides with $[\zeta]$.
Set $\xi_0 = \nu_0(\zeta_0)$  where $\nu_0 H^1(\mathcal{O}_v , \mathscr{T}) \to  H^1(\mathcal{O}_v, \mathscr{G}_0)$ is the natural map. Then the twist ${}_{\xi_0} \mathscr{G}_0$ is a reductive $\mathcal{O}_v$-group scheme with generic fiber $G' \times_k k_v$, contradicting our assumptions that $G'$ does not have good reduction at $v$.

Next, suppose that $\delta(\gamma_1([\zeta]))$ is unramified but $\gamma_2([\zeta])$ is ramified. Then the residue of the invariant $g_3(G' \times_k k_v)$ equals $$s \cdot (\text{image of} \ \delta(\gamma_1([\zeta])) \ \text{in} \ {}_3 \mathrm{Br}(k^{(v)})), \ \  \text{where} \ \  s = 1, 2.$$ This element is trivial if and only if the element $\delta(\gamma_1([\zeta]))$ is trivial, in which case $g_3(G' \times_k k_v)$ is trivial. Since the $f_3$- and $f_5$-invariants of $G' \times_k k_v$ are also trivial as $G'$ splits over a cubic extension, we conclude that $G' \times_k k_v$ splits, hence has good reduction, a contradiction.

Suppose now that $\delta(\gamma_1([\zeta]))$ is ramified. It obviously splits over $\ell$, and therefore is represented by a cyclic cubic algebra of the form $(\ell/k_v , a)$ with $a \in k_v^{\times}$. Since $\ell/k_v$ is unramified, the valuation $v(a)$ is either 1 or 2. Then $g_3(G'\times_k k_v)$ can be written in the form $g_3(G'\times_k k_v) = \delta(\gamma_1([\xi])) \cup (c)$, where $c$ is a unit, so its residue is $s \cdot \chi_{{\ell}^{(v)}} \cup (\bar{c})$, where $\chi_{{\ell}^{(v)}}$ is the character corresponding to the residue field $\ell^{(v)}$ of $\ell$ and $s = 1, 2$. If this element is trivial, then the element $\chi_{\ell} \cup (c)$ is trivial, and therefore $g_3(G' \times_k k_v)$ is trivial. So, again $G' \times_k k_v$ splits, hence  $G'$ has good reduction at $v$, a contradiction. \hfill $\Box$

\vskip7mm

\centerline{{\large  Appendix 1. {\sc On a result of A.~Klyachko}}}

\vskip3mm

\noindent {\bf A1.1. Proof of Theorem \ref{T:Klyachko10}.} We will freely use the notations introduced in \S\,\ref{S:Klyachko}. Let $( \: , \: )$ be a $W(\Phi)$-invariant inner product on $V$.
As usual, for $\alpha \in
\Phi$, we define the {\it dual root} $\displaystyle \alpha^{\vee} =
\frac{2 \alpha}{(\alpha , \alpha)}$, and let $\Phi^{\vee} = \{
\alpha^{\vee} \: \vert \: \alpha \in \Phi \}$ denote the
corresponding {\it dual root system}. Given a subgroup $\Gamma
\subset \mathrm{Aut}(\Phi)$ containing $W(\Phi)$ and a $\Gamma$-invariant
lattice $M$ in $V$, we can write the inflation-restriction exact sequence
$$
0 \to H^1(\Gamma/W(\Phi) , M^{W(\Phi)}) \longrightarrow H^1(\Gamma , M) \longrightarrow H^1(W(\Phi) , M)^{\Gamma/W(\Phi)} \longrightarrow
H^2(\Gamma/W(\Phi) , M^{W(\Phi)}).
$$
Since $M^{W(\Phi)} = 0$, we obtain an isomorphism
$$
H^1(\Gamma , M) \, \simeq \, H^1(W(\Phi) , M)^{\Gamma/W(\Phi)}.
$$
It follows that it is enough to prove both assertions of the theorem for $\Gamma = W(\Phi)$, which we will assume to be the case throughout the rest of the argument.

Now, fix a system of simple roots $\Pi \subset \Phi$. It is
well-known that $\Gamma$ is a Coxeter group; more precisely,
it is generated by reflections $s_{\alpha}$ for $\alpha \in \Pi$ (where $s_{\alpha}(v)
= v - (\alpha^{\vee} , v)\alpha$ for $v \in V$) and is defined by
the following relations
\begin{equation}\tag{A.1}\label{E:rel}
s_{\alpha}^2 =1 \ \ , \ \ (s_{\alpha} s_{\beta})^{n_{\alpha ,
\beta}} = 1 \ \ \ \text{for} \ \ \alpha \: , \: \beta \in \Pi, \ \
\alpha \neq \beta,
\end{equation}
where $n_{\alpha , \beta}$  is
the order of the product $s_{\alpha} s_{\beta}$ in $\Gamma$. For a
$\Gamma$-invariant lattice $M \subset V$, we let $Z^1(\Gamma , M)$
and $B^1(\Gamma , M)$ denote the corresponding groups of cocycles
and coboundaries, respectively, and for a function $f \colon \Gamma \to M$ set
$$
\mu(f) = (f(s_{\alpha})) \in \bigoplus_{\alpha \in \Pi} M.
$$

\vskip1mm

\noindent {\bf Lemma A1.1.} {\it The map $\mu$ sets up an isomorphism between $Z^1(\Gamma , M)$ and
$\displaystyle \bigoplus_{\alpha \in \Pi} (\mathbb{Q} \alpha \bigcap
M).$ Under this isomorphism, $B^1(\Gamma , M)$ corresponds to
$$
\{ ((\alpha^{\vee} , m)\alpha) \: \vert \: m \in M \}.
$$}

\begin{proof}
Any $f \in Z^1(\Gamma , M)$ is completely determined by its values
on any generating set, making $\mu$ injective on $Z^1(\Gamma , M)$.
We have
$$
f(1) = 0 = f(s_{\alpha}) + s_{\alpha} f(s_{\alpha}),
$$
implying that $f(s_{\alpha}) \in \mathbb{Q} \alpha$. So,
$\displaystyle \mu(Z^1(\Gamma , M)) \subset \bigoplus_{\alpha \in \Pi}
(\mathbb{Q}\alpha \bigcap M)$. To prove that this inclusion is
actually an equality, take any $(m_{\alpha})$ in the right-hand
side, and let $\tilde{\Gamma}$ be the free group on $s_{\alpha}$,
$\alpha \in \Pi$. Recall that given a group $\Delta$ and a
$\Delta$-module $T$, a function $\varphi \colon \Delta \to T$ is a
1-cocycle iff the map $$\varphi^+ \colon \Delta \to T \rtimes
\Delta, \ \ \ \ x \mapsto (\varphi(x) \: , \: x)$$ is a group
homomorphism. This allows us to define $\tilde{f} \in
Z^1(\tilde{\Gamma} , M)$ by $\tilde{f}(s_{\alpha}) =
m_{\alpha}$, and observe that $\tilde{f}$ descends to $f \in
Z^1(\Gamma , M)$ satisfying $f(s_{\alpha}) = m_{\alpha}$ if and only
if $\tilde{f}$ vanishes on the relations (\ref{E:rel}). The equation
$\tilde{f}(s_{\alpha}^2) = 0$ immediately follows from the fact that
$m_{\alpha} \in \mathbb{Q}\alpha$. Furthermore,
$$
\tilde{f}((s_{\alpha} s_{\beta})^{n_{\alpha , \beta}}) = (1 +
(s_{\alpha}s_{\beta}) + \cdots + (s_{\alpha} s_{\beta})^{n_{\alpha ,
\beta} - 1}) \tilde{f}(s_{\alpha} s_{\beta}).
$$
Note that the right-hand side is fixed by $s_{\alpha} s_{\beta}$. On
the other hand, since $$\tilde{f}(s_{\alpha}s_{\beta}) = m_{\alpha}
+ s_{\alpha} m_{\beta} \in \mathbb{Q}\alpha + \mathbb{Q}\beta,$$ it
belongs to $\mathbb{Q}\alpha + \mathbb{Q}\beta$. But
$s_{\alpha}s_{\beta}$ is a nontrivial rotation of this 2-dimensional vector space,
and therefore has no nonzero fixed vectors. So,
$\tilde{f}((s_{\alpha} s_{\beta})^{n_{\alpha , \beta}}) = 0$,
completing the proof of the first assertion. The second assertion
about $\mu(B^1(\Gamma , M))$ follows from the formula for
$s_{\alpha}$.
\end{proof}

\vskip.5mm

It is now easy to complete the proof of Theorem
\ref{T:Klyachko10}. Set $M = P(\Phi)$, which, by definition, is the dual
lattice of the lattice $Q(\Phi^{\vee})$ generated by the dual root system.
Thus, $M$ has a basis consisting of {\it weights} $\omega_{\beta}$ $(\beta
\in \Pi)$ satisfying $(\alpha^{\vee} , \omega_{\beta}) = \delta_{\alpha\beta}$
(Kronecker delta).  The crucial observation is
that unless $\Phi$ is of type $\textsf{A}_1$ or $\textsf{C}_{\ell}$ $(\ell \geq 2)$, for
any $\alpha \in \Pi$ we have
\begin{equation}\tag{A.2}\label{E:roots}
\mathbb{Q} \alpha \bigcap M = \mathbb{Z} \alpha.
\end{equation}
Indeed,  this is obvious if $P(\Phi) = Q(\Phi)$ since $\Pi$ is always a basis of $Q(\Phi)$.
This proves (\ref{E:roots})  if $\Phi$ is one of the types
$\mathsf{E}_8$, $\mathsf{F}_4$, or $\mathsf{G}_2$. On the other hand, if $\Phi$ has rank $> 1$
with all the roots of the same length (equal to $\sqrt{2}$), then for a given $\alpha \in \Pi$, one can
pick $\beta \in \Pi$ so that
$$(\alpha ,
\beta) = -1 = (\alpha , \beta^{\vee}).
$$
If $\lambda \alpha \in M$, then $-\lambda = (\lambda \alpha , \beta^{\vee}) \in \mathbb{Z}$, proving
(\ref{E:roots}) in this case. Apart from  types $\textsf{A}_1$ and $\textsf{C}_{\ell}$ $(\ell \geq 2)$,
these two cases cover all types except $\textsf{B}_{\ell}$ with $\ell >
2$. For this remaining type, (\ref{E:roots}) follows immediately
from the description of $P(\Phi)$ given in \cite[Table
II]{Bour}. Thus, if $\Phi$ is not of type $\textsf{A}_1$ or
$\textsf{C}_{\ell}$ then
\begin{equation}\tag{A.3}\label{E:roots1}
\bigoplus_{\alpha \in \Pi} (\mathbb{Q} \alpha \bigcap M) =
\bigoplus_{\alpha \in \Pi} \mathbb{Z}\alpha.
\end{equation}
For types $\textsf{A}_1$ and $\textsf{C}_{\ell}$, one easily finds, using
\cite[Tables I and II]{Bour}, that the left-hand side of
(\ref{E:roots1}) contains the right-hand side as a subgroup of index
two. On the other hand, in all cases
\begin{equation}\tag{A.4}\label{E:roots2}
\{((\alpha^{\vee} , m)\alpha) \: \vert \: m \in M \} =
\bigoplus_{\alpha \in \Pi} \mathbb{Z} \alpha.
\end{equation}
Indeed, the inclusion $\subset$ follows from the definition of
$P(\Phi)$.
On the other hand, for the weight $m = \omega_{\beta}$ with $\beta \in \Pi$, the element
$((\alpha^{\vee} , m)\alpha)_{\alpha \in \Pi}$ has $\beta$ in the $\beta$-slot
and $0$ in all other slots. Thus, the left-hand side of (\ref{E:roots2}) contains
all $\beta \in \Pi$, and (\ref{E:roots2}) follows.
Comparing these computations
with Lemma A1.1, we obtain Theorem \ref{T:Klyachko10}. \hfill $\Box$

\vskip.5mm

\noindent {\bf Remark A1.2.} Since $\textsf{B}_2 = \textsf{C}_2$, the type $\textsf{B}_2$ in Theorem \ref{T:Klyachko10} should be treated
as exceptional along with the types $\textsf{A}_1$ and $\textsf{C}_{\ell}$ $(\ell \geq 2)$.

\vskip1mm

\noindent {\bf A1.2. On the mistake in \cite{Klyachko}.} In \cite{Klyachko}, Klyachko made a claim (p. 73,
item c)) that for any subgroup $\Gamma \subset \mathrm{Aut}(\Phi)$
that contains $W(\Phi)$, and any $\Gamma$-invariant lattice $M
\subset V$ satisfying $Q(\Phi) \subset M \subset P(\Phi)$, one has
$H^1(\Gamma , M) = 0$ {\it except} in the following three cases, where
$\Gamma$ coincides with $W(\Phi) = \mathrm{Aut}(\Phi)$: (1) $\Phi =
\textsf{A}_1$; (2) $\Phi = \textsf{C}_{\ell}$ and $M = P(\Phi)$; and (3)
$\Phi = \textsf{B}_{\ell}$ and $M = Q(\Phi)$, where $H^1(\Gamma , M) =
\mathbb{Z}/2\mathbb{Z}$. As we already mentioned, this result is {\it false}
as stated. We will now indicate where the argument in
\cite{Klyachko} fails for $M = Q(\Phi)$, and then present a counter-example
based on explicit computations.

For this $M$, we immediately have
$$
\bigoplus_{\alpha \in \Pi} (\mathbb{Q}\alpha \bigcap M) \ = \
\bigoplus_{\alpha \in \Pi} \mathbb{Z}\alpha.
$$
Then Klyachko observes that if  $\Phi$ is of type different from
$\textsf{A}_1 , \textsf{B}_{\ell}$, then for each $\alpha \in \Pi$, the
g.c.d. of the integers $(\alpha^{\vee} , \beta)$ as $\beta$ varies
in $\Pi$, equals 1, and concludes from this that
\begin{equation}\label{E:mistake}
\{ ((\alpha^{\vee} , m) \alpha) \: \vert \: m \in M \} =
\bigoplus_{\alpha \in \Pi} \mathbb{Z}\alpha;
\end{equation}
in view of Lemma A.1 this would prove that $H^1(\Gamma , M) = 0$. The problem is
that this
``argument" proves only that for each $\alpha \in \Pi$, the
projection of the left-hand side of (\ref{E:mistake}) to
$\mathbb{Z}\alpha$ is surjective, but does not fully justify
(\ref{E:mistake}). In fact, if we canonically identify the
right-hand side of (\ref{E:mistake}) with $\Z^{\ell}$, then the
left-hand side gets identified with the submodule spanned by the
columns of the Cartan matrix of $\Phi$. It follows that
$H^1(W(\Phi) , Q(\Phi))$ is always isomorphic to the quotient
$P(\Phi)/Q(\Phi)$, and in particular is nontrivial unless $P(\Phi) =
Q(\Phi)$ (in other words,  $\Phi$ has one of the following types $\textsf{E}_8$, $\textsf{F}_4$,
or $\textsf{G}_2$).  We will now illustrate this by an explicit
computation for the root system $\Phi$ of type $\textsf{A}_{\ell}$.

Set $n = \ell + 1$,
and consider the usual realization of $\Phi$ as the set of vectors
$$
\varepsilon_i - \varepsilon_j, \ \ i, j = 1, \ldots , n, \ i \neq j,
$$
where $\varepsilon_1, \ldots , \varepsilon_n$ is the standard basis
of $\Q^n$; then $\Gamma := W(\Phi)$ is identified with the symmetric
group $S_n$ acting by permutation of indices. Let
$$N = \bigoplus_{i = 1}^n \Z \varepsilon_i.$$
Then for $M = Q(\Phi)$, we have the exact sequence of
$\Gamma$-modules
$$
0 \to M \longrightarrow N \stackrel{\delta}{\longrightarrow}
\mathbb{Z} \to 0,
$$
with $\delta\left( \sum a_i e_i  \right) = \sum a_i$, which yields the
following exact sequence in cohomology
\begin{equation}\tag{A.5}\label{E:Exact-sequence}
N^{\Gamma} \stackrel{\delta}{\longrightarrow} \mathbb{Z}
\longrightarrow H^1(\Gamma , M) \longrightarrow H^1(\Gamma , N).
\end{equation}
We have an isomorphism of $\Gamma$-modules $N \simeq
\Z[S_n/S_{n-1}]$, so by Shapiro's lemma $$H^1(\Gamma , N) =
H^1(S_{n-1} , \mathbb{Z}) = 0.$$ Then (\ref{E:Exact-sequence})
implies that
$$
H^1(\Gamma , M) = \Z/n\Z,
$$
which is consistent with our previous discussion. We note that this computation is another way to interpret
the computation given in the second part of Example 4.3.

\vskip7mm

\begin{center}

{\large Appendix 2. {\sc On cohomological invariants and good reduction of groups of type $\textsf{F}_4$}}

\end{center}

\vskip3mm

\noindent {\bf A2.1. Cohomological invariants.} Let $k$ be an infinite field of characteristic $\neq 2, 3$, and let $G_0$ be the simple $k$-split group of type $\textsf{F}_4$ (which is both simply connected and adjoint). Then the $k$-isomorphism classes of simple $k$-groups of this type correspond bijectively to the elements of the Galois 1-cohomology set $H^1(k , G_0)$. We recall that J.-P.~Serre constructed two cohomological invariants with coefficients in the group $\mathbb{Z}/2\mathbb{Z}$:
$$
f_3 \colon H^1(k , G_0) \longrightarrow H^3(k , \mathbb{Z}/2\mathbb{Z}) \ \ \text{and} \ \ f_5 \colon H^1(k , G_0) \longrightarrow H^5(k , \mathbb{Z}/2\mathbb{Z})
$$
(see \cite[Theorems 22.4 and 22.5]{GMS}).
Furthermore, M.~Rost \cite{Rost} defined a cohomological invariant with coefficients in $\mathbb{Z}/3\mathbb{Z}$:
$$
g_3 \colon H^1(k , G_0) \longrightarrow H^3(k , \mathbb{Z}/3\mathbb{Z}).
$$
(We note that the maps $f_3, f_5$, and $g_3$ are natural in the base field $k$.) Given a cocycle $\xi \in Z^1(k , G_0)$ with corresponding twisted group $G = {}_{\xi}G_0$, we will often write $f_3(G)$, $f_5(G)$, and $g_3(G)$ instead of $f_3([\xi])$,  $f_5([\xi])$, and $g_3([\xi])$. One assembles these three invariants into a map
$$
\phi \colon H^1(k , G_0) \stackrel{(f_3, f_5, g_3)}{\longrightarrow} H^3(k , \mathbb{Z}/2\mathbb{Z}) \times H^5(k , \mathbb{Z}/2\mathbb{Z}) \times H^3(k , \mathbb{Z}/3\mathbb{Z}),
$$
and one of the remaining fundamental open problems in the theory of Jordan algebras is to determine if $\phi$ is injective.
The following theorem contains a partial result in this direction.
We let  $H^1(k , G_0)_{g_3=0}$ denote the subset of $H^1(k , G_0)$ consisting of cohomology classes/forms having trivial  $g_3$-invariant.

\vskip2mm

\noindent {\bf Theorem A2.1.} {\rm (\cite{Springer})} {\it The map
$$
H^1(k , G_0)_{g_3 = 0} \stackrel{(f_3 , f_5)}{\longrightarrow} H^3(k , \mathbb{Z}/2\mathbb{Z}) \times H^5(k , \mathbb{Z}/2\mathbb{Z})
$$
is injective.}

\vskip2mm

\noindent {\bf Corollary A2.2.} {\it The map $\phi$ has trivial kernel.}

\vskip2mm

In Theorems 1.10-1.12, we deal with forms of type $\textsf{F}_4$ that have trivial $g_3$-invariant. It is well-known (cf. \cite[26.18]{Invol}, \cite{Pettersson}) that to each $k$-group $G$ of type $\textsf{F}_4$, one can associate a 27-dimensional simple exceptional Jordan $k$-algebra $J$ known as the {\it Albert algebra}. Then the $g_3$-invariant of $G$ vanishes if and only if $J$ is {\it reduced}, i.e. has zero divisors. In this case, $J$ admits a natural construction that involves an octonion algebra $\mathcal{O} = \mathcal{O}(a, b, c)$ corresponding to a triple $a, b, c \in k^{\times}$ and two additional parameters $d, e \in k^{\times}$. Then the cohomological invariants of $G$ with coefficients in $\mathbb{Z}/2\mathbb{Z}$ are the following symbols
\begin{equation}\label{E:invar}
f_3(G) = (a) \cup (b) \cup (c) \ \ \text{and} \ \ f_5(G) = f_3(G) \cup (d) \cup (e),
\end{equation}
where $(t) \in H^1(k , \Z/2Z)$ denotes the cohomology class corresponding to $t {k^{\times}}^2$ under the canonical isomorphism $H^1(k , \Z/2\Z) \simeq k^{\times}/{k^{\times}}^2$.
A group $G$ of type $\textsf{F}_4$ with trivial $g_3$-invariant is split over $k$ if and only if the octonion algebra $\mathcal{O}$ is split, i.e. the invariant $f_3(G)$ vanishes, and $G$ is $k$-isotropic if and only if the invariant $f_5(G)$ vanishes.

We will now establish an alternative characterization of groups of type $\textsf{F}_4$ having trivial $g_3$-invariant. Given a (separable) quadratic extension $\ell/k$, we say that a $k$-torus $T$ is $\ell/k$-{\it admissible} if it is anisotropic over $k$ and is split over $\ell$, or, equivalently, if the nontrivial automorphism $\sigma$ of $\ell/k$ acts on the character group $X(T)$ as multiplication by $(-1)$.

\vskip2mm

\noindent {\bf Proposition A2.3.} {\it Let $G$ be a $k$-group of type $\textsf{F}_4$. Then $g_3(G) = 0$ if and only if $G$ becomes split over some quadratic extension $\ell/k$, in which case $G$ possesses a maximal $k$-torus $T$ that is $\ell/k$-admissible.}

\vskip2mm

We say that a finite separable extension $\ell/k$ splits $x \in H^i(k , \mathbb{Z}/2\mathbb{Z})$ if the image of $x$ under the restriction map $H^i(k , \mathbb{Z}/2\mathbb{Z}) \to H^i(\ell , \mathbb{Z}/2\mathbb{Z})$ is trivial. We have the following statement.

\vskip2mm

\noindent {\bf Lemma A2.4.} {\it Let $G$ be a $k$-group of type $\textsf{F}_4$ for which $g_3(G) = 0$. If $\ell/k$ is a quadratic extension that splits $f_3(G)$, then it also splits $f_5(G)$, and therefore splits $G$.}

\vskip2mm

Indeed, the first assertion immediately follows from the fact that $f_5(G) = f_3(G) \cup (d) \cup (e)$. Thus, over $\ell$, all 3 invariants $f_3(G), f_5(G)$, and $g_3(G)$ become trivial. Then it follows from Corollary A2.2 (over $\ell$) that the cohomology class in  $H^1(k , G_0)$ that corresponds to $G$ becomes trivial in $H^1(\ell , G_0)$, and therefore $G$ is $\ell$-isomorphic to $G_0$, hence is $\ell$-split.

\vskip1mm

\noindent {\it Proof of Proposition A2.3.} Suppose $g_3(G) = 0$. Since the cohomology class $f_3(G)$ is a symbol, it splits over some quadratic extension $\ell/k$. By Lemma A2.4, $G$ splits over $\ell$. The group $G$ has an $\ell$-defined Borel subgroup $B$ such that $T := B \cap B^{\sigma}$ is a maximal $k$-torus of $G$ (cf. \cite[Lemma 6.17]{PR}). If $\Pi$ is the system of positive roots in the root system $\Phi = \Phi(G , T)$ that corresponds to $B$, then for the action of $\sigma$ on $X(T)$, we have $\sigma(\Pi) = -\Pi$. But the only element in $\mathrm{Aut}(\Phi)$ that has this property is multiplication by $(-1)$. Thus, $\sigma = -1$, i.e. $T$ is $\ell/k$-admissible.

Conversely, if $G$ becomes split over a quadratic extension $\ell/k$, then $\ell$ splits the invariant $g_3(G)$. Using a standard restriction-corestriction argument, we see that $2 \cdot g_3(G) = 0$. But every element in $H^3(k , \mathbb{Z}/3\mathbb{Z})$ satisfies $3 \cdot g_3(G) = 0$. Thus, $g_3(G) = 0$. \hfill $\Box$

\vskip3mm

\noindent {\bf A2.2. An alternative description of invariants and good reduction.} First, we recall some basic facts about absolutely almost simple $k$-groups $G$ that possess a maximal $k$-torus $T$ that is admissible with respect to a quadratic extension $\ell/k$, assuming that $\mathrm{char}\: k \neq 2$. These results were initially obtained in \cite{Weis-quadr} and then systematically redeveloped in \cite{Chern1}, \cite{Chern2} in the more general situation of groups over regular local rings (in particular, over discrete valuation rings). This generalization becomes particularly useful when we consider forms with good reduction. We will now review the theory over fields. Let  $\mathfrak{g} = L(G)$ denote the Lie algebra of $G$, and let $\Phi = \Phi(G , T)$ be the root system of $G$ with respect to the maximal torus $T$. Fix a Chevalley basis
$$
\{ H_{\alpha_1}, \ldots , H_{\alpha_r} \} \cup \{ X_{\alpha} \}_{\alpha \in \Phi}
$$
of $\mathfrak{g}(\ell)$ associated with $T$ (where $r = \dim T$ is the rank of $G$ and $\Pi = \{ \alpha_1, \ldots , \alpha_r \}$ is a system of simple roots). The key observation is that the action of the nontrivial automorphism $\sigma$ of $\ell/k$ on the root elements  of the Chevalley basis is described by  equations of the form
$$
\sigma(X_{\alpha}) = c_{\alpha} X_{-\alpha} \ \ \text{with} \ \ c_{\alpha} \in k^{\times}.
$$
These constants $c_{\alpha}$ completely determine the $k$-isomorphism class of $G$ (assuming that the latter is simply connected). One checks (cf. \cite{Chern2}) that $c_{-\alpha} = c_{\alpha}^{-1}$ and $c_{\alpha+\beta} = \pm c_{\alpha}c_{\beta}$ (with the sign depending only on $\alpha$ and $\beta$ as elements of $\Phi$). This means that the $k$-isomorphism class of $G$ is determined by the quadratic extension $\ell/k$ and the constants $c_{\alpha}$ for only simple roots $\alpha$.

In the rest of this subsection, $G$  will denote a $k$-group of type $\textsf{F}_4$ with trivial $g_3$-invariant. According to Proposition A2.3, the group $G$ has a maximal $k$-torus that is admissible over a quadratic extension $\ell = k(\sqrt{a})$. We  will use the labeling of the simple roots  $\alpha_1, \ldots , \alpha_4$ introduced in Bourbaki \cite{Bour}.

\vskip2mm

\noindent {\bf Theorem A2.5.}
{\rm (\cite[Theorems 6.1 and 6.6]{Chern2})}
{\it Let $G$ be a simple algebraic $k$-group of type $\textsf{F}_4$ that has a maximal $k$-torus $T$ that is admissible over a quadratic
extension $\ell = k(\sqrt{a})$. Fix a system of simple roots $\Pi = \{ \alpha_1, \ldots , \alpha_4 \}$ in the root system $\Phi(G , T)$, and let
$c_{\alpha_1}, \ldots , c_{\alpha_4}$ be the constants defined above for some choice of root vectors in a Chevalley
basis. Then
$$
f_3(G) = (a) \cup (c_{\alpha_1}) \cup (c_{\alpha_2}) \ \ \text{and} \ \ f_5(G) = (a) \cup (c_{\alpha_1}) \cup (c_{\alpha_2}) \cup (c_{\alpha_3}) \cup (c_{\alpha_4}).
$$}

\vskip2mm

We will now use this description of the invariants $f_3$ and $f_5$ to prove that they are unramified if the group $G$ has good reduction
at a discrete valuation $v$ of $k$ (see subsection A2.4 below for a more general result).

\vskip2mm

\noindent {\bf Theorem A2.6.} {\it Let $G$ be a $k$-group of type $\textsf{F}_4$ with trivial $g_3$-invariant that has good reduction at a discrete valuation $v$ of $k$ with $\mathrm{char}\: k^{(v)} \neq 2$. Then the invariants $f_3(G)$ and $f_5(G)$ are unramified at $v$.}

\vskip2mm


The proof will involve an application of the
results from \cite{Chern2} over discrete valuation rings.
So, suppose that our base field $k$ is equipped with a discrete valuation $v$, and let $\mathscr{O}_{k,v} \subset k$ be the corresponding valuation ring. Let $G$ be a simple algebraic $k$-group of type $\textsf{F}_4$ that splits over a quadratic extension $\ell =  k(\sqrt{a})$. It follows from the description of $f_3(G)$ as a symbol  (see (\ref{E:invar})) that,  without loss of generality, we can always assume that one of the elements is a unit, and we take this element for $a$. Then $\ell = k(\sqrt{a})$ is unramified over $k$ at $v$, and therefore $\tilde{\mathscr{O}} := \mathscr{O}_{k,v}(\sqrt{a})$ is an \'etale extension of $\mathscr{O}_{k,v}$.

\vskip2mm

\noindent {\it Proof of Theorem A2.6.} Suppose that $G$ has good reduction at $v$, i.e. there exists a reductive group scheme $\mathscr{G}$ over $\mathscr{O}_{k,v}$ with  generic fiber $G$ (see the discussion in \S2.3). Then $\mathscr{G} \times_{\mathscr{O}_{k,v}} \tilde{\mathscr{O}}$ is a reductive group scheme with generic fiber $G \times_k \ell$. Applying Proposition \ref{P:UnMod} in this situation and taking into account that $G \times_k \ell$ is split, we see that $\mathscr{G} \times_{\mathscr{O}_{k,v}} \tilde{\mathscr{O}}$ is split. Then one shows that $\mathscr{G}$ contains a maximal torus $\mathscr{T}$ whose generic fiber $T$ is an $\ell/k$-admissible torus. Furthermore, one verifies that the constants $c_{\alpha_i}$ $(i = 1, \ldots , 4)$ belong to $\mathscr{O}_{k,v}^{\times}$ (see \cite[Remark 44]{Chern2}). Combining this with the formulas for $f_3(G)$ and $f_5(G)$ given in Theorem A2.5
completes
the proof of Theorem A2.6. \hfill $\Box$

\vskip1mm

We conclude this subsection with the following, which in some sense provides a converse to Theorem~A2.6.

\vskip2mm

\noindent {\bf Proposition A2.7.} {\it Let $G$ be a simple algebraic $k$-group of type $\textsf{F}_4$ that has a maximal $k$-torus $T$ that is $\ell/k$-admissible for a quadratic extension $\ell/k$. Assume that $\ell/k$ is unramified at $v$ and the constants $c_{\alpha_1}, \ldots , c_{\alpha_4}$ are $v$-units. Then $G$ has good reduction at $v$.}

\vskip2mm

\begin{proof}
Let $\mathscr{G}_0$ (resp., $\mathfrak{g}_0$) be the split group scheme (resp., split Lie algebra) of type $\textsf{F}_4$ over $\mathscr{O}_{k,v}$, and let $\tilde{\mathscr{O}}$ be the integral closure of $\mathscr{O}_{k,v}$ in $\ell$. It is enough to construct a Lie algebra $\mathfrak{g}$ over $\mathscr{O}_{k,v}$ such that
$$
\mathfrak{g} \otimes_{\mathscr{O}_{k,v}} \tilde{\mathscr{O}} \simeq \mathfrak{g}_0 \otimes_{\mathscr{O}_{k,v}} \tilde{\mathscr{O}} \ \ \text{and} \ \
\mathfrak{g} \otimes_{\mathscr{O}_{k,v}} k \simeq L(G)_k,
$$
where $L(G)$ is the Lie algebra of $G$. Indeed, it is well-known that the automorphism group of a split Lie algebra of type $\textsf{F}_4$ is a simple split algebraic group of type $\textsf{F}_4$, which is both adjoint and simply connected. So, $\mathfrak{g}$ can be obtained from $\mathfrak{g}_0$ by twisting using an $\ell/k$-cocycle with values in $\mathscr{G}_0(\tilde{\mathscr{O}})$. Then twisting $\mathscr{G}_0$ by the same cocycle, we obtain the required reductive group $\mathscr{O}_{k,v}$-scheme $\mathscr{G}$ with generic fiber $G$ (since its Lie algebra coincides with that of $G$ by construction). On the other hand, the Lie algebra $\mathfrak{g}$ with the above properties is constructed from $\mathfrak{g}_0 \otimes_{\mathscr{O}_{k,v}} \tilde{\mathscr{O}}$ by Galois descent (which can be implemented due to the fact that $\ell/k$ is unramified at $v$) using the automorphism defined by
$$
H_{\alpha} \mapsto - H_{\alpha} \ \ \text{and} \ \ X_{\alpha} \mapsto c_{\alpha} X_{-\alpha}
$$
for all simple roots $\alpha$ (recall that $c_{\alpha} \in \mathscr{O}_{k,v}^{\times}$ by assumption).
\end{proof}

We note that the reduction $\underline{G}^{(v)}$ possesses a maximal $\ell^{(v)}/k^{(v)}$-admissible torus for which the corresponding constants are the reductions $\bar{c}_{\alpha_1}, \ldots , \bar{c}_{\alpha_4}$.

\vskip1mm

\noindent {\bf A2.3. Two results about forms of type $\textsf{F}_4$ that split over a cubic extension.} We will keep the notations introduced in \S \ref{S:F4}
prior to the proof of Proposition \ref{P:degree3}. First, we will prove Proposition \ref{P:F4-1}.

\vskip1mm

\noindent {\it Proof of Proposition \ref{P:F4-1}.} (1): We begin with the following general fact.

\vskip2mm

\noindent {\bf Lemma A2.8.} {\it Let $G$ be a $k$-group of type $\textsf{F}_4$, and let $\ell/k$ be a cubic Galois extension. Given two maximal $k$-tori $T_1$ and $T_2$
of $G$ that are anisotropic over $k$ and split over $\ell$, there exists $g \in G(\ell)$ such that the restriction
of the inner automorphism $\mathrm{Int}\: g$ induces a $k$-defined isomorphism $T_1 \to T_2$.}

\vskip2mm

\begin{proof}
Let us return to the notations introduced immediately before the statement of Proposition \ref{P:F4-1}. The Weyl group $W(R_i , \iota_i(S))$ is isomorphic
to the symmetric group $\Sigma_3$, and we let $V_i \subset W(R_i , \iota_i(S))$ be its Sylow 3-subgroup. Then $V = V_1V_2 \subset W(G_0 , T)$
has order 9, and therefore is a Sylow 3-subgroup of $W(G_0 , T)$. It follows from this description that $W(G_0 , T)$  has a unique
conjugacy class of elements $w$ of order 3 such
that $X(T)^w = \{ 0 \}$. Let $\theta^{(i)} \colon \mathrm{Gal}(\ell/k) \to W(G , T_i)$ for $i = 1, 2$ be the natural homomorphism (cf. subsection \ref{S:generic}.1), and fix
a generator $\sigma \in \mathrm{Gal}(\ell/k)$. Pick an arbitrary $g \in G(\ell)$ such that for the inner automorphism $\iota_g = \mathrm{Int}\: g$, we have
$\iota_g(T_1) = T_2$, and let $\iota_g^* \colon X(T_2) \to X(T_1)$ be the corresponding comorphism. Considering $W(G , T_i)$ as a subgroup of $\mathrm{GL}(X(T_i))$, we can define an isomorphism $W(G , T_1) \to W(G , T_2)$ by $w \mapsto (\iota_g^*)^{-1} \circ w \circ \iota_g^*$. Since $w_i = \theta^{(i)}(\sigma) \in W(G , T_i)$ is an element of order 3 such that $X(T_i)^{w_i} = \{ 0 \}$, it follows from the above remark that by replacing $g$ with $gn$ for an appropriate $n \in N_G(T_1)$, we can assume that
$$
\theta^{(1)}(\sigma) \circ \iota_g^* = \iota_g^* \circ \theta^{(2)}(\sigma).
$$
This means that the restriction $\iota_g \vert T_1$ is defined over $k$, as required.
\end{proof}

\vskip2mm

\noindent {\bf Corollary A2.9.} {\it
With notations and conventions as in Lemma A2.8, the maps $H^1(\ell/k , T_i) \to H^1(\ell/k , G)$ for $i = 1, 2$ have the same image.}

\vskip2mm

\begin{proof}
By the lemma, we can find $g \in G(\ell)$ such that the restriction of $\iota := \mathrm{Int}\: g$ induces a $k$-defined isomorphism $T_1 \to T_2$.
Then for any $\sigma \in \mathrm{Gal}(\ell/k)$, we have $g \cdot \sigma(g)^{-1} \in T_2(\ell)$.  It follows that an arbitrary cocycle $\xi(\sigma)$ on $\mathrm{Gal}(\ell/k)$ with values in $T_1(\ell)$ is equivalent in $H^1(\ell/k , G)$ to the cocycle
$$
g \xi(\sigma) \sigma(g)^{-1} = (g \xi(\sigma) g^{-1}) \cdot (g \cdot \sigma(g)^{-1})
$$
which has values in $T_2(\ell)$, and vice versa.
\end{proof}

It is now easy to conclude the proof of part (1) of Proposition \ref{P:F4-1}. Let $\xi \in Z^1(k , G_0)$ be a cocycle such that $G = {}_{\xi}G_0$ contains a maximal $k$-torus $T_1$ that is anisotropic over $k$ and splits over $\ell$. It follows from  Steinberg's theorem (cf. \cite[Prop. 6.19]{PR}) that there exists a $k$-embedding $T_1 \hookrightarrow G_0$ such that $[\xi]$ lies in the image of the corresponding map $H^1(k , T_1) \to
H^1(k , G_0)$. But according to Corollary A2.9, the image of this map coincides with the image of the map $H^1(k , T) \to H^1(k , G_0)$, and the required fact follows.

\vskip1mm

Turning now to part (2) of Proposition \ref{P:F4-1}, we recall that for $S = \mathrm{R}^{(1)}_{\ell/k}(\mathbb{G}_m)$, we have $H^1(k , S) = k^{\times}/N_{\ell/k}(\ell^{\times})$; in particular, we can write $\gamma_2([\zeta]) = b N_{\ell/k}(\ell^{\times})$ for some $b \in k^{\times}$. On the other hand, $\delta(\gamma_1([\zeta]))$ corresponds to a Brauer class $[A] \in H^2(k , \mu_3) =
{}_3\mathrm{Br}(k)$. The algebra $A$ splits over $\ell$, so that the corresponding division algebra has degree dividing $3$. According to \cite[7.4]{GMS}, we have $g_3(G) = [A] \cup (b)$, as required. (We note that since $\ell$ splits $A$, the cup-product does not depend on the choice of $b$ in the coset modulo the norm subgroup $N_{\ell/k}(\ell^{\times})$.)

\vskip3mm

The second result of this subsection is the following.

\vskip2mm

\noindent {\bf Lemma A2.10.} {\it {\rm (1)} $\delta$ is injective.

\noindent {\rm (2)} \parbox[t]{16cm}{Assume that $k$ is complete with respect to a discrete valuation $v$ with  $\mathrm{char}\: k^{(v)} \neq 3$, and let $\mathcal{O}_v$ be the valuation ring in $k = k_v$. Furthermore, assume that the extension $\ell/k$ is unramified, so that there is an $\mathcal{O}_v$-torus $\mathscr{S}_1$ with generic fiber $S_1$. If $x \in H^1(k_v , S_1)$ is such that the image $\delta(x) \in H^2(k , \mu_3)$ is unramified, then $x$ belongs to the image of the map $H^1(\mathcal{O}_v , \mathscr{S}_1) \to H^1(k_v , S_1)$.}}

\vskip2mm

\begin{proof}
(1): We have the following long exact sequence
$$
H^1(k , \mu_3) \stackrel{\alpha}{\longrightarrow} H^1(k , \tilde{S}_1) \longrightarrow H^1(k , S_1) \stackrel{\delta}{\longrightarrow} H^2(k , \mu_3).
$$
Since $\alpha$ is surjective, $\delta$ is injective.

\vskip1mm

(2): Let $\tilde{\mathscr{S}}_1$ be an $\mathcal{O}_v$-torus with generic fiber $\tilde{S}_1$. We first show that the map $$H^2(\mathcal{O}_v , \tilde{\mathscr{S}}_1) \stackrel{\varepsilon}{\longrightarrow} H^2(k_v , \tilde{S}_1)$$ is injective. Consider the $k_v$-tori $T_0 = \mathbb{G}_m$ and $T = R_{\ell/k}(\mathbb{G}_m)$, and let $\mathscr{T}_0$ and $\mathscr{T}$ be the $\mathcal{O}_v$-tori with generic fibers $T_0$ and $T$. The exact sequence
$$
1 \to \tilde{S}_1 \longrightarrow T \stackrel{N}{\longrightarrow} T_0 \to 1,
$$
where $N$ is the norm map associated with the extension $\ell/k$, induces the following commutative diagram with exact rows
$$
\begin{CD}
H^1(\mathcal{O}_v , \mathscr{T}_0) @ >>>  H^2(\mathcal{O}_v , \tilde{\mathscr{S}}_1) @ >>> H^2(\mathcal{O}_v , \mathscr{T}) \\
@ V {\rho_1}VV @ V {\varepsilon} VV @ V {\rho_2} VV \\
H^1(k_v , T_0) @ >>> H^2(k_v , \tilde{S}_1) @ >>> H^2(k_v , T)
\end{CD}.
$$
But $H^1(k_v , T_0) = \{ 1 \}$ by Hilbert's Theorem 90, and $H^1(\mathcal{O}_v , \mathscr{T}_0) = \mathrm{Pic}\: \mathcal{O}_v = \{ 1 \}$. On the other hand, by Shapiro's Lemma, the homomorphism $\rho_2$ can be identified with the homomorphism $$H^2(\mathcal{O}(\ell) , \mathscr{T}_0 \times_{\mathcal{O}_v} \mathcal{O}(\ell)) \to H^2(\ell , \mathscr{T}_0 \times_{k} \ell),$$ where $\mathcal{O}(\ell)$ is the valuation ring of $\ell$. So, the injectivity of $\rho_2$ immediately follows from the injectivity of the canonical map of the Brauer group of a discrete valuation ring to the Brauer group of its field of fractions (cf. \cite[3.6]{CT-unram}, \cite[Ch. IV, \S 2]{Milne-EC}). Now, the injectivity of $\varepsilon$ follows from the above commutative diagram.

Next, we have the following commutative diagram with exact rows
$$
\begin{CD}
H^1(\mathcal{O}_v , \mathscr{S}_1) @ >>> H^2(\mathcal{O}_v , \mu_3) @ >>> H^2(\mathcal{O}_v , \tilde{\mathscr{S}}_1) \\
@ VVV @ V{\omega} VV @ V{\varepsilon} VV \\
H^1(k_v , S_1) @ >{\delta}>> H^2(k_v , \mu_3) @ >>> H^2(k_v , \tilde{S}_1)
\end{CD}.
$$
It is well-known that $\mathrm{Im}\: \omega$ coincides with the subgroup of unramified cohomology classes. Then the required assertion follows from the injectivity of $\delta$ and $\varepsilon$ by a diagram chase.
\end{proof}

\vskip1mm

\noindent {\bf A2.4. Cohomological invariant $g_3$ of forms with good reduction.} The goal of this section is to prove the following.

\vskip2mm

\noindent {\bf Proposition A2.11.} {\it Let $k$ be a field with a discrete valuation $v$ such that $\mathrm{char}\: k^{(v)} \neq 3$. If $G$ is a $k$-form of type $\textsf{F}_4$ that has good reduction at $v$, then the invariant $g_3(G)$ is unramified at $v$.}

\vskip2mm

\begin{proof}
Without loss of generality, we may suppose that $k$ is complete with respect to $v$, and  let $\mathcal{O}_v$ be the valuation ring of $k$. By assumption, there exists  a reductive group $\mathcal{O}_v$-scheme $\mathscr{G}$ with generic fiber $G$. Then the reduction $\underline{G}^{(v)} = \mathscr{G} \times_{\mathcal{O}_v} k^{(v)}$ is the automorphism group of a simple exceptional Jordan $k^{(v)}$-algebra $J^{(v)}$. It follows from \cite[Theorem 58]{Pettersson} that there exists a quadratic extension $\bar{\ell}/k^{(v)}$ such that the algebra $J^{(v)} \otimes_{k^{(v)}} \bar{\ell}$ is isomorphic to the Albert algebra $(A^{(v)} , \bar{\mu})$ obtained by Tits' first construction from a central cubic $k^{(v)}$-algebra $A^{(v)}$ and some $\bar{\mu} \in \bar{\ell}^{\times}$. Let $\ell$ be the unramified extension of $k$ with residue field $\ell^{(v)} = \bar{\ell}$, and let $\mathcal{O}(\ell)$ be the valuation ring of $\ell$. We now consider an Azumaya $\mathcal{O}(\ell)$-algebra $\mathscr{A}$ with residue algebra $A^{(v)}$, and let $\mu \in \mathcal{O}(\ell)^{\times}$ be an element with residue $\bar{\mu}$. Applying Tits' first construction with these $\mathscr{A}$ and $\mu$, we obtain a Jordan $\mathcal{O}(\ell)$-algebra $\mathscr{J}$.
It follows from Hensel's Lemma that the reductive group $\mathcal{O}(\ell)$-scheme $\tilde{\mathscr{G}}$ corresponding to $\mathscr{J}$ is isomorphic to $\mathscr{G} \times_{\mathcal{O}_v} \mathcal{O}(\ell)$. In particular, the generic fiber $\tilde{G}$ of $\tilde{\mathscr{G}}$ is isomorphic to $G \times_k \ell$ and corresponds to the Albert algebra $(A , \mu)$, where $A = \mathscr{A} \otimes_{\mathscr{O}(\ell)} \ell$. Let $\mathrm{res}_{\ell/k}$ denote the restriction map in cohomology. Then it follows from the definition of $g_3$ (see \cite{Rost}) that
$$
\mathrm{res}_{\ell/k}(g_3(G)) = g_3(G \times_k \ell) = [A] \cup (\mu) \in H^3(\ell , \Z/3\Z),
$$
where $[A]$ denotes the class of $A$ in $_{3}\mathrm{Br}(\ell) = H^2(\ell , \mu_3)$. Since by construction $A$ comes from an Azumaya algebra and $\mu \in \mathcal{O}(\ell)^{\times}$, we conclude that $\mathrm{res}_{\ell/k}(g_3(G)) \in H^3(\ell , \Z/3\Z)$ is unramified.
On the other hand, since $\ell/k$ is unramified, we have the following commutative diagram
$$
\begin{CD}
H^3(k , \Z/3\Z) @ > \rho_k >> H^2(k^{(v)} , \mu_3) \\
@ V \mathrm{res}_{\ell/k} VV @ VV \mathrm{res}_{\bar{\ell}/k^{(v)}} V\\
H^3(\ell , \Z/3\Z) @ > \rho_{\ell} >> H^2(\bar{\ell} , \mu_3)
\end{CD}
$$
where $\rho_k$ and $\rho_{\ell}$ are the corresponding residue maps. We have
$$
\rho_{\ell}(\mathrm{res}_{\ell/k}(g_3(G))) = 0 = \mathrm{res}_{\bar{\ell}/k^{(v)}}(\rho_k(g_3(G))).
$$
But since $[\bar{\ell} : k^{(v)}] = 2$, a standard restriction-corestriction argument shows that $2 \cdot \rho_k(g_3(G)) = 0$. On the other hand, $3 \cdot H^2(k^{(v)} , \mu_3) = 0$, so $\rho_k(g_3(G)) = 0$, as required.
\end{proof}


\noindent {\small {\bf Acknowledgements.} We are grateful to Eva Bayer-Fluckiger and Skip Garibaldi for their comments, and to Philippe Gille for making the preprint \cite{Gille-Gos} available to us. Thanks are also due to the anonymous referee whose suggestions helped to improve the exposition. The first-named author was partially supported
by an NSERC research grant. The third-named author was partially supported by NSF grant DMS-2154408.}


\bibliographystyle{amsplain}

\end{document}